\documentclass[12pt]{amsart}

\usepackage[dvipsnames]{xcolor}
\usepackage{graphicx,enumitem, booktabs, array, verbatim, longtable} 
\usepackage{tikz,tikz-cd}
\usepackage{csvsimple}
\usepackage{svg}

\usepackage[normalem]{ulem}

\usepackage{mathrsfs}
\usepackage{nameref}
\usepackage{kbordermatrix}
\usepackage{varioref}
\usepackage[colorlinks=true,
linkcolor=teal,
citecolor=Periwinkle 
]{hyperref}
\usepackage{cleveref}
\usepackage{todonotes}

\usepackage{fullpage}

\usepackage[style=alphabetic,doi=false,isbn=false,url=false,eprint=true,
backend=biber,giveninits=true,maxalphanames=5,maxcitenames=3,maxbibnames=8,hyperref]{biblatex}
\DeclareMathOperator{\conv}{conv}
\newcommand{\groupnamea}{\scshape{id}}

\usepackage{subcaption}
\newcommand{\A}{\mathcal{A}}
\newcommand{\B}{\mathcal{B}}

\newcommand{\W}{\mathcal{W}}
\newcommand{\FF}{\mathscr{F}}
\newcommand{\KK}{\mathscr{K}}
\newcommand{\cP}{\mathcal{P}}
\newcommand{\cR}{\mathcal{R}}
\newcommand{\R}{\mathbb{R}}
\newcommand{\KRW}[1]{\operatorname{KRW}(#1)}

\newcommand{\GPF}[1]{\FF_{#1}} 
\newcommand{\GPFclosed}[1]{\overline{\FF}_{#1}} 
\newcommand{\Tn}[1]{\mathscr{T}_{#1}}
\newcommand{\metric}{\rho}
\newcommand{\CA}[1]{\W_{#1}}
\newcommand{\mcone}[1]{\mathbf M_{#1}}
\newcommand{\mconeclosure}[1]{\overline{\mathbf M}_{#1}}
\newcommand{\mconeinterior}[1]{\mathring{\mathbf M}_{#1}}
\newcommand{\kfan}[1]{\KK_{#1}} 
\newcommand{\kfanclosed}[1]{\overline{\KK}_{#1}} 
\newcommand{\splitpm}[2]{\delta_{#1,#2}}
\newcommand{\elsplitsclosed}[1]{\overline{\mathscr E}_{#1}}
\newcommand{\elsplits}[1]{\mathscr E_{#1}}

\newcommand{\tk}[2]{t^{(#1,#2)}}

\newcommand{\stratsharp}[1]{\mathscr S_{#1}}
\newcommand{\pstratsharp}[1]{\overline{\mathscr S}_{#1}}
\newcommand{\MF}[1]{\operatorname{MF}(#1)}
\newcommand{\mfan}[1]{\overline{\mathscr M}_n}

\newcommand{\coh}[1]{H_{#1}}

\newcommand{\hk}[2]{h^{(#1)}_{#2}}
\newcommand{\hhk}[2]{\hat{h}^{(#1)}_{#2}}
\newcommand{\vk}[2]{v^{(#1)}_{#2}}

\numberwithin{equation}{section}

\theoremstyle{plain}
\newtheorem{theorem}{Theorem}[section]

\newtheorem{lemma}[theorem]{Lemma}

\newtheorem{corollary}[theorem]{Corollary}
\newtheorem{proposition}[theorem]{Proposition}
\newtheorem*{problem}{Problem}
\theoremstyle{definition}
\newtheorem{definition}[theorem]{Definition}
\newtheorem{remark}[theorem]{Remark}

\newtheorem{example}[theorem]{Example}
\newtheorem{question}[theorem]{Question}

\crefname{proposition}{proposition}{propositions}
\crefname{lemma}{lemma}{lemmas}
\crefname{corollary}{corollary}{corollarys}
\crefname{remark}{remark}{remarks}

\makeatletter
\newcommand{\eqnum}{\refstepcounter{equation}\textup{\tagform@{\theequation}}}
\makeatother

\title[The Wasserstein Arrangement]{Combinatorial invariants of finite metric spaces and the Wasserstein arrangement
}
\author{Emanuele Delucchi}
\address{Emanuele Delucchi, IDSIA USI-SUPSI, University of applied arts and sciences, Lugano, Switzerland}
\email{emanuele.delucchi@supsi.ch}

\author{Lukas~K\"uhne}
\address{Lukas K\"uhne, Universit\"at Bielefeld, Fakult\"at f\"ur Mathematik, Bielefeld, Germany}
\email{lkuehne@math.uni-bielefeld.de}

\author{Leonie Mühlherr}
\address{Leonie Mühlherr, Universit\"at Bielefeld, Fakult\"at f\"ur Mathematik, Bielefeld, Germany}
\email{lmuehlherr@math.uni-bielefeld.de}

\date{\today}

\begin{document}


\begin{abstract}
	In 2010, Vershik proposed a new combinatorial invariant of metric spaces given by a class of polytopes that arise in the theory of optimal transport and are called ``Wasserstein polytopes'' or ``Kantorovich-Rubinstein polytopes'' in the literature. 
	Answering a question posed by Vershik, we describe the stratification of the metric cone induced by the combinatorial type of these polytopes via a  hyperplane arrangement. Moreover, we study its relationship with the stratification by combinatorial type of injective hulls (i.e., tight spans) and, in particular, with several types of metrics arising in phylogenetic analysis. We also compute enumerative invariants in the case of metrics on up to six points.
\end{abstract}

\maketitle

\hfill{{\em In memory of Andreas W.\ M.\ Dress.}}

\section{Introduction}

We study {finite metric spaces} through the lens of polytopes and hyperplane arrangements. 

Finite metric spaces arise in several applied contexts. For instance, in   mathematical biology finite metrics  model genetic dissimilarities between different species \cite{Phylogenetics}. A main research direction in this setting is to identify suitable classes of metric spaces and study the combinatorics and geometry of the associated subset of the metric cone, e.g., for geometric statistics. This field of research goes back to the study of phylogenetic trees \cite{Bandelt-Dress-Trees,BHV}, has grown to include more general phylogenetic networks \cite{HRS} and is presently very active (see e.g., \cite{Devadoss-petti,Francis_Steel,Fischer_Francis}).
 The \emph{injective hull}, introduced by Isbell~\cite{Isbell} and rediscovered by Dress~\cite{Dress} under the name \emph{tight span}, is a classical combinatorial invariant of metric spaces that is widely used in theoretical as well as in applied settings.
The stratification of the metric cone by combinatorial type of tight spans is equivalent to the one determined by the subdivisions of the second hypersimplex studied by Sturmfels and Yu~\cite{SturmfelsYu} and recently by Casabella, Joswig, and Kastner~\cite{CJK24}.

Motivated by the theory of optimal transport, Vershik described a   correspondence between finite metric spaces and a class of symmetric convex polytopes, the so-called \emph{Kantorovich-Rubinstein-Wasserstein (KRW) polytopes} or \emph{fundamental polytopes}~\cite{Vershik}. 
\begin{definition}\label{def:KRW}
	The KRW polytope of an $n$-metric $\metric$ is a polytope in $\R^n$ defined as the convex hull
	\[
	\KRW{\metric} = \conv\left\{\left. \frac{e_i-e_j}{\metric_{ij}} \right\vert 1\le i , j \le n\right\},
	\]
	where $e_i$ is the $i$-th standard basis vector of $\R^n$. From now on, we abbreviate \begin{equation}\label{equ:pij}
    p_{ij}:=\frac{e_i-e_j}{\metric_{ij}} \textrm{ for } i\neq j.
    \end{equation}
	The metric is called \emph{generic} if it is strict and the KRW polytope is simplicial.
\end{definition}

	The class of KRW polytopes subsumes several well-studied families of polytopes. For instance, if $\metric$ is the $n$-metric with $\metric_{ij}=1$ for all $i\neq j$, then $\KRW{\metric}$ is known as the type $A_n$ \emph{root polytope}.

	If $\metric$ is a graph-metric (i.e., there is a graph $G$ with vertex set $[n]$ such that $\metric_{ij}$ is the number of edges of a shortest path from $i$ to $j$ in $G$), then $\KRW{\metric}$ is the {\em symmetric edge polytope} of~$G$~\cite{dali-delucchi-michalek}, also known as the {\em adjacency polytope} of $G$ \cite{Chen,CDM}. If $\rho$ is a tree metric (i.e., it can be represented as the distances among a subset of the vertices of a metric tree), the number of faces of $\KRW{\metric}$ can be computed explicitly, and the dual polytope is a zonotope \cite{DelucchiHoessly}.

The dual of a KRW polytope is a so-called \emph{Lipschitz polytope} \cite{Vershik,SS18}.
These belong to the class of (symmetric) \emph{alcoved polytopes}~\cite{LP1}, which  recently gained significant attention due to their connections to theoretical physics in the form of positive geometries and polypositroids~\cite{LP2}.

\bigskip

The problem of understanding the combinatorial structure (e.g., computing face numbers) of KRW and Lipschitz polytopes is open and significant for applications, see e.g., \cite{Becedasetal,venturelloetal}. In this article, we focus on a related problem that goes back to a question by Vershik. 
\begin{problem}[{\cite[``General Problem", \S 1]{Vershik}}]
	Study and classify finite metric spaces according to the combinatorial properties of their KRW polytopes.
\end{problem}

The first progress on this question was achieved by Gordon and Petrov who gave a description of the face poset of KRW polytopes via linear inequalities on the values of the metric~\cite{GordonPetrov}, see~\Cref{sec:GP} for further details.
This is our starting point for an in-depth study of KRW polytopes using the theory of hyperplane arrangements.

\subsection{Summary of the main results}

\begin{enumerate}
\item The subdivision of the metric cone by combinatorial types of KRW polytopes is a (half-open) polyhedral fan: we call this the {\em type fan of KRW polytopes}.
 We introduce a \emph{Wasserstein arrangement} which refines the type fan of KRW polytopes. This arrangement builds upon the  description of the faces of KRW polytopes given by Gordon and Petrov \cite{GordonPetrov} and may be of independent interest.
Moreover, via the above connection to alcoved polytopes, intersecting this arrangement with the metric cone naturally yields a refinement of the type fan of symmetric alcoved polytopes.

\item Using the computer algebra system \texttt{OSCAR}~\cite{OSCAR-book} together with the Julia package \texttt{CountingChambers.jl}~\cite{BEK23} and the software \texttt{TOPCOM}~\cite{Rambau} we obtain an enumeration of the combinatorial types of KRW polytopes of generic metrics on $n=4,5,6$ points. This is displayed in \Cref{sim_table}, where the number of combinatorial types of generic KRW polytopes is shown in the column titled ``Labeled''. The column ``Unlabeled'' shows the number of combinatorial types up to combinatorial isomorphisms (where we are allowed to permute the elements in our metric space which induces an action on the KRW polytopes).
The last column of this table shows the number of chambers of the Wasserstein arrangement which agrees with the number of combinatorial type of generic KRW polytopes up to $n\le 5$. This stems from the fact that for $n\ge 6$ the Wasserstein arrangement is a strict refinement of the type fan of KRW polytopes, see~\Cref{prop:refinement}.

This table shows that there are many more combinatorial types of KRW polytopes than there are of tight spans (\Cref{defIH}), as
for $n=6$ there are 194,160 different tight spans which up to symmetry come in $339$ orbits~\cite{SturmfelsYu}.
	
	\begin{table}[th!] 
		\centering 
		
		\begin{tabular}{cccc}
			\toprule
			$n$  & Unlabeled & Labeled  & \# Chambers\\ 
			\midrule 
			3  & 1 & 1  & 1\\ 
			4 & 1 & 6  & 6\\ 
			5 & 12 & 882 & 882\\ 
			6 & 25,224 & 17,695,320 & 6,677,863,200 \\ 
			\bottomrule  
		\end{tabular} 
	\caption{Combinatorial types of KRW polytopes of generic $n$-metrics and chambers of the Wasserstein arrangement.} 
		\label{sim_table} 
	\end{table}

\item We clarify the relationship between KRW polytopes and tight spans by providing examples of five-point metrics with isomorphic tight spans and combinatorially different KRW polytopes and vice versa (\Cref{newthm}). Hence, the two fan structures on the metric cone induced by the tight spans and the KRW polytopes are not refinements of each other. For a comparison of the various fan structures we consider on the (pseudo)metric cone, with references to the relevant definitions see \Cref{fan_table}.
\def\incomp{\not\asymp}
\def\incomp{\#}
\begin{table}[h!] 
		\centering
        {\small 
		\begin{tabular}{lllll}
			\toprule
			Name  & Reference & Support & Relations \\ 
			\midrule 
            $\mcone{n}$, metric cone & Def.~\ref{def:pmc} & & \\
            $\mconeclosure{n}$, pseudometric cone & Def.~\ref{def:pmc} & & 
            Eucl.~closure of $\mcone{n}$ \\
           $\mathscr{T}_n$, type fan of KRW polytopes & Def.~\ref{def:type_fan_KRW} & $\mcone{n}$ & $\mathscr{F}_n\preceq \Tn{n}$\\
            $\mathscr{F}_n$, Wasserstein fan & Def.~\ref{defWF}  & $\mcone{n}$ & -  \\ 
             $\GPFclosed{n}$ & Def.~\ref{defWF}  & $\mconeclosure{n}$ & $\GPF{n}=\GPFclosed{n}\cap \mcone{n}$ \\ 
            $\MF{n}$, metric fan (wrt.~tight spans) & Def.~\ref{def:metric_fan} & $\mconeclosure{n}$ & $\MF{n}\succeq\GPFclosed{n}$ for $n<5$  \\ 
             &  &  & $\MF{n}\incomp\GPFclosed{n}$ for $n\geq 5$  \\
			\bottomrule  
		\end{tabular} 
        }
	\caption{Overview of the fan structures on the metric cone mentioned in this paper. The relation $\mathscr A \preceq \mathscr B$ means that $\mathscr A$ is a refinement of $\mathscr B$, and $\mathscr A \incomp \mathscr B$ means that 
    neither $\mathscr A \preceq \mathscr B$ nor $\mathscr A \succeq \mathscr B$. 
    } 
		\label{fan_table} 
	\end{table}
\item We compare the fan structure induced by the Wasserstein arrangement on the metric cone with several classes of metrics that have appeared prominently in the literature on phylogenetic analysis. In particular, we show that the space of tree-like metrics and of circular-decomposable metrics are subfans of the Wasserstein fan (\Cref{prop:tl}, \Cref{thm:cd}), thus so is the simplicial fan of circular split networks introduced by Devadoss and Petti \cite{Devadoss-petti} (and appearing also in work of Hacking, Keel and Tevelev \cite{HKT}), see \Cref{DP-HKT}. Moreover, the set of metrics satisfying the \emph{six-point condition} of Dress-Huber-Koolen-Moulton~\cite{SixPoints} is a subfan of the Wasserstein fan as well (\Cref{prop:6p}). However, the set of all totally split-decomposable metrics is not a subfan of the Wasserstein fan (see \Cref{example_decomp}), and neither is the set of all \emph{consistent, totally split-decomposable metrics} (\Cref{rem:consistent}).We also prove that the set of {\em antipodal metrics} is a union of (relatively open) cones of the Wasserstein fan (\Cref{prop:anti}) that form indeed a subfan  (\Cref{antisubfan}). 
These results are summarized in \Cref{phyl_table}.  

\begin{table}[th!] 
		\centering	
		\begin{tabular}{lll}
			\toprule
			Class of metric spaces  
            
            & Supports a subfan of & Reference\\ 
			\midrule 
            tree-like 
            
            & $\Tn{n}$, hence of $\GPF{n}$  & Prop.~\ref{prop:tl}\\ 
            circular-decomposable 
            
            & $\Tn{n}$, hence of $\GPF{n}$ & Thm.~\ref{thm:cd}\\ 
            satisfying the six-point-condition 
            
            & $\Tn{n}$, hence of $\GPF{n}$ & Prop.~\ref{prop:6p}\\ 
            antipodal
            & $\Tn{n}$, hence of $\GPF{n}$
            & Prop.~\ref{prop:anti}, Rem.~\ref{antisubfan} \\
            split-decomposable  
            
            & Neither $\GPF{n}$ nor $\Tn{n}$ ($n\geq 5$) & Ex.~\ref{example_decomp}\\ 
            consistent split-decomposable 
            
            & Neither $\GPF{n}$ nor $\Tn{n}$ ($n\geq 5$)& Prop.~\ref{rem:consistent} \\ 
			\bottomrule  
		\end{tabular} 
	\caption{Overview on whether some notable properties of metric spaces are encoded by the combinatorial type of the KRW polytope. This is expressed by whether the corresponding subset of the metric cone is a subfan of the type fan $\Tn{n}$. Note that we provide explicit witnesses for the negative results.} 
		\label{phyl_table} 
	\end{table}

We suggest the following open problem, which we cannot address with our current techniques.
\begin{question}[Tree-based metrics] The notion of \emph{tree-based phylogenetic network} has gathered considerable attention in the literature on phylogenetic analysis, since its introduction by Francis and Steel in 2015  \cite{Francis_Steel,Francis_Semple_Steel}. The space of such tree-based phylogenetic networks has been the object of recent research by Fischer and Francis 
\cite{Fischer_Francis}.
In our context, it is therefore natural to ask whether the metric spaces arising from tree-based phylogenetic networks are a union of strata of the Wasserstein fan.
\end{question}

\item We derive a formula for the number of edges of the KRW polytopes of non-generic strict metrics. To the best of our knowledge, the following is open.

 \begin{question}
 Compute the full $f$-vector of KRW polytopes of non-generic strict metrics. Can the approach and the formulae given in \Cref{sec:strict_case} be generalized?
 \end{question}

\end{enumerate}

\subsection{Roadmap of the article}

In~\Cref{sec:BG} we introduce the metric cone, we establish some terminology related to polyhedral geometry, we state Gordon and Petrov's characterization of faces of the KRW polytope and prove a few preparatory lemmas.
Our main tool for capturing the combinatorial structures of KRW polytopes, the Wasserstein arrangement, is defined and discussed in~\Cref{sec:WA}.
In~\Cref{sec:triangulations} we explain the connection between KRW polytopes and a certain class of regular triangulations of type $A$  root polytopes.
\Cref{sec:tight_spans} is devoted to the relationship between KRW polytopes and tight spans.
In~\Cref{sec:subfans} we explore several noteworthy subfans of the Wasserstein arrangement stemming from combinatorial phylogenetics, and in~\Cref{sec:strict_case}
we present some results concerning the $f$-vector of KRW polytopes associated with strict metrics.
Lastly, in~\Cref{sec:compuations} we summarize our computational findings on KRW polytopes for metrics on up to $6$ points. The paper is rounded up by two appendixes. In~\Cref{ats} we focus on the relationship between the tight span and a certain subcomplex that we call the ``pruned tight span'', while  in~\Cref{sec:five_point_metrics_comp} we list the types of KRW polytopes of generic and strict non-generic metric spaces on $n=5$ points along with some of their enumerative invariants.

\subsection*{Acknowledgments}

We thank Nick Early and Benjamin Schr\"oter for fruitful discussions, as well as Laura Casabella and Michael Joswig for helpful comments on the manuscript.
We are grateful to J\"org Rambau for providing us with relevant triangulations of the root polytope computed using \texttt{TOPCOM}.
We would like to thank Alexander Postnikov for pointing out a mistake in an earlier version of the article.
We also thank two anonymous referees for their valuable comments and suggestions.

LK and LM are funded by the Deutsche Forschungsgemeinschaft (DFG, German Research Foundation)
– Project-ID 491392403 – TRR 358 and LK is supported by
the DFG SPP 2458 -- 539866293.

\section{The main characters}
\label{sec:BG}

\subsection{The metric cone}

 \begin{definition}
 An $n$\emph{-metric} is a real symmetric $n\times n$ matrix~$\metric$ with entries $\metric_{ij}$ for $1\le i,j\le n$ satisfying
\begin{enumerate}
	\item $\metric_{ii}=0$ for all $1\le i \le n$,
	\item $\metric_{ij}>0$ for all $1\le i\neq j\le n$ and
	\item $\metric_{ij}+\metric_{jk}\ge \metric_{ik}$ for all $1\le i,j,k\le n$.
\end{enumerate}
The metric $\metric$ is called \emph{strict} if $\metric_{ij}+\metric_{jk}> \metric_{ik}$ for all $j\in \left[n\right]\setminus \{i,k\}$.
\end{definition}

Every $n$-metric $\metric$ is given as a symmetric $n\times n$ matrix with zero diagonal, thus it is determined by the $\binom{n}{2}$ values $\metric_{ij}$ with $i < j$.
This establishes a bijective correspondence between the set of all $n$-metrics and the following subset of  $\mathbb R^{\binom{n}{2}}$.

\begin{definition} \label{def:pmc}
Consider the vector space $\mathbb R^{\binom{n}{2}}$ with coordinates $x_{\{i,j\}}$ indexed by pairs of elements of $[n]$. The {\em metric cone} on $n$ elements is the subset $\mcone{n} \subseteq \mathbb R^{\binom{n}{2}}$ defined by 
$$
x_{\{i,j\}}>0, \quad x_{\{i,j\}} + x_{\{j,k\}} \geq x_{\{i,k\}},  \textrm{ for all pairwise distinct } i,j,k \in [n].
$$
The {\em pseudometric cone}, denoted by $\mconeclosure{n}$, is the Euclidean (topological) closure of $\mcone{n}$. 
\end{definition}

\begin{remark}
The space $\mconeclosure{n}$ is a polyhedral cone. It contains points that do not correspond to metrics but rather to pseudometrics (where $\metric_{ij}=0$ can be allowed). The set of all {\em strict} metrics on $n$ elements is the interior of $\mcone{n}$. 
\end{remark}

\begin{example}[Split pseudometrics] \label[example]{defsplits}
Let $A\uplus B$ be a non-trivial bipartition of $[n]$. The pseudometric $\splitpm{A}{B}$ with
$$
(\splitpm{A}{B})_{i,j}:= \min \{\vert A\cap\{i,j\}\vert,\vert B\cap\{i,j\}\vert \}
$$
has value $0$ when $i$ and $j$ are in the same part of the bipartition and has value $1$ otherwise. Any metric of this form is called an \emph{$n$-split}, the partition itself $\sigma = A\uplus B$ is called a \emph{split}.
\end{example}

\subsection{Admissible graphs and strict metrics
} \label{sec:GP}
Let $\metric$ be an $n$-metric and $F$ any face of the polytope $KRW(\metric)$.
We associate with the face $F$ a directed graph $G(F)$ on the vertex set $[n]$ which contains the edge $(i,j)$ if the point $p_{ij}$ (defined in \eqref{equ:pij}) lies on $F$. Following \cite{GordonPetrov}, the collection of graphs of the form $G(F)$ is called \emph{the combinatorial structure} of $\KRW{\metric}$.

  \begin{theorem}[\cite{GordonPetrov}, Theorem 3]\label{thm:gordon_petrov}
  Let $\metric$ be an $n$-metric and $G = ([n], E)$ be a directed graph with the set of vertices $[n]$. The following are equivalent: 
  \begin{enumerate}
      \item There exists a facet $F$ of the polytope $\KRW{\metric}$ containing every vertex $p_{ij}$ for $(i,j)\in~E$.
      \item For any $k$ and any array of directed edges $(x_i,y_i), 1\leq i \leq k$ of $G$ with all $x_i$ pairwise distinct and all $y_i$ pairwise distinct: 
		\begin{equation}\label{eq:gp_inequality}
		\sum_{i = 1}^k \metric_{x_iy_i} \leq \sum_{i= 1} ^k \metric_{x_iy_{i+1}},
		\end{equation} where $y_{k +1} = y_1$. 
  \end{enumerate}
  \end{theorem}

   \begin{definition} 
		A directed graph $G$ on the vertex set $[n]$ is called \emph{admissible} for an $n$-metric $\metric$, or $\metric$-admissible, if it satisfies one of the two equivalent conditions of \Cref{thm:gordon_petrov}. 

  \end{definition}
  
  In general, the pairwise distinctness of the $x_i$ and $y_i$ does not mean that the edges $(x_i,y_i)$ are vertex-disjoint. The following lemma addresses this distinction.

\begin{lemma}\label[lemma]{lem:disjoint}
Let $\rho$ be an $n$-pseudometric and let $G$ be a directed graph on the set $[n]$. The following are equivalent
\begin{itemize}
\item[(i)] $G$ satisfies \Cref{thm:gordon_petrov}.(2) 
\item[(ii)] $G$  satisfies the inequalities \eqref{eq:gp_inequality} for
\begin{itemize}
\item[(ii.1)] every array of directed edges of $G$ the form $(x,y),(y,z)$, and
\item[(ii.2)]  every array $(x_1,y_1),\ldots,(x_k,y_k)$ of vertex-disjoint directed edges of $G$, for all $k\geq 1$.
\end{itemize}
\end{itemize}
\end{lemma}
\begin{proof}
Clearly (i) implies (ii). For the other direction, we will say that a given array of directed edges $(x_1,y_1),\ldots,(x_k,y_k)$ is $m$-overlapping if $\vert\{x_1,\ldots,x_k\}\cap\{y_1,\ldots,y_k\}\vert = m$. Let $m\geq 1$,  assume that \eqref{eq:gp_inequality} holds for all arrays that are at most $m-1$ overlapping and
consider an $m$-overlapping array $(x_1,y_1),\ldots,(x_k,y_k)$. Since $m\geq 1$ and \eqref{eq:gp_inequality} is invariant under cyclic permutation of indices, without loss of generality we may suppose that $x_1=y_j$ for some $j>1$. If $k=2$, then the associated inequality holds by assumption (ii.1). Assume  $k>2$. If $j=2$, then
\begin{align*}
\metric_{x_1,y_1} + \metric_{x_2,x_1} + \metric_{x_3,y_3} + \ldots +\metric_{x_k,y_k} &\leq \metric_{x_2,y_1} + \metric_{x_3,y_3} + \ldots +\metric_{x_k,y_k}\\
 &\leq \metric_{x_2,y_3} + \metric_{x_3,y_4} + \ldots +\metric_{x_k,y_1} \\
 &= \metric_{x_1,x_1}+\metric_{x_2,y_3} + \ldots +\metric_{x_k,y_1} 
\end{align*}
(the first inequality by (ii.1), the second inequality by (ii.2) for the $m-1$-overlapping array $(x_2,y_1),(x_3,y_3),\ldots,(x_k,y_k)$, the last equality since $\metric_{x,x}=0$ for all $x$), and we obtain inequality \eqref{eq:gp_inequality} for the given array. The case $j=k-1$ is analogous, so consider the case $2<j<k-1$. Then, the triangle inequalities
$\metric_{x_{k},y_1} \leq \metric_{x_{k},x_1} + \metric_{x_1,y_1}$ and 
$\metric_{x_{j-1},y_1} \leq \metric_{x_{j-1},x_1}+  \metric_{x_1,y_1}$ are nontrivial, and substracting the first from the second gives
$\metric_{x_{j-1},y_1} - \metric_{x_{k},y_1} \leq  \metric_{x_{j-1},x_1} - \metric_{x_{k},x_1}$. Reordering and recalling that $x_1=y_j$, we have
\begin{equation}\label{eq:2tri}
\metric_{x_{j-1},y_1} + \metric_{x_{k},y_j} \leq \metric_{x_{j-1},y_j} + \metric_{x_{k},y_1}.
\end{equation}
Now consider the arrays $(x_1,y_1),\ldots,(x_{j-1},y_{j-1})$ and $(x_j,y_j),\ldots,(x_k,y_k)$: both are at most $(m-1)$-overlapping and thus we assume they satisfy \eqref{eq:gp_inequality}. Adding the inequalities \eqref{eq:gp_inequality} for these arrays and using \eqref{eq:2tri} we obtain
\begin{align*}
\sum_{i=1}^{j-1} \metric_{x_i,y_i} +
\sum_{i=j}^{k} \metric_{x_i,y_i}
&\leq 
\left[\sum_{i=1}^{j-2}\metric_{x_i,y_{i+1}} \right] + \metric_{x_{j-1},y_{1}}
\left[\sum_{i=j}^{k-1}\metric_{x_i,y_{i+1}} \right] + \metric_{x_{k},y_{j}}\\
&\leq 
\left[\sum_{i=1}^{j-2}\metric_{x_i,y_{i+1}} \right] + \metric_{x_{j-1},y_{j}}
\left[\sum_{i=j}^{k-1}\metric_{x_i,y_{i+1}} \right] + \metric_{x_{k},y_{1}}\\
\end{align*}
That is, \eqref{eq:gp_inequality} is satisfied by  any $m$-overlapping array $(x_1,y_1),\ldots,(x_k,y_k)$ if it is satisfied for all arrays of strictly smaller overlap. Since this holds for any $m\geq 1$ and since condition (ii.2) states that \eqref{eq:gp_inequality} holds for all $0$-overlapping arrays, the claim follows.
\end{proof}

\begin{lemma}\label[lemma]{lem:admissible_paths}
    A directed path on $n\geq 1$ edges with vertices $x_1, \dots, x_{n+1}$ is $\metric$-admissible (i.e., satisfies \Cref{thm:gordon_petrov}.(2)) if and only if 
    \begin{itemize}
    \item[(a)] all triangle inequalities of the form $\metric_{x_i, x_j} + \metric_{x_j, x_l} \geq \metric_{x_i, x_l}$ for $1\leq i \leq j \leq l \leq n+1$ are equalities and
    \item[(b)] all inequalities of the form~\eqref{eq:gp_inequality} hold for sets of vertex-disjoint edges in \\ $\{(x_1,x_2),(x_2,x_3),\dots,(x_n,x_{n+1})\}$.
    \end{itemize}
\end{lemma}
\begin{proof} 
If a directed path satisfies (a) and (b) with respect to $\metric$, then it satisfies condition (ii) in \Cref{lem:disjoint} (indeed, for directed paths (a) is equivalent to (ii.1) and (b) equivalent to (ii.2)), and thus it is admissible.

For the other implication note that if a directed path $x_1,\ldots,x_{n+1}$ is admissible, it satisfies all inequalities in \Cref{thm:gordon_petrov}.(2), and so it trivially satisfies condition (b).   We have to show that admissible directed paths satisfy (a). The case $n=1$ is trivial, thus 
 let $n\geq 2$, assume by induction hypotheses that all admissible paths with less than $n$ edges satisfy (a), and consider an admissible directed path $x_1,\ldots,x_{n+1}$.  
Fix a triple $1\leq i \leq j \leq l \leq n+1$. If $\vert l-i\vert <n$, condition (a) holds by induction hypothesis on either of the paths $x_1,\ldots,x_n$ and $x_2,\ldots,x_{n+1}$. Otherwise we have $i=1$, $l=n+1$ and we must show $\metric_{x_1,x_j}+\metric_{x_j,x_{n+1}}=\metric_{x_1,x_{n+1}}$, which is trivial if $j=1$ or $j=n+1$ and so we can assume without loss of generality that $1<j<l$. By (b), the inequality~\eqref{eq:gp_inequality} holds for the array  $(x_1,x_2),(x_{n},x_{n+1}),\ldots,(x_2,x_{3})$, i.e.,
$$
\metric_{x_1,x_2} + \metric_{x_2,x_3} \ldots +\metric_{x_{n},x_{n+1}}
\leq \metric_{x_1,x_n+1} +\metric_{x_2,x_2} + \ldots +\metric_{x_n,x_n} 
$$
Since $j\neq 1$ and $j\neq n+1$ the induction hypothesis applies to the paths $x_1,\ldots,x_j$ and $x_j,\ldots,x_{n+1}$ and hence we can rewrite the l.h.s. as $\metric_{x_1,x_j}+\metric_{x_j,x_{n+1}}$, and the r.h.s. trivially equals $\metric_{x_1,x_{n+1}}$ since $\metric$ is a metric.
The triangle inequality $\metric_{x_1,x_j}+\metric_{x_j,x_{n+1}}\ge \metric_{x_1,x_{n+1}}$ now implies the claimed equality and hence the path $x_1,\ldots,x_{n+1}$ satisfies condition (a).
\end{proof}

\begin{corollary}[{\cite{GordonPetrov} Corollary 1.(2)}]\label[corollary]{strict_nopaths} If $\metric$ is strict, $\metric$-admissible graphs do not contain any directed path of length greater than $1$.

  \end{corollary}
  \begin{proof}
  This is immediate from \Cref{lem:admissible_paths} since, by definition, for strict metrics all triangle inequalities are strict. 
  \end{proof}

\begin{definition}
Let $\metric$ be an $n$-metric. A {\em $\metric$-tight triple} is any triple $(i,j,k)$ of distinct elements in $[n]$ for which $\metric_{i,j}+\metric_{j,k}=\metric_{i,k}$.
\end{definition}

We can now give a description of admissibility of graphs that will be of use later. 

\begin{corollary}\label[corollary]{cor:disjointight}  
Let $\metric$ be an $n$-metric and let $G$ a directed graph on the vertex set $[n]$. Then $G$ is admissible if and only if 
\begin{itemize}
\item[(a)] it satisfies inequality \eqref{eq:gp_inequality} for all arrays of vertex-disjoint edges and
\item[(b)] if three elements $i,j,k\in [n]$ appear in this order as vertices in a directed path in $G$ then the triple $(i,j,k)$ is $\metric$-tight.
\end{itemize}
\end{corollary}
\begin{proof} If $G$ is admissible, then (a) holds by \Cref{lem:disjoint}.(ii.2) and (b) holds by \Cref{lem:admissible_paths} since every subgraph of $G$ is admissible as well. Now suppose that $G$ satisfies (a) and (b). We check that $G$ satisfies condition (ii) of  \Cref{lem:disjoint}, with (ii.2) being equivalent to (a). In order to check (ii.1) consider any array $(x,y),(y,z)$ of directed edges of $G$. Then $x,y,z$ are indeed three consecutive vertices in a directed path of $G$, and so by (b) the triple $x,y,z$ is $\metric$-tight. Since $\metric$ is a metric, tightness of $x,y,z$ is equivalent to the inequality $\metric_{x,y}+\metric_{y,z}\leq \metric_{x,z}$. Thus $(x,y),(y,z)$ satisfies (ii.1).
\end{proof}

    We close this section with a characterization of generic metric spaces (in the sense of \Cref{def:KRW}).
  
  \begin{theorem}[\cite{GordonPetrov}, Theorem 5]
      	A strict $n$-metric $\metric$ is generic if and only if for arbitrary $2k$ distinct points $x_1,\dots,x_k, y_1,\dots,y_k$ in $\left[n\right]$ the minimum of the terms \[\sum_{i = 1}^k \metric_{x_i, y_{\pi(i)}}\] is attained by a unique permutation $\pi\in S_k$.
  \end{theorem}

  \subsection{Polyhedra and their combinatorial type}

  We refer to \cite{ZieglerLoP} for basics and terminology about polyhedra and only recall some fundamental definitions. A \emph{polyhedron} is any subset of Euclidean space that is obtained by intersecting a finite number of closed half-spaces. A \emph{polytope} is a bounded polyhedron -- equivalently, the convex hull of a finite set of points. Every polyhedron is topologically closed. The \emph{affine hull} of a polytope $P$ is the smallest affine subspace containing $P$, and the relative interior of $P$ is the interior of $P$ as a subset of its affine hull. If $P$ is any convex set in Euclidean space, we call {\em face} of $P$ any intersection of $P$ with the boundary of a closed half-space containing $P$.
  
  A  polyhedral complex is a family $\mathscr{K}$ of polyhedra (all in the same ambient space) that contains the empty polyhedron and all faces of each of its members, and such the intersection of any two members of $\mathscr{K}$ is a face of both. The {\em face poset} of $\mathscr{K}$ is the set $\mathscr{K}$ partially ordered by inclusion. 

\begin{definition}[Labeled and unlabeled combinatorial type of polyhedral complexes]
  We say that two finite polyhedral complexes $\mathscr{K},\mathscr{K'}$ {\em have the same combinatorial type} (or ``are combinatorially isomorphic", written $\mathscr{K}\simeq \mathscr{K'}$) if there is an abstract polyhedral isomorphism (i.e., a face-preserving PL-homeomorphism) between them. If some or all of the cells of $\mathscr{K}$ and $\mathscr{K'}$ are labeled using the same set of labels, and there is an abstract polyhedral isomorphism that restricts to a label-preserving bijection on the set of labeled faces, then $\mathscr{K}$ and $\mathscr{K'}$ have {\em the same ``labeled" combinatorial type}, and we express this as $\mathscr{K}\simeq_\ell \mathscr{K'}$.

  Note that $\mathscr{K}\simeq \mathscr{K'}$ if and only if the posets of faces of $\mathscr{K}$ and $\mathscr{K'}$ are isomorphic \cite[\S 2.2]{ZieglerLoP}, and $\mathscr{K}\simeq_\ell \mathscr{K'}$ is equivalent to the existence of such an isomorphism of posets of faces that preserves the labels. 
  \end{definition}

  \begin{remark}[On the term {\em Combinatorial structure}]\label[remark]{cts} 
  If $\metric$ is an $n$-metric, the polytope $\KRW{\metric}$ together with all its faces defines a polyhedral complex, whose vertices are naturally labeled by ordered pairs of distinct elements of $[n]$ (i.e., if $p_{ij}$ is a vertex of $\KRW{\metric}$, it is labeled by $(i,j)$). Note that the ``combinatorial structure" in the sense of Gordon and Petrov, see \Cref{sec:GP}, is equivalent to the labeled combinatorial type. Indeed, by \Cref{thm:gordon_petrov} the KRW polytopes of two metrics $\metric_1$, $\metric_2$ have the same labeled combinatorial type if and only if the set of admissible labeled graphs for $\rho_1$ coincides with the set of admissible labeled graphs for $\metric_2$, where the labeling of the graphs is induced by the labeling of the edges $(i,j)$, which are in one-to-one correspondence with the vertices $p_{ij}$.
  \end{remark}

\begin{remark}[Cones and fans]\label[remark]{remcone}

We define a {\em cone} as any subset $\mathbf{C}\subseteq \mathbb R^d$ that contains every linear combination with non-negative coefficients of any of its finite subsets, see \cite[\S 1.1]{ZieglerLoP}.  We call {\em polyhedral cone} any polyhedron 
that is also a cone.

A {\em half-open} polyhedral cone is a polyhedral cone from which the relative interior of some of its boundary faces has been  removed -- thus the

metric cone $\mcone{n}$ is half-open, since it can be obtained from the (polyhedral) $\mconeclosure{n}$ by removing all the faces intersecting some coordinate hyperplane. 

A {\em polyhedral fan} is any polyhedral complex all of whose members are polyhedral cones, see \cite[\S 7.1]{ZieglerLoP}.
We will need to consider families of possibly half-open polyhedral cones that contain all faces of each of their members and such that the intersection of any two of its members is a face of both: we will call this a {\em half-open fan}. We will simply use the word {\em fan} if its type is clear from the context. The {\em support} of any fan is, by definition, the union of its members. By a subfan of a (half-open) fan $\mathscr X$ we will mean a set of cones of $\mathscr X$ that contains all of its boundary faces that are members of $\mathscr X$ (i.e., any $\mathscr Y\subseteq \mathscr X$ such that, for every $y\in \mathscr Y$, if a face of $y$ is a member of $\mathscr X$ then it is also a member of $\mathscr Y$). Note that a set of members of a (half-open) fan is a subfan if its support is closed inside the support of the bigger fan.
This allows us, e.g., to speak of the ``Wasserstein fan" $\GPF{n}$ in \Cref{defWF}. 
\end{remark}

\begin{definition}\label{def:type_fan_KRW}
Let $n\ge 2$ be an integer.
We write $\Tn{n}$ for the \emph{type fan of KRW~polytopes}, i.e., the fan supported on the metric cone $\mcone{n}$ such that two metrics  $\metric^{(1)},\metric^{(2)}$ are in the open part of the same cone of $\Tn{n}$ if and only if the the KRW polytopes $\KRW{\metric^{(1)}}$ and $\KRW{\metric^{(2)}}$ have the same labeled combinatorial type (equivalently, the same set of admissible directed graphs).
\end{definition}

Given a real $m\times n$ matrix $A$, McMullen studied the space of all polytopes $Ax\le b$ for $b\in \R^n$ ~\cite{McMullen}.
Hence all normal vectors of these polytopes appear as rows of $A$.
McMullen proved that there is a polyhedral cone, the so-called \emph{inner region}, where all rows correspond to irredundant linear conditions.
Another result of McMullen states that this cone is subdivided into \emph{type cones} where the interior of each cone corresponds to a fixed labeled combinatorial type of polytopes.
The collection of these type cones is the \emph{type fan}.

In the case that the rows of $A$ are $e_i-e_j$ for all $i\neq j$, this yields the \emph{type fan of alcoved polytopes}, an actively studied object in its own right~\cite{EKM25,BSS26}.
Restricting the corresponding inner region to the metric cone yields the type fan of (duals of) KRW polytopes which explains our terminology introduced above.

\begin{remark}
According to~\Cref{thm:gordon_petrov}, the region within the metric cone consisting of the metrics whose KRW polytopes have a fixed combinatorial type is cut out by a collection of linear inequalities.
This shows that $\Tn{n}$ is indeed a fan, and the relatively open part of any of its cones is the relative interior of a polyhedral cone.
\end{remark}

\section{The Wasserstein arrangement and the Wasserstein fan}
\label{sec:WA}

We introduce an arrangement of hyperplanes whose faces encode the combinatorial structures of KRW polytopes.

\begin{definition}[Wasserstein arrangement]
    Let $n$ be a positive integer. We define a hyperplane arrangement in $\R^{\binom{n}{2}}$ as follows. 
    Given $1<k$ and $k$-tuples $\mathbf{a}=(a_1,\ldots,a_k)$, $\mathbf{b}=(b_1,\ldots,b_k)$ with $a_i,b_i\in[n]$ for all $i=1,\ldots,k$, define a hyperplane 
    
    $$
    H_{\mathbf{a,b}} :=\left\{x\in \R^{\binom{n}{2}}\left\vert  \sum_{i=1} ^kx_{\{a_i,b_i\}} = \sum_{i=1}^k x_{\{a_i,b_{i+1}\}}\right.\right\}
    \quad \textrm{where }b_{k+1}=b_1.
    $$
    The ``non-negative side'' of $H_{\mathbf{a,b}}$ is
        $$
    H_{\mathbf{a,b}}^+ :=\left\{x\in \R^{\binom{n}{2}}\left\vert \sum_{i=1}^{k} x_{\{a_i,b_i\}} \leq \sum_{i=1}^{k} x_{\{a_i,b_{i+1}\}}\right.\right\}.
    $$
    The \emph{Wasserstein arrangement} is then the set of hyperplanes
    $$
    \CA{n}:=\left\{H_{\mathbf{a,b}}  \left\vert
    \begin{array}{l}
    1<k\leq n,\,\, \mathbf{a,b}\in [n]^k\\
    a_1,\ldots, a_k,b_1,\ldots, b_k \textrm{ mutually distinct}
    \end{array}
    \right.\right\}.
    $$
\end{definition}
\begin{remark}
    In the case $ k = 2$, the tuple $(a_1,a_2)$ forms the same hyperplane when combined with either of the tuples $(b_1, b_2)$ or $(b_2, b_1)$. In order for $H_{\mathbf{a,b}}^+$ to be well-defined, we establish the convention that we chose the tuple starting with $\min\{b_1,b_2\}$. 
\end{remark}
\begin{definition}[Wasserstein fan]\label[definition]{defWF} 
Let $\GPF{n}$ denote the fan determined by the intersection of the Wasserstein arrangement $\CA{n}$ with the metric cone $\mcone{n}$. 
	We call $\GPF{n}$ the \emph{Wasserstein fan}.
	In addition, write $\GPFclosed{n}$ for the polyhedral fan of pseudometrics defined by the intersections of $\CA{n}$ with all faces of the cone $\mconeclosure{n}$.
\end{definition}

\begin{remark}\label{FanArr}

  	The metric cone $\mcone{n}$ intersects every face of the Wasserstein arrangement~$\CA{n}$.
    Indeed, let $p$ be any point on some face of $\CA{n}$.
    The point $p+k\mathbf{1}$ for some $k\ge 0$ and $\mathbf{1}$ the all-ones vector in $\R^{\binom{n}{2}}$ lies on exactly the same hyperplanes of the arrangement $\CA{n}$ as the point $p$ since the defining equations involve the same number of variables on both sides of the equation.
    But for $k$ large enough, we observe that the point $p+k\mathbf{1}$ is in $\mcone{n}$. For the analogous statement about the secondary fan of the second hypersimplex see \cite[Lemma 24]{CJK24}.
    
    Therefore,  the arrangement $\CA{n}$ and the fan $\GPF{n}$ have the same number of $i$-cells for any~$i$.
    In particular, we can count the number of maximal cells of the fan $\GPF{n}$ by counting the number of chambers of the arrangement $\CA{n}$.
\end{remark}

\begin{example}
	The Wasserstein arrangement $\CA{4}$ in $\R^{\binom{4}{2}}$ consists of the three hyperplanes given by the equations
\begin{align*}
	H_{(1,2),(3,4)}:& \quad   x_{\{1,3\}} + x_{\{2,4\}}=x_{\{1,4\}} + x_{\{2,3\}},\\
	H_{(1,3),(2,4)}:& \quad   x_{\{1,2\}} + x_{\{3,4\}}=x_{\{1,4\}} + x_{\{2,3\}},\\
	H_{(1,4),(2,3)}:& \quad   x_{\{1,2\}} + x_{\{3,4\}}=x_{\{1,3\}} + x_{\{2,4\}}.
\end{align*}
This arrangement has a $4$-dimensional lineality space, see~\Cref{lem:lineality_space}. \Cref{fig:arrangement4} represents
$\CA{4}$ modulo its lineality space.
\end{example}

\begin{remark}\label[remark]{rem:cycle_arr}
Each hyperplane $H_{\mathbf{a,b}}$ corresponds to a cycle $C_{H_{\mathbf{a,b}}}$ of the complete graph $K_n$ determined by the sequence of vertices $a_k,b_k,\ldots,a_1,b_1$ (see \Cref{fig:cycle_example} for an example with  $k=3$). Write $C_{H_{\mathbf{a,b}}}^+:=\{(a_i, b_i)_{1\leq i\leq k}\}$ and $C_{H_{\mathbf{a,b}}}^-:=\{(a_i, b_{i+1})_{1\leq i\leq k}\}$. 
The edges in the cycle have alternating orientations (all vertices have indegree either 0 or 2). 

Whether the sets $C_{H_{\mathbf{a,b}}}^+$ and $C_{H_{\mathbf{a,b}}}^-$ are admissible for a given metric depends on which side of $H_{\mathbf{a,b}}$ the metric lies.

\begin{figure}[hbt]
	\includegraphics[width=0.25\textwidth]{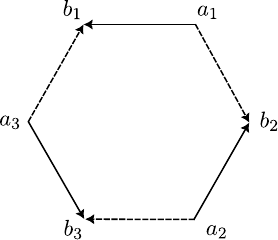} 
	\caption{A cycle defining a hyperplane $H_{\mathbf{a,b}}$. The edges in $C_{H_{\mathbf{a,b}}}^+$ and $C_{H_{\mathbf{a,b}}}^-$ correspond to solid and dashed arrows, respectively.}
	\label{fig:cycle_example}
\end{figure}

\end{remark}

\begin{definition} \label[definition]{def:2k_cycle_planes}
    For $k \in \mathbb{N}$, denote by \[\CA{n}^k := \left\{H_{\mathbf{a,b}} \left\vert
    \begin{array}{l}
     a,b\in [n]^k\\
    a_1,\ldots, a_k,b_1,\ldots, b_k \textrm{ mutually distinct}
    \end{array}
    \right.\right\}\]
the set of hyperplanes in $\CA{n}$ corresponding to cycles of length $2k$ in the complete graph $K_n$.
\end{definition}

\begin{proposition}\label[proposition]{prop:size_wn}
	The arrangement $\CA{n}$ consists of
	\[
	|\CA{n}|= |\bigsqcup_{k =2}^{\lfloor \frac{n}{2}\rfloor} \CA{n}^k| = \frac{1}{2}\sum_{k=2}^{\lfloor \frac{n}{2}\rfloor} \binom{n}{2k}(2k-1)!
	\]
	many hyperplanes.
\end{proposition}
\begin{proof}
	As discussed in \Cref{rem:cycle_arr}, each hyperplane in $\CA{n}$ is determined by an oriented cycle in in the complete graph $K_n$ of even length at least four.
	Such cycles arise by choosing $2k$ vertices amongst the $n$ vertices and then choosing a cyclic permutation of length $2k$.
	There are $(2k-1)!$ such permutations.
	As reversing the orientation of the cycle yields the same hyperplane, we count every hyperplane exactly twice.
\end{proof}

\begin{remark} \label{rem:wassersteinarr}

 In order to study the combinatorial types of strict metrics, it suffices to consider the intersection of the interior of $\mcone{n}$ with $\CA{n}$, since by \Cref{strict_nopaths} admissible graphs for strict metrics cannot contain directed paths with more than one edge.
\end{remark}

\begin{definition}\label[definition]{elem_splits}
Recall \Cref{defsplits}. The pseudometrics of the form $\splitpm{\{i\}}{[n]\setminus \{i\}}\in \R^{\binom{n}{2}}$ for $i=1,\ldots,n$ are called {\em elementary splits}. 

Let $W$ denote the linear span of the elementary splits and let  $\elsplitsclosed{n}:=\mconeclosure{n}\cap W$ be the set of all pseudometrics that are linear combination of elementary splits. 
Accordingly, we write $\elsplits{n}:=\mcone{n}\cap W$ for the cone of metrics that are combinations of elementary splits.
\end{definition}
\begin{remark} \label[remark]{rem_dimW}
The set of elementary splits is ``compatible'' in the sense of \cite{BandeltDress} (e.g., because the metric obtained as the sum of all elementary splits is represented by a tree with $n$ leaves, each at distance $1$ from a unique inner vertex), and thus linearly independent by \cite[Corollary 4]{BandeltDress}. Therefore, $\dim(W)=n$ and  $\elsplitsclosed{n}$ is an $n$-dimensional simplicial cone given as the convex hull of the rays generated by the elementary splits. 
\end{remark}

\begin{lemma}\label[lemma]{lem:lineality_space}
The lineality space of the Wasserstein arrangement $\CA{n}$ is the $n$-dimensional subspace $W$ spanned by the elementary splits (i.e., $W=\bigcap \CA{n}$).
\end{lemma}
\begin{proof}
Let $L:=\bigcap \CA{n}$.
We first prove that $W\subseteq L$.
If a metric $\metric$ is in a hyperplane $H_{\mathbf{a,b}}$ and we define \[\metric' := \metric +\lambda \cdot \splitpm{\{i\}}{[n]\setminus \{i\}}\in \R^{\binom{n}{2}}\] for $\lambda > 0$ and some $i$, then the $\lambda$ appears in  an equal number of summands on both sides of the defining equation of $H_{\mathbf{a,b}}$ and therefore, $\metric'\in H_{\mathbf{a,b}}$. Thus, $W$ is contained in $H_{\mathbf{a,b}}$ for all $\mathbf{a}$ and $\mathbf{b}$ and is thus contained in the lineality space of the arrangement $\CA{n}$.

Now it is enough to prove that $\dim(W)\geq\dim(L)$. Since $\dim(W)=n$ (\Cref{rem_dimW}), it is enough to find ${\binom{n}{2}} - n$ hyperplanes in $\CA{n}$ whose normal vectors are linearly independent. 

For this, identify the elements of the metric space $[n]$ with the vertices of a regular $n$-gon and consider the set $D$ of all pairs of distinct elements of $[n]$ that define a diagonal (i.e., not an edge) of the $n$-gon. Thus $\vert D\vert = \binom{n}{2}-n$. For $\{i,j\}\in D$ consider
$
H_{(i,j+1),(j,i+1)}
$ (we understand the sum modulo $n$, so that $n+1=1$), with normal vector
$v_{ij}= e_{\{i,j\}} + e_{\{i+1,j+1\}} - e_{\{i,i+1\}} - e_{\{j,j+1\}}$.

We claim that the set $\{v_{ij} \mid \{i,j\} \in D\}$ is a linearly independent family. In order to do so, we arrange these vectors as columns of a matrix.
Consider the following total ordering on the set of all pairs of distinct elements of  $[n]$: 
\begin{align*}
\{1,2\}\prec \{2,3\}\prec \dots\prec \{n,1\}\prec \{1,3\}\prec \{2,4\}\prec \dots \prec  \{n,2\}\prec \{1,4\}\prec\{2,5\}\prec\dots
\end{align*}
Identifying sides and diagonals of a regular $n$-gon by the associated (unordered) pairs of indices this corresponds to ordering these segments in increasing order of length and, for sides and diagonals of the same length $k$, in the cyclic order starting with the diagonal $\{1,k\}$.

Let $A$ be the $\binom{n}{2} \times \vert D \vert$-matrix whose columns are the vectors $v_{ij}$ listed in increasing $\prec$-order of $\{i,j\}$, and where the $k$-th row corresponds to the coordinate $e_{\{i,j\}}$ with $\{i,j\}$ in $k$-th place with respect to the ordering $\prec$. Then $A$ has the following block-form:
$$
A =\left(\begin{array}{c|c|c|c}
W_1 & W_2 & W_3 & \cdots \\\hline
A_{1} & 0 & 0 & \cdots \\
 0 & A_2 & 0 & \cdots \\
  0 & 0 & A_{3} & \cdots \\
\vdots &  & & \ddots \\
\end{array}\right);\quad
\textrm{ with }
A_j = \begin{pmatrix}
1 & 0 & 0 & \cdots \\
1 & 1 & 0 & 0 & \cdots \\
0 & 1 & 1 & 0 & \cdots \\
0 & 0 & 1 & 1 & \ddots \\
\vdots &  & \ddots & \ddots & \ddots
\end{pmatrix},
$$
and the matrices $W_i$ have $n$ rows.
Since clearly each $A_i$ has full rank, so does $A$. 
This proves that the $v_{ij}$ are linearly independent and thus that
\[
\dim(L) \leq \dim\left(\bigcap_{\{i,j\}\in D} H_{\{i,j\},\{i+1,j+1\}}\right) = {\binom{n}{2}}- \vert D \vert =n.\qedhere
\]
\end{proof}

The following proposition relates the combinatorial structure of KRW polytopes with the face structure of the Wasserstein fan.

\begin{theorem}\label{prop:refinement2}
Let $n\ge 2$ be an integer and consider two $n$-metrics $\metric^{(1)},\metric^{(2)}$. If  $\metric^{(1)},\metric^{(2)}$ lie in the open part of one cone of the fan $\GPFclosed{n}$, then the labeled combinatorial type of the polytopes $\KRW{\metric^{(1)}}$ and $\KRW{\metric^{(2)}}$ are the same.
Hence the Wasserstein fan $\GPFclosed{n}$ refines the type fan of KRW polytopes $\Tn{n}$.
\end{theorem}
\begin{proof}
As noted in \Cref{cts} the labeled combinatorial type of $\KRW{\metric}$ is determined by the collection of $\metric$-admissible graphs. By \Cref{cor:disjointight}, the collection of $\metric$-admissible graphs is determined by two conditions: the first of these (\Cref{cor:disjointight}.(a)) is equivalent to checking which cone of $\GPF{n}$ contains $\metric$ (i.e., which face of the arrangement $\CA{n}$ contains $\metric$), and (\Cref{cor:disjointight}.(b)) is equivalent to ascertaining in which boundary face of $\mconeclosure{n}$ the metric $\metric$ is contained.
\end{proof}

\begin{corollary}
If the metrics $\metric^{(1)},\metric^{(2)}$ are strict and lie in the open part of one cone of the Wasserstein fan $\GPF{n}$, then the labeled combinatorial type of the polytopes $\KRW{\metric^{(1)}}$ and $\KRW{\metric^{(2)}}$ are the same.
\end{corollary}
\begin{proof} For strict metrics condition (b) of \Cref{cor:disjointight} is void, and the remaining condition (a) of \Cref{cor:disjointight} amounts to checking membership in cones of $\GPF{n}$, as in the proof of \Cref{prop:refinement2}.
\end{proof}

We study the relationship between these two fans further.

\begin{proposition}\label[proposition]{prop:refinement}
Let $n\ge 2$ and let  $\CA{n}^4$ denote  the subset of $\CA{n}$ consisting of all hyperplanes of the form $H_{(i,j),(k,l)}$ for distinct $i,j,k,l\in [n]$. 
	\begin{enumerate}
        \item For every $H\in \CA{n}^4$, the set $H\cap \mcone{n}$ is a union of cones of the type fan~$\Tn{n}$. 
		\item For $n\le 5$, the Wasserstein fan $\GPF{n}$ agrees with the type fan $\Tn{n}$.
		\item For $n\ge 6$, the Wasserstein fan $\GPF{n}$ is a strict refinement of the type fan $\Tn{n}$.
	\end{enumerate}
\end{proposition}

\begin{proof}

Let $\metric$ be an $n$-metric, fix distinct $i,j,k,l\in [n]$ and consider the directed graphs $G^+$, $G^-$ on the vertex set $[n]$ with edge sets, respectively, $\{(i,j),(k,l)\}$ and $\{(i,l),(k,j)\}$. By \Cref{thm:gordon_petrov}, $G^+$ is admissible if and only if $\metric_{ij}+\metric_{kl}\le \metric_{ik}+\metric_{jl}$ (i.e., $\rho\in H_{(i,j),(k,l)}^+$), and  $G^-$ is admissible if and only if $\metric_{ij}+\metric_{kl}\ge \metric_{ik}+\metric_{jl}$ (i.e., $\rho\in H_{(i,j),(k,l)}^+$). Moreover, for $\epsilon=+,-$ the graph $G^\epsilon$ is admissible for $\rho$ if and only if it is admissible for every metric whose KRW polytope has the same labeled combinatorial type as $\KRW{\rho}$ (see, e.g., \Cref{cts}). Thus, the open cone of $\Tn{n}$ containing $\rho$ is either contained in $H_{(i,j),(k,l)}$ (in case both $G^+$ and $G^-$ are admissible) or disjoint from it. This proves (1).

For $n\le 6$ the statements (2) and (3) follow from~\Cref{prop:refinement2} and~\Cref{sim_table}. In order to complete the proof of (3), we let $n\ge 6$ and consider the triple of edges $T = \{(1,2),(3,4),(5,6)\}$.
According to \Cref{thm:gordon_petrov}, the directed graph defined by $T$ is admissible if and only if the following two inequalities hold:
\begin{align*}
	&X:=\metric_{12} + \metric_{34} + \metric_{56} \le \metric_{16} + \metric_{45} + \metric_{23}:=Y,\\
	&X:=\metric_{12} + \metric_{34} + \metric_{56} \le \metric_{14} + \metric_{36} + \metric_{25}:=Z.
\end{align*}
The Wasserstein arrangement contains the three hyperplanes $X=Y$, $X=Z$, and $Y=Z$. These subdivide $\mcone{n}$ in $6$ regions,
in two of which the collection $T$ is admissible.
The cell in which $X$ is minimal is the set of metrics where $T$ is admissible, and it is obtained as the union of the two Wasserstein regions in which $T$ is admissible and of their separating wall. 
Choosing a generic point on this separating wall
and moving infinitesimally in either  direction
off this wall yields two points $\mcone{n}$ such that both points are in the open part of one cone of the type fan $\Tn{n}$ but these two points are in different chambers of the Wasserstein arrangement and thus also in different open cells of the Wasserstein fan $\GPF{n}$.
\end{proof}

\begin{remark}
    The hyperplanes $H_{\mathbf{a,b}}$ in $\CA{n}$ were already mentioned as ``exceptional planes" in the proof of Theorem 7 in \cite{GordonPetrov}. The arrangement  $\A\B_n$ of  \cite[Definition 26]{Tran} is an arrangement of hyperplanes in the space $\mathbb R^{n^2-n}$ with coordinates $y_{(i,j)}$ indexed by ordered pairs of distinct elements of $[n]$. The arrangement $\CA{n}$ can be thought of as the restriction of $\A\B_n$ to the symmetric subspace defined by $y_{(i,j)}=y_{(j,i)}$ for all $i\neq j$.
\end{remark}
\noindent 

\noindent 
\section{Regular triangulations and the secondary fan}\label{sec:triangulations}
Following~\cite{DLRS} we recall the definition of a regular subdivision of a set of points in $\R^n$, in particular of a convex polytope.
\begin{definition}\label{defRS}
	Let $\cP=\{a_1,\dots,a_k\}$ be  a set of points in $\R^n$.
	A point $\omega\in \R^k$ defines a \emph{regular subdivision} $\cP(\omega)$ as follows:
	A subset $\{{i_1},\dots,{i_r}\}$ is a facet of $\cP(\omega)$ if there exists a vector $c\in \R^n$ such that
	\begin{equation}\label{eq:subdivision}
	\begin{cases}
	c\cdot a_j  = \omega_j &\text{ if } j\in \{i_1,\dots,i_r\}\text{ and}\\
	c \cdot a_j < \omega_j&\text{ if } j\notin \{i_1,\dots,i_r\}.\\
	\end{cases}
	\end{equation}
	The subdivision $\cP(\omega)$ is called a \emph{regular triangulation} if all facets of the subdivision are simplices.
	This happens if the vector $\omega$ is sufficiently generic.
\end{definition}

Geometrically, a regular subdivision arises by ``lifting'' the points in $\cP$ to $\R^{n+1}$ where we use the coordinates of $\omega$ as heights.
The faces of $\cP(\omega)$ are then the ``lower faces'' of the convex hull of the lifted points, where a lower face is a face that has a normal vector with negative last coordinate.

\begin{definition}
Let $\cR_n^0=\{e_i-e_j\mid i,j\in \left[n\right]\}$. This is  the set of vertices of the (full) \emph{root polytope} of type $A$ together with the origin.
\end{definition}
The next proposition establishes a relationship between regular subdivisions of  $\cR_n^0$ and KRW polytopes.
A similar statement also appeared in the second arXiv version of the preprint~\cite{JS19} and more recently in~\cite[Theorem 3.8]{AF24}.

\begin{proposition}\label[proposition]{RootTriangulations}
	Let $\metric$ be an $n$-metric. 
	\begin{enumerate}
		\item The metric $\metric$ induces a regular subdivision $\cR_n^0(\metric)$ of  $\cR_n^0$ by assigning the height $\metric_{ij}$ to the point $e_i-e_j$ and the height $0$ to the origin.
		\item The subdivision $\cR_n^0(\metric)$ is point-symmetric with respect to the origin and every maximal cell contains the origin.
		\item For any index set $I\subseteq [n]^2$ the following are equivalent:
		\begin{enumerate}
			\item[(i)] The set $\{p_{ij} \mid (i,j) \in I\}$ is the set of all vertices contained in a facet of $\KRW{\metric}$.
			\item[(ii)] The set $\{e_i-e_j \mid (i,j)\in I\}\cup\{0\}$ is a maximal cell of the regular subdivision $\cR_n^0(\metric)$.
		\end{enumerate}
	\end{enumerate}
\end{proposition}

\begin{proof} 
	The symmetry of the subdivision $\cR_n^0(\metric)$ follows immediately from the symmetry of the metric $\metric$.
	Moreover as the height of the origin is chosen to be $0$, the condition~\eqref{eq:subdivision} is satisfied for $a_j=0$ with an equality for any choice $c$.
	Hence, every facet of $\cR_n^0(\metric)$ contains the origin.
	
	Now fix $I\subseteq [n]^2$.
	The collection $\{p_{ij} \mid (i,j) \in I\}$ is the set of all $p_{ij}$  contained in a facet of $\KRW{\metric}$ if and only if 
	\[
	\textrm{there exists a vector }c\in \mathbb R^n
	\textrm{ with }
	\begin{cases}  c\cdot p_{ij}= 1 & \textrm{ if } (i,j)\in I \\
	c\cdot p_{ij} < 1 & \textrm{ if } (i,j)\not\in I.
	\end{cases}
	\]
	This is equivalent to requiring that
	\[
	\begin{cases}  c\cdot (e_i-e_j)= \metric_{ij} & \textrm{ if } (i,j)\in I \\
	c\cdot ( e_i-e_j  ) < \metric_{ij} & \textrm{ if } (i,j)\not\in I.
	\end{cases}
	\]
	So by definition of $\cR_n^0(\metric)$ (see \Cref{rem_dimW}) the latter condition is equivalent to the statement that the set $\{e_i-e_j \mid (i,j)\in I\}\cup\{0\}$ is a maximal cell of the regular subdivision $\cR_n^0(\metric)$.
\end{proof}

A construction by Gel'fand, Kapranov, and Zelevinsky encodes the regular triangulations of a point configuration $\cP$ as points in a polyhedral fan called the {\em secondary fan} of $\cP$.

\begin{theorem}[\cite{GKZ,DLRS}]\label{GKZ}
	Let $\cP=\{a_1,\dots,a_k\}$ be  set of points in $\R^n$.
	\begin{enumerate}
		\item The \emph{secondary fan} $\Sigma(\cP)$ is a polyhedral fan in $\R^k$ such that two points $\omega,\omega'\in \R^k$ lie in the same open (secondary) cone of $~\Sigma(\cP)$ if and only if the regular subdivisions $\cP(\omega)$ and $\cP(\omega')$ are equal.
		\item The fan $\Sigma(\cP)$ is complete, that is the support of $~\Sigma(\cP)$ is $\R^k$, as all points define some regular triangulation.
		\item The subdivision $\cP(\omega)$ is a refinement of ~$\cP(\omega')$ if and only if $\omega'$ is in the closure of the open cone that contains $\omega$.
		\item The regular triangulations of ~$~\cP$ correspond to the top-dimensional open cones of ~$~\Sigma(\cP)$.
	\end{enumerate}
\end{theorem}

\begin{remark}
\Cref{RootTriangulations} is the starting point for the study of KRW polytopes through symmetric subdivisions of $\cR_n^0$ where each maximal cell contains the origin.
In particular, one can enumerate the combinatorial types of KRW polytopes by enumerating the combinatorial types of such triangulations via the secondary fan of $\cR_n^0$ described above in \Cref{GKZ}.
This approach leads to the enumeration of the combinatorial types of KRW polytopes for the case $n=6$ using the \texttt{TOPCOM} software package~\cite{Rambau,Ram23}, see \Cref{sec:compuations}.
\end{remark}

Building on \Cref{RootTriangulations} we can reconstruct the Wasserstein fan as a `symmetrized version' of the secondary fan of the root polytope:

\begin{proposition}
	Let $\cR_n^0$ be as above.
	Then, $\Sigma(\cR_n^0)$ is a fan in $\R^{n(n-1)+1}$, say with coordinates $x_{ij}$ for the vertices $e_i-e_j$ for $i\neq j\in \left[n\right]$ and $x_0$ for the origin.
	Then
	\[
		\Tn{n}=\pi\big(\Sigma(\cR_n^0)\cap 
  Q
  \big)\cap \mcone{n},
	\]
	where $\pi:\R^{n(n-1)+1}\to \R^{\binom{n}{2}}$ is the projection sending $x_{ij}$ to $x_{\{i,j\}}$ for $1\le i<j\le n$, and $Q:=\{x\in \mathbb R^{n(n-1)+1}\mid x_{ij}=x_{ji}\textrm{ for }1\leq i<j\leq n \textrm{ and }x_0=0\}$.

Moreover, the above equality is an equality of fans.
\end{proposition}
\begin{proof}
	Let $\metric$ be an $n$-metric in $\mcone{n}$ and let $F_\metric$ be the unique open face of the type fan $\Tn{n}$ that contains $\metric$.
	Furthermore, let $\metric'$ be another $n$-metric. By construction of the regular subdivisions (\Cref{RootTriangulations}.(1)) and by symmetry of $\metric$ and $\metric'$, surely $\cR_n^0(\metric)$ and $\cR_n^0(\metric')$ lie in $Q$.
 
 Now, by construction of $\Tn{n}$ the polytopes $\KRW{\metric}$ and $\KRW{\metric'}$ determine the same combinatorial structure if and only if $\metric'\in F_\metric$.
	By \Cref{RootTriangulations}, this means that $\cR_n^0(\metric)$ and $\cR_n^0(\metric')$ define the same regular subdivision of $\cR_n^0$ if and only if $\metric' \in F_\metric$. 
	Hence $\metric$ and $\metric'$ lie in the same open cone of $\Sigma(\cR_n^0)\cap Q
 $ if and only if $\metric'\in F_\metric$ which implies the claim.
\end{proof}

\section{Comparison with tight spans and injective hulls}\label{sec:tight_spans}

The aim of this section is to relate the stratification of the metric cone by combinatorial type of KRW polytopes to the one in terms of the combinatorial type of a well-known polyhedral complex associated to metric spaces: the {\em injective hull} defined by Isbell \cite{Isbell} and rediscovered by Dress \cite{Dress}  with the name {\em tight span}. It provides a detailed explanation of the fact that for $n\geq 5$ the classification of metric spaces by (labeled, resp.\ unlabeled) combinatorial type of KRW-polytopes is neither finer nor coarser than the classification by (labeled, resp.\ unlabeled) combinatorial type of tight spans. This incomparability statement remains true when restricting to a ``pruned" version of the tight span, and it implies incomparability of the fan $\GPF{n}$ and the secondary fan of the second hypersimplex.

In order to make precise statements, we start by outlining the construction and some key facts, and refer to \cite{Lan13} for a more detailed account and for a thorough analysis of the structure of injective hulls of general metric spaces.

Given an $n$-metric $\metric$ we consider the (unbounded) polyhedron 
$$P(\metric):=\{x\in \R^n \mid x_i + x_j \geq \metric_{ij} \textrm{ for all } i,j\in [n]\}.$$
For $x,y\in \R^n$ write $x\leq y$ if and only if $x_i\leq y_i$ for all $i=1,\ldots,n$ and define
$$
E(\metric):=\{y\in P(\metric) \mid
\textrm{if }x\in P(\metric)
\textrm{ and }x\leq y,
\textrm{ then }x=y
\},
$$
the set of coordinate-wise minimal elements of $P(\metric)$.

\begin{remark} Equivalently, $E(\metric)$ is the set of bounded faces of the polyhedron $P(\metric)$. This fact is stated in \cite[Lemma 1]{Dress89}. We refer to \cite[Lemma 5.10]{BasicPhyl} for a proof. In particular, the set $E(\metric)$ is naturally a polyhedral complex. 
\end{remark}

\begin{remark}[Injective hulls]  
A metric space $X'$ is called an {\em injective hull} of $X$ if $X'$ is injective (see, e.g., \cite{Lan13} for a definition) and there is an isometric embedding $h$ of $X$ into $X'$ such that every isometric embedding of $X$ into an injective metric space $Y$ factors through  $h$. 
 By work of Isbell \cite[Theorem 2.1]{Isbell} every metric space has an injective hull that is unique up to isometry.   
\end{remark}

Isbell's existence proof in \cite[\S2]{Isbell} constructs a special realization of the injective hull as a polyhedral complex which, as is detailed in \cite{Lan13}, specializes to $E(\rho)$ for instance when $\rho$ is finite. The same construction was independently rediscovered by Dress in \cite{Dress}, under the name {\em tight span}. Both names have been widely used in the literature and we will use them interchangeably in order to refer to the object we define next.

\begin{definition}[{\cite[\S2]{Isbell}}, {\cite[\S 3]{Lan13}}, {\cite[Theorem 3.(v)]{Dress}}]\label[definition]{defIH}
The tight span of $\rho$ is the realization of the injective hull of $\rho$ given by the set $E(\metric)$ endowed with the metric induced by the supremum norm of $\R^n$, i.e., the distance between $x,y\in E(\metric)$ is 
$\sup\{\vert x_i - y_i \vert \mid i\in[n]\}$. 
\end{definition}

\begin{definition}[Distinguished vertices of $E(\metric)$]
\label[definition]{dve}
For every $k\in [n]$, the point $\hk{k}{}$ defined as $\hk{k}{i}=\metric_{ik}$ (i.e., the $k$-th column of the distance matrix associated to $\rho$) is a vertex of $E(\metric)$. It is the image of $k$ under the canonical isometric embedding into $E(\rho)$ (see \cite[Item (1.4)]{Dress} and \Cref{ats}).
\end{definition}

\begin{example}\label{example_IH}
\Cref{fig:tight_span_rho_1,fig:tight_span_rho_2} show a schematic view of the complexes $E(\metric)$ associated to the metrics given in \Cref{example_decomp}. The numbers along the edges give the distance between the edge's endpoints with respect to the metric of the injective hull. The solid dots are the image of the canonical embedding of the metric space, and are labeled by the corresponding element of $[n]$.
\end{example}

\begin{remark}[Pruned Tight Span]
\label[remark]{pts}
We will have use for a subcomplex $E^\sharp(\metric)$ of $E(\metric)$ that we suggest to call ``pruned Tight Span" of $\metric$ (see \Cref{AppDef}). In \Cref{AppAntennas} we prove that it can be characterized as being obtained by removing from $E(\metric)$ all vertices of the form $\hk{k}{}$ that are contained in a single one-dimensional cell of $E(\metric)$.

The complexes $E^\sharp(\metric_1)$, $E^\sharp(\metric_2)$ for the metrics in \Cref{example_decomp} are obtained from the complexes depicted in \Cref{fig:tight_span_rho_1,fig:tight_span_rho_2} by removing all solid-color vertices and the edges incident to them. 

A note of caution: in the computational geometry package \texttt{polymake}~\cite{polymake}, the tight span $E(\rho)$ is computed by the method \texttt{metric\char`_extended\char`_tight\char`_span} and our ``pruned tight span'' $E^\sharp(\rho)$ by the method \texttt{metric\char`_tight\char`_span}.
\end{remark}

\begin{remark}[Combinatorial types of tight spans] \label[remark]{ctts}
We will speak of the labeled combinatorial type of tight spans according to the terminology of \Cref{cts} with respect to the following labeling of vertices. Some vertices of $E(\metric)$ have a natural labeling by elements of $[n]$ discussed in \Cref{dve}. When passing to $E^\sharp(\metric)$ via the characterization given in \Cref{pts}, for each $k$ such that $\hk{k}{}$ is pruned, there is a unique vertex of $E^\sharp(\metric)$, called $\vk{k}{}$ in \Cref{AppAntennas}, that is incident to the pruned edge. Note that if $\metric$ is a metric, for different $k$ all $\hk{k}{}$ are distinct, but some $\vk{k}{}$ may concide. Thus we can label vertices of $E^\sharp(\metric)$ either by the label they inherit from $E(\metric)$ or by the set of labels of adjacent pruned vertices in~$E(\metric)$.

\end{remark}

Now recall the {\em second hypersimplex} 
  $~\Delta_{2,n}:=\operatorname{conv}\{e_i+e_j \mid i,j\in[n],\, i\neq j\} \subseteq \R^{n}$.
Following the setup of \cite{HJ08}, every  $\metric\in \R^{{n\choose 2}}$ defines a lift of $\Delta_{2,n}$:
$$
L(\metric):=\operatorname{conv}\{(e_i+e_j,-\metric_{ij}) \mid i,j\in [n], i\neq j\} + \R_{\geq 0} e_{n+1}\subseteq \R^{n+1}.
$$
The projection of the bounded faces of $L(\metric)$ onto $\R^n$ defines the regular subdivision  $\Delta_{2,n}(-\metric)$. The stratification of $\R^{{n \choose 2}}$ according to combinatorial type of this regular subdivision gives rise to the secondary fan $\Sigma(\Delta_{2,n})$, a well-studied object in its own right (see, e.g., \cite{DST,SturmfelsYu}).

Restricting to the metric, resp.\ pseudometric cone, we obtain the set of strata
$$
\stratsharp{n}=\{\sigma\cap \mcone{n} \mid  \sigma \in \Sigma(\Delta_{2,n})\}
\quad\quad
\pstratsharp{n}=\{\sigma\cap \mconeclosure{n} \mid  \sigma \in \Sigma(\Delta_{2,n})\}.
$$
Note that, although every stratum is a cone, these are not necessarily  polyhedral fans because $\mcone{n}$ and $\mconeclosure{n}$ are cones and not collections of faces, see \Cref{remcone}.\footnote{For instance, it could happen that some ray $\gamma$ of the cone  $\mconeclosure{n}$ is contained in the interior of some  $\sigma\in\Sigma(\Delta_{2,n})$. Then, $\pstratsharp{n}$ will contain the cone $\sigma\cap\mconeclosure{n}$ but not its face $\gamma$ (this is because $\mconeclosure{n}$ is not a fan but a single cone).} 

\begin{remark}\label[remark]{hjrem} A well-known fact going back to work of Herrmann and Joswig \cite{HJ08} states that for any two pseudometrics $\metric_1,\metric_2\in\mconeclosure{n}$ the complexes $E^\sharp(\metric_1)$ and $E^\sharp (\metric_2)$ have the same labeled combinatorial type if and only if they are in the same stratum of $\pstratsharp{n}$.
\end{remark}

\begin{definition}[Metric fan] \label{def:metric_fan}
The {\em metric fan} as defined in \cite{CJK24} is the fan 
$$
\MF{n}=\{\mu \cap \sigma \mid \mu\textrm{ a face of } \mconeclosure{n},\sigma\in \Sigma(\Delta_{2,n})\}.
$$ 
\end{definition}

\begin{remark} 
\label[remark]{interiorstory}
The fan $\MF{n}$ is polyhedral, and its rays are described in detail in \cite[Proposition 25]{CJK24}.  
Note that since all non-trivial faces of $\mconeclosure{n}$ are contained in the boundary of $\mcone{n}$, the stratification given by $\MF{n}$ and $\stratsharp{n}$ coincide in the topological interior $\mconeinterior{n}$ of $\mcone{n}$. Precisely, we have
$
\{\sigma \cap \mconeinterior{n} \mid \sigma \in \MF{n}\} = \{\tau  \cap \mconeinterior{n}\mid \tau\in \stratsharp{n}\}.
$
Note also that the labeled combinatorial type of $E(\rho)$ is determined by the cone of $\MF{n}$ containing $\rho$ (see \Cref{EMF}).
\end{remark}

The main theorem of this section is the following incomparability result between the Wasserstein fan and the secondary fan of the second hypersimplex.

\begin{theorem} \label{thm:coarsefine}
For $n\geq 5$, 
the fans $\GPF{n}$ and $\Tn{n}$ are neither coarsenings nor refinements of the subdivision $\stratsharp{n}$ of the metric cone induced by the secondary fan of $\Delta_{2,n}$. 
Similarly, for $n\geq 5$ the fan $\GPFclosed{n}$ is neither a coarsening nor a refinement of the metric fan $\MF{n}$. 
\end{theorem}
\begin{remark}[The case $n=4$]
    The fan $\GPF{4}=\Tn{4}$ is a refinement of the subdivision $\stratsharp{4}$ of the metric cone induced by the  secondary fan of $\Delta_{2,4}$. See \Cref{fig:tight_span_4} and \Cref{sec_n4}.
\end{remark}

\Cref{thm:coarsefine} is a consequence of the following statement.

\begin{proposition}\label[proposition]{newthm}
For every $n\geq 4$ there is a pair of strict metrics whose injective hulls have the same labeled combinatorial type, but whose KRW-polytopes have different unlabeled combinatorial type. For $n\geq 5$ there is a pair of strict metrics for which the KRW-polytopes lie in the same open cell of the Wasserstein fan $\GPF{n}$ (and hence they have the same labeled combinatorial type), but whose injective hulls have different unlabeled combinatorial type.
\end{proposition}

\begin{proof}[Proof of \Cref{thm:coarsefine} under assumption of \Cref{newthm}] 
All metrics cited in \Cref{newthm} are strict, hence they lie in the topological interior of the metric cone and so \Cref{interiorstory} applies. This means that the statement of \Cref{newthm} remains true substituting the injective hulls with the pruned tight span of \Cref{pts}. Now the first claim holds by the result of Herrmann and Joswig cited in \Cref{hjrem}, and the second claim follows from \Cref{interiorstory}.
\end{proof}

We need two general facts for the proof of \Cref{newthm} which we state as lemmas as they may be of independent interest. Recall that if $P$ and $Q$ are two polytopes in $\mathbb R^n$ both containing the origin and for which the affine hull of $P$ intersects the affine hull of $Q$ only in the origin, then the {\em free sum} of $P$ and $Q$ is $P\oplus Q:=\conv(P\cup Q)$, see for instance~\cite{Kal88}.

\begin{lemma}\label[lemma]{lem:directsum} Let $\metric_1$, $\metric_2$ be metrics on the sets 
$[n]$ and $\{n,\ldots,n+m\}$, respectively. 
Consider the function
$
\metric_1\oplus \metric_2: [n+m]^2 \to \mathbb R
$
defined via 
$$
(\metric_1\oplus \metric_2)(i,j):=\begin{cases}
(\metric_1)_{ij} &
\text{if}\quad i,j\leq n, \\
(\metric_1)_{in}+(\metric_2)_{nj} & \text{if}\quad i\leq n < j, \\
(\metric_2)_{ij} & \text{if}\quad n\leq i,j.
\end{cases}
$$
Then $\metric_1\oplus \metric_2$ is a metric, and we have a direct (free) sum decomposition $$
\KRW{\metric_1\oplus \metric_2} = \KRW{\metric_1}\oplus \KRW{\metric_2}.
$$
\end{lemma}
\begin{proof}
A straightforward check of the definitions shows that $\metric_1\oplus\metric_2$ is a metric. 

The set $X_1$ of pairs of distinct elements from $[n]$ is disjoint from the set $X_2$ of pairs of distinct elements of $\{n,\ldots,n+m\}$. The KRW polytopes of $\metric_1$ and $\metric_2$ both contain the origin and lie in the subspaces $\mathbb R^{X_1}$ and $\mathbb R^{X_2}$ of $\mathbb R^{\binom{n+m}{2}}$, respectively.
Since the $X_i$ are disjoint, $\mathbb R^{X_1}\cap \mathbb R^{X_2} = \{0\}$. Then, by definition 
$\KRW{\metric_1}\oplus \KRW{\metric_2}=\conv(\KRW{\metric_1}\cup \KRW{\metric_2}).$
For brevity call $K$ the latter polytope, and let $\hat{\metric}:= \metric_1\oplus \metric_2$. 
We have to prove that $K=\KRW{\hat{\metric}}$.

By definition of $\hat{\metric}$, both $\KRW{\metric_1}$ and $\KRW{\metric_2}$ are contained in $\KRW{\hat{\metric}}$. Therefore, $K$  is contained in $\KRW{\hat{\metric}}$. The only potential vertices of $\KRW{\hat{\metric}}$ not contained in $K$ are given by vectors $\hat{p_{ij}}:=\frac{e_i-e_j}{\hat{\metric}_{ij}}$ with $i\leq n < j$.
But by definition of $\hat{\metric}$ we have 
$$
\frac{e_i-e_j}{\hat{\metric}_{ij}} =
\lambda\frac{e_i-e_n}{{\metric_1}_{in}}
+
\mu\frac{e_n-e_j}{{\metric_2}_{nj}}
$$
with $\lambda=\frac{({\metric_1})_{in}}{(\metric_1)_{in}+(\metric_2)_{nj}}$ and $\mu=\frac{({\metric_2})_{nj}}{(\metric_1)_{in}+(\metric_2)_{nj}}$ being positive numbers with $\lambda+\mu=1$. Therefore $\hat{p_{ij}}$ is in $K$ for all $i\leq n<j$, thus $\KRW{\hat{\metric}}=K$.
\end{proof}

\begin{lemma}\label[lemma]{10zeilen} If $\rho$ is a metric on the set $[n]$ and $t$ a metric on the set $\{n,n+1\}$, then  the tight span of $\rho\oplus t$ is obtained from the tight span of $\rho$ by gluing a new edge at the vertex given by $\hk{n}{}$ (defined in \Cref{dve}).
\end{lemma}
 
\begin{proof} The statement follows from \cite[Theorem 1.1]{Miesch} (see also \cite[Lemma 2.2]{HJ08}), but we give a self-contained proof here. Write  $\lambda:=t_{n\, n+1}=(\rho\oplus t)_{n\, n+1}$, and consider the pseudometrics $\sigma$ and $\delta$ on $[n+1]$ where $\sigma$ is given by
$\sigma_{ij}:=\rho_{ij} $ for $j\neq n+1$, $\sigma_{i\, n+1}:=\rho_{in}$, and  $\delta:= \lambda\cdot\delta_{[n],\{n+1\}}$ is a (non-zero) multiple of the split pseudometric defined in \Cref{defsplits}. As in \cite{BandeltDress}, we can apply the construction of the tight span also to pseudometrics, obtaining
\begin{align*}
E(\sigma) &=\{x\in\mathbb R^{[n+1]} \mid (x_1,\ldots,x_n)\in E(\rho), x_{n+1}=x_n\},\\ 
E(\delta) &=\{x\in\mathbb R^{[n+1]} \mid x_1=x_2=\ldots=x_n,(x_n,x_{n+1})\in E(t)\}\\ 
&= \{(\alpha,\ldots,\alpha,\lambda-\alpha)\in\mathbb R^{n+1} \mid 0\leq \alpha\leq \lambda\},
\end{align*}
which are  polyhedral complexes isomorphic to $E(\rho)$ and $E(t)$, respectively.
Let $$\hhk{n}{} :=(\rho_{in},\ldots,\rho_{nn},\rho_{nn})$$ denote the image of $\hk{n}{}\in E(\rho)$ in $E(\sigma)$, and 
let\footnote{Sum of subsets of Euclidean space is intended elementwise, i.e., as Minkowski sum.} 
\def\hE{\hat{E}}
$$\hE(\sigma) := E(\sigma) + (0,0,\ldots,0,\lambda),\quad
\hE(\delta):= E(\delta) + \hhk{n}{}.$$

Since $\hE(\sigma) \cap \hE(\delta)=\{\hhk{n}{}\}$, it will be enough to show that
\begin{equation}\label{almostlast}
 E(\rho\oplus t) = \hE(\sigma) \cup \hE(\delta),
\end{equation}
since $\hE(\delta)$ (isometric to the unit interval) is joined to $\hE(\sigma)$ (isomorphic to $E(\rho)$) at the unique common point $\hhk{n}{} + (0,\ldots,0,\lambda)$, the image of $\hk{n}{}$ in $E(\rho)$.

The inclusion $E(\rho\oplus t) \supseteq \hE(\sigma) \cup \hE(\delta)$ can either be checked directly with \Cref{defIH} or deduced from the fact that $\hE(\sigma)$, resp.\ $\hE(\delta)$ are the images of $E(\rho)$, resp.\ $E(t)$, under the canonical inclusion into $E(\rho\oplus t)$ described in \cite[Item (1.11)]{Dress}\footnote{In the language of \cite[Item (1.11)]{Dress}, we are considering the inclusions of tight spans $\tau_{[n]}$, resp.\ $\tau_{n+1}$, induced by the extensions of $\rho$ to $\rho\oplus t$, resp.\ of $t$ to $\rho\oplus t$.}.

For the other inclusion start by noticing that, obviously, $\rho\oplus t = \sigma + \delta$. The isolation index of $n+1$ with respect to $\rho\oplus t$ (see \Cref{def:ii} below for the definition) is   $\alpha_{[n],\{n+1\}}^{\rho\oplus t} = \lambda$. Therefore \cite[Theorem 7]{BandeltDress} applies and shows $P(\rho\oplus t) = P(\sigma)+P(\delta)$. As is explained, e.g., in \cite[p. 89]{BandeltDress}\footnote{In \cite{BandeltDress}  our $E(\rho)$ is denoted by $T(\rho)$.}, this implies $E(\rho\oplus t)\subseteq E(\sigma) + E(\delta)$.  Now for the inclusion $E(\rho\oplus t) \subseteq \hE(\sigma) \cup \hE(\delta)$ it is enough to show that 
\begin{equation}\label{atlast}
E(\rho\oplus t) \cap (E(\sigma) + E(\delta)) \subseteq 
\hE(\sigma) \cup \hE(\delta).
\end{equation}
Take $p\in (E(\sigma) + E(\delta))$. In our setup, without loss of generality we can write
$$
p= (y_1+\alpha,\ldots, y_n+\alpha, y_n + (\lambda - \alpha)),\quad \textrm{ with }
y\in E(\rho), \;\; 0\leq \alpha\leq \lambda.
$$
We surely have that $p\in P(\rho\oplus t)$, since:

\begin{itemize}
\item[(i)] for $i,j\in [n]$: $p_i + p_j = y_i+ y_j + 2\alpha \geq \rho_{ij} + 2\alpha \geq (\rho\oplus t)_{ij} +2\alpha \geq (\rho\oplus t)_{ij}$;
\item[(ii)] for $i\in [n]$: $p_i+p_{n+1} = y_i + \alpha + y_n + \lambda - \alpha = y_i + y_n +\lambda \geq \rho_{in} + t_{nn+1} = (\rho\oplus t)_{i,n+1}$.
\end{itemize}

If $\alpha = 0$, then $y\in E(\rho)$  implies that $(p_1,\ldots,p_n)$ is already coordinate-wise minimal satisfying  (i). In particular, there is $i_0\leq n$ with $y_{i_0}+y_n = \rho_{{i_0}n}$. This implies that $p_{n+1}$ cannot be made smaller either, because $p_{{i_0}}+p_{n+1}=y_{i_0} + y_n +\lambda = \rho_{{i_0}n}+t_{n n+1} = (\rho + t)_{{i_0} n+1} $ is an equality. In this case, then, $p\in E(\rho\oplus t)$ implies $p\in \hE(\sigma)$.

If $\alpha > 0$, then $y\in E(\rho)$  implies that none of the inequalities in (i) is an equality, indeed $p_i+p_j\geq \rho_{ij}+2\alpha$ for all $i,j\in [n]$. 
Now if $y_j + y_n > \rho_{in}$ for some $j$, define $\tilde{p}$ by $\tilde{p}_j:=p_j-{\alpha}$ and $\tilde{p}_i:=p_i$ for $i\neq j$. Then $\widetilde{p}\in P(\rho\oplus t)$, yet $\tilde{p}$ is smaller than $p$ in the $j$-th coordinate. Thus $p\in E(\rho\oplus t)$ implies that $y_j + y_n = \rho_{in}$ for all $j\in [n]$, so $p\in \hE(\delta)$. This concludes the proof of \Cref{atlast}, hence of \Cref{almostlast} and, with it, of the claim.
\end{proof}

\begin{definition}\label[definition]{def:tk}
For given $k>0$ let $\tk{n}{k}$ denote the $k$-metric defined on the set $\{n,\ldots,n+k\}$ by setting $\tk{n}{k}_{i,j} := \vert j-i\vert$. This is a graph metric realized on the full set of vertices of a path with $k+1$ vertices.
\end{definition}

\begin{remark}\label[remark]{appendpath} 
Let $\rho$ be any $n$-metric. Then $$\rho\oplus \tk{n}{k}=
(\cdots (\rho\oplus \tk{n}{1}) \oplus \tk{n+1}{1})\cdots )\oplus \tk{n+k-1}{1}.
$$
By repeated application of \Cref{10zeilen} we obtain that the tight span of $\rho\oplus \tk{n}{k}$ is obtained from the tight span of $\rho$ by appending a $(k+1)$-path at the vertex $h_n$ of $E(\rho)$.
\end{remark}

\begin{proof}[Proof of \Cref{newthm}]
\Cref{example_decomp}  exhibits two $5$-metrics that lie in the same open cone of~
$\CA{5}$ (thus they have combinatorially isomorphic KRW polytopes), but lie in different cones of the secondary fan of $\Delta_{2,5}$ (and thus have different tight spans). 
Moreover, \Cref{ex:4pt} shows two metrics on $4$ points that are in the same cone of the secondary fan of $\Delta_{2,4}$ but have non-isomorphic KRW polytopes. 
Starting from those, we are going to construct examples for any (finite) number of points.

Fix $k>0$ and recall the metrics  $\tk{n}{k}$ defined in \Cref{def:tk}. Since these are graph metrics realized on the full set of vertices of a tree with $k+1$ vertices, $\KRW{\tk{n}{k}}$ is unimodularly equivalent to the cross-polytope of dimension $k$. Moreover, the tight span of $\tk{n}{k}$ is combinatorially equivalent to the tree realizing it, i.e., a path with $k$ edges, with vertices labeled $n,\ldots,n+k$ (see, e.g., \cite[Theorem 8]{Dress}).

Now given an $n$-metric $\metric$ consider the $(n+k)$-metric $\metric^{(k)}:=\metric\oplus \tk{n}{k}$. 
By 
\Cref{appendpath}, the 
tight span $E(\metric^{(k)})$ is obtained from that of $\metric$ by appending a $k$-path at the vertex labeled $n$.
Moreover, by \Cref{lem:directsum} there is a direct (free) sum decomposition $\KRW{\metric^{(k)}}=\KRW{\metric}\oplus \KRW{\tk{n}{k}}$, implying that the poset of non-maximal faces of $\KRW{\metric^{(k)}}$ is the Cartesian product of the posets of non-maximal faces of $\KRW{\metric}$ and of $\KRW{\tk{n}{k}}$ (see, e.g., \cite{Kal88}).

Let $\metric_1$, $\metric_2$ be the two $5$-metrics of \Cref{example_decomp} which lie in the same cone of the Wasserstein fan $\GPF{5}$.
The isomorphism between  the polytopes $\KRW{\metric_1}$ and $\KRW{\metric_2}$ induces an isomorphism between $\KRW{\metric_1^{(k)}}$ and $\KRW{\metric_2^{(k)}}$. On the other hand, any isomorphism between the tight spans of $\metric_1^{(k)}$ and $\metric_2^{(k)}$ must map the $t^{(n,k)}$-part to itself (since the tight spans of $\metric_1$ and $\metric_2$ do not contain any path of length more than $1$ consisting of vertices with degree $\leq 2$, see \Cref{example_decomp}), and thus would induce an isomorphism between the tight spans of $\metric_1$ and $\metric_2$. The latter cannot exist, and so we show that $\GPF{n}$ is not a refinement of the secondary fan of $\Delta_{2,n}$ for all $n\geq 5$.

Let $\metric_1$, $\metric_2$ be the two $4$-metrics of \Cref{ex:4pt}. The combinatorial isomorphism between the tight span of $\metric_1$ and $\metric_2$ extends to an isomorphism of the tight spans of $\metric_1^{(k)}$ and $\metric_2^{(k)}$. On the other hand the fact that $\KRW{\metric_2}$ is simplicial implies simpliciality of $\KRW{\metric_2^{(k)}}$, and since $\KRW{\metric_1}$ is not simplicial, neither is $\KRW{\metric_1^{(k)}}$. In particular, $\KRW{\metric_1^{(k)}}$ is not isomorphic to $\KRW{\metric_2^{(k)}}$ for any $k>0$. We have proved that for all $n\geq 5$ the fan $\GPF{n}$ is not a coarsening of the secondary fan of $\Delta_{2,n}$.
\end{proof}

\begin{example}\label[example]{ex:4pt}
Consider the following two $4$-metrics.
$$
\metric_1:=\begin{pmatrix}
0 & 3 & 3 & 4 \\
3 & 0 & 4 & 3 \\
3 & 4 & 0 & 3 \\
4 & 3 & 3 & 0
\end{pmatrix}
\quad\quad
\metric_2:=\begin{pmatrix}
0 & 4 & 3 & 5 \\
4 & 0 & 5 & 3 \\
3 & 5 & 0 & 4 \\
5 & 3 & 4 & 0
\end{pmatrix}
$$
These metrics have combinatorially isomorphic tight spans, but non-isomorphic KRW polytopes, see \Cref{fig:4pt}.
The polytopes are not combinatorially isomorphic as $\KRW{\metric_1}$ has two square faces whereas $\KRW{\metric_2}$ is simplicial.
\end{example}

\begin{figure}[ht]
\centering
     \begin{subfigure}[t]{0.3\textwidth}
         \centering
			\resizebox{\textwidth}{!}{\begin{tikzpicture}
	[x={(-0.964234cm, -0.208442cm)},
	y={(-0.265051cm, 0.758364cm)},
	z={(0.000029cm, 0.617605cm)},
	scale=10.000000,
	back/.style={loosely dotted, thin},
	edge/.style={color=blue!95!black, thick},
	facet/.style={fill=blue!95!black,fill opacity=0.250000},
	vertex/.style={inner sep=1pt,circle,draw=green!25!black,fill=green!75!black,thick}]
%
%
\coordinate (0.19612, 0.13074, 0.66667) at (0.19612, 0.13074, 0.66667);
\coordinate (-0.22880, 0.13074, 0.66667) at (-0.22880, 0.13074, 0.66667);
\coordinate (0.03269, -0.26149, 0.66667) at (0.03269, -0.26149, 0.66667);
\coordinate (0.29417, 0.47939, 0.44444) at (0.29417, 0.47939, 0.44444);
\coordinate (-0.09806, 0.35955, 0.33333) at (-0.09806, 0.35955, 0.33333);
\coordinate (0.55566, 0.08716, 0.44444) at (0.55566, 0.08716, 0.44444);
\coordinate (-0.35955, 0.04358, 0.22222) at (-0.35955, 0.04358, 0.22222);
\coordinate (0.29417, -0.22880, 0.33333) at (0.29417, -0.22880, 0.33333);
\coordinate (-0.09806, -0.34865, 0.22222) at (-0.09806, -0.34865, 0.22222);
\coordinate (0.16343, 0.39223, 0.00000) at (0.16343, 0.39223, 0.00000);
\coordinate (0.42492, 0.00000, 0.00000) at (0.42492, 0.00000, 0.00000);
\coordinate (0.00000, 0.00000, 0.00000) at (0.00000, 0.00000, 0.00000);
\draw[edge,back] (0.19612, 0.13074, 0.66667) -- (-0.22880, 0.13074, 0.66667);
\draw[edge,back] (0.19612, 0.13074, 0.66667) -- (0.03269, -0.26149, 0.66667);
\draw[edge,back] (0.19612, 0.13074, 0.66667) -- (0.29417, 0.47939, 0.44444);
\draw[edge,back] (0.19612, 0.13074, 0.66667) -- (0.55566, 0.08716, 0.44444);
\draw[edge,back] (-0.22880, 0.13074, 0.66667) -- (0.03269, -0.26149, 0.66667);
\draw[edge,back] (0.03269, -0.26149, 0.66667) -- (0.55566, 0.08716, 0.44444);
\draw[edge,back] (0.03269, -0.26149, 0.66667) -- (0.29417, -0.22880, 0.33333);
\draw[edge,back] (0.03269, -0.26149, 0.66667) -- (-0.09806, -0.34865, 0.22222);
\draw[edge,back] (0.55566, 0.08716, 0.44444) -- (0.29417, -0.22880, 0.33333);
\draw[edge,back] (0.29417, -0.22880, 0.33333) -- (-0.09806, -0.34865, 0.22222);
\draw[edge,back] (0.29417, -0.22880, 0.33333) -- (0.42492, 0.00000, 0.00000);
\node[vertex] at (0.29417, -0.22880, 0.33333)     {};
\node[vertex] at (0.19612, 0.13074, 0.66667)     {};
\node[vertex] at (0.03269, -0.26149, 0.66667)     {};
\fill[facet] (0.42492, 0.00000, 0.00000) -- (0.55566, 0.08716, 0.44444) -- (0.29417, 0.47939, 0.44444) -- (0.16343, 0.39223, 0.00000) -- cycle {};
\fill[facet] (-0.09806, 0.35955, 0.33333) -- (-0.22880, 0.13074, 0.66667) -- (0.29417, 0.47939, 0.44444) -- cycle {};
\fill[facet] (-0.35955, 0.04358, 0.22222) -- (-0.22880, 0.13074, 0.66667) -- (-0.09806, 0.35955, 0.33333) -- cycle {};
\fill[facet] (0.16343, 0.39223, 0.00000) -- (-0.09806, 0.35955, 0.33333) -- (-0.35955, 0.04358, 0.22222) -- cycle {};
\fill[facet] (0.16343, 0.39223, 0.00000) -- (0.29417, 0.47939, 0.44444) -- (-0.09806, 0.35955, 0.33333) -- cycle {};
\fill[facet] (0.00000, 0.00000, 0.00000) -- (-0.09806, -0.34865, 0.22222) -- (0.42492, 0.00000, 0.00000) -- cycle {};
\fill[facet] (0.00000, 0.00000, 0.00000) -- (0.16343, 0.39223, 0.00000) -- (0.42492, 0.00000, 0.00000) -- cycle {};
\fill[facet] (0.00000, 0.00000, 0.00000) -- (-0.35955, 0.04358, 0.22222) -- (0.16343, 0.39223, 0.00000) -- cycle {};
\fill[facet] (0.00000, 0.00000, 0.00000) -- (-0.35955, 0.04358, 0.22222) -- (-0.09806, -0.34865, 0.22222) -- cycle {};
\draw[edge] (-0.22880, 0.13074, 0.66667) -- (0.29417, 0.47939, 0.44444);
\draw[edge] (-0.22880, 0.13074, 0.66667) -- (-0.09806, 0.35955, 0.33333);
\draw[edge] (-0.22880, 0.13074, 0.66667) -- (-0.35955, 0.04358, 0.22222);
\draw[edge] (0.29417, 0.47939, 0.44444) -- (-0.09806, 0.35955, 0.33333);
\draw[edge] (0.29417, 0.47939, 0.44444) -- (0.55566, 0.08716, 0.44444);
\draw[edge] (0.29417, 0.47939, 0.44444) -- (0.16343, 0.39223, 0.00000);
\draw[edge] (-0.09806, 0.35955, 0.33333) -- (-0.35955, 0.04358, 0.22222);
\draw[edge] (-0.09806, 0.35955, 0.33333) -- (0.16343, 0.39223, 0.00000);
\draw[edge] (0.55566, 0.08716, 0.44444) -- (0.42492, 0.00000, 0.00000);
\draw[edge] (-0.35955, 0.04358, 0.22222) -- (-0.09806, -0.34865, 0.22222);
\draw[edge] (-0.35955, 0.04358, 0.22222) -- (0.16343, 0.39223, 0.00000);
\draw[edge] (-0.35955, 0.04358, 0.22222) -- (0.00000, 0.00000, 0.00000);
\draw[edge] (-0.09806, -0.34865, 0.22222) -- (0.42492, 0.00000, 0.00000);
\draw[edge] (-0.09806, -0.34865, 0.22222) -- (0.00000, 0.00000, 0.00000);
\draw[edge] (0.16343, 0.39223, 0.00000) -- (0.42492, 0.00000, 0.00000);
\draw[edge] (0.16343, 0.39223, 0.00000) -- (0.00000, 0.00000, 0.00000);
\draw[edge] (0.42492, 0.00000, 0.00000) -- (0.00000, 0.00000, 0.00000);
\node[vertex] at (-0.22880, 0.13074, 0.66667)     {};
\node[vertex] at (0.29417, 0.47939, 0.44444)     {};
\node[vertex] at (-0.09806, 0.35955, 0.33333)     {};
\node[vertex] at (0.55566, 0.08716, 0.44444)     {};
\node[vertex] at (-0.35955, 0.04358, 0.22222)     {};
\node[vertex] at (-0.09806, -0.34865, 0.22222)     {};
\node[vertex] at (0.16343, 0.39223, 0.00000)     {};
\node[vertex] at (0.42492, 0.00000, 0.00000)     {};
\node[vertex] at (0.00000, 0.00000, 0.00000)     {};
\end{tikzpicture}}
         \caption{$\KRW{\metric_1}$}
         \label{fig:4pta}
     \end{subfigure}
     \begin{subfigure}[t]{0.3\textwidth}
         \centering
         \includegraphics[width=.7\textwidth]{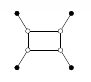}
         \caption{tight span of both metrics}
         \label{fig:4ptb}
     \end{subfigure}
     \begin{subfigure}[t]{0.3\textwidth}
         \centering
         \resizebox{\textwidth}{!}{\begin{tikzpicture}%
	[x={(0.957905cm, 0.217955cm)},
	y={(-0.287086cm, 0.727313cm)},
	z={(0.000031cm, 0.650777cm)},
	scale=10.000000,
	back/.style={loosely dotted, thin},
	edge/.style={color=blue!95!black, thick},
	facet/.style={fill=blue!95!black,fill opacity=0.250000},
	vertex/.style={inner sep=1pt,circle,draw=green!25!black,fill=green!75!black,thick}]
%
%
\coordinate (0.06489, 0.17342, 0.53452) at (0.06489, 0.17342, 0.53452);
\coordinate (-0.34607, 0.17342, 0.53452) at (-0.34607, 0.17342, 0.53452);
\coordinate (-0.04867, -0.13007, 0.53452) at (-0.04867, -0.13007, 0.53452);
\coordinate (0.19467, 0.52027, 0.35635) at (0.19467, 0.52027, 0.35635);
\coordinate (-0.12978, 0.31216, 0.21381) at (-0.12978, 0.31216, 0.21381);
\coordinate (0.35689, 0.13007, 0.40089) at (0.35689, 0.13007, 0.40089);
\coordinate (-0.29200, 0.04336, 0.13363) at (-0.29200, 0.04336, 0.13363);
\coordinate (0.19467, -0.13874, 0.32071) at (0.19467, -0.13874, 0.32071);
\coordinate (-0.12978, -0.34684, 0.17817) at (-0.12978, -0.34684, 0.17817);
\coordinate (0.11355, 0.30349, 0.00000) at (0.11355, 0.30349, 0.00000);
\coordinate (0.41096, 0.00000, 0.00000) at (0.41096, 0.00000, 0.00000);
\coordinate (0.00000, 0.00000, 0.00000) at (0.00000, 0.00000, 0.00000);
\draw[edge,back] (-0.34607, 0.17342, 0.53452) -- (-0.12978, 0.31216, 0.21381);
\draw[edge,back] (0.19467, 0.52027, 0.35635) -- (-0.12978, 0.31216, 0.21381);
\draw[edge,back] (0.19467, 0.52027, 0.35635) -- (0.11355, 0.30349, 0.00000);
\draw[edge,back] (0.19467, 0.52027, 0.35635) -- (0.41096, 0.00000, 0.00000);
\draw[edge,back] (-0.12978, 0.31216, 0.21381) -- (-0.29200, 0.04336, 0.13363);
\draw[edge,back] (-0.12978, 0.31216, 0.21381) -- (0.11355, 0.30349, 0.00000);
\draw[edge,back] (-0.29200, 0.04336, 0.13363) -- (0.11355, 0.30349, 0.00000);
\draw[edge,back] (-0.29200, 0.04336, 0.13363) -- (0.00000, 0.00000, 0.00000);
\draw[edge,back] (-0.12978, -0.34684, 0.17817) -- (0.00000, 0.00000, 0.00000);
\draw[edge,back] (0.11355, 0.30349, 0.00000) -- (0.41096, 0.00000, 0.00000);
\draw[edge,back] (0.11355, 0.30349, 0.00000) -- (0.00000, 0.00000, 0.00000);
\draw[edge,back] (0.41096, 0.00000, 0.00000) -- (0.00000, 0.00000, 0.00000);
\node[vertex] at (-0.12978, 0.31216, 0.21381)     {};
\node[vertex] at (0.11355, 0.30349, 0.00000)     {};
\node[vertex] at (0.00000, 0.00000, 0.00000)     {};
\fill[facet] (0.19467, 0.52027, 0.35635) -- (0.06489, 0.17342, 0.53452) -- (-0.34607, 0.17342, 0.53452) -- cycle {};
\fill[facet] (0.35689, 0.13007, 0.40089) -- (0.06489, 0.17342, 0.53452) -- (0.19467, 0.52027, 0.35635) -- cycle {};
\fill[facet] (0.35689, 0.13007, 0.40089) -- (0.06489, 0.17342, 0.53452) -- (-0.04867, -0.13007, 0.53452) -- cycle {};
\fill[facet] (-0.04867, -0.13007, 0.53452) -- (0.06489, 0.17342, 0.53452) -- (-0.34607, 0.17342, 0.53452) -- cycle {};
\fill[facet] (0.19467, -0.13874, 0.32071) -- (-0.04867, -0.13007, 0.53452) -- (0.35689, 0.13007, 0.40089) -- cycle {};
\fill[facet] (-0.12978, -0.34684, 0.17817) -- (-0.34607, 0.17342, 0.53452) -- (-0.29200, 0.04336, 0.13363) -- cycle {};
\fill[facet] (-0.12978, -0.34684, 0.17817) -- (-0.04867, -0.13007, 0.53452) -- (0.19467, -0.13874, 0.32071) -- cycle {};
\fill[facet] (-0.12978, -0.34684, 0.17817) -- (-0.34607, 0.17342, 0.53452) -- (-0.04867, -0.13007, 0.53452) -- cycle {};
\fill[facet] (0.41096, 0.00000, 0.00000) -- (0.19467, -0.13874, 0.32071) -- (-0.12978, -0.34684, 0.17817) -- cycle {};
\fill[facet] (0.41096, 0.00000, 0.00000) -- (0.35689, 0.13007, 0.40089) -- (0.19467, -0.13874, 0.32071) -- cycle {};
\draw[edge] (0.06489, 0.17342, 0.53452) -- (-0.34607, 0.17342, 0.53452);
\draw[edge] (0.06489, 0.17342, 0.53452) -- (-0.04867, -0.13007, 0.53452);
\draw[edge] (0.06489, 0.17342, 0.53452) -- (0.19467, 0.52027, 0.35635);
\draw[edge] (0.06489, 0.17342, 0.53452) -- (0.35689, 0.13007, 0.40089);
\draw[edge] (-0.34607, 0.17342, 0.53452) -- (-0.04867, -0.13007, 0.53452);
\draw[edge] (-0.34607, 0.17342, 0.53452) -- (0.19467, 0.52027, 0.35635);
\draw[edge] (-0.34607, 0.17342, 0.53452) -- (-0.29200, 0.04336, 0.13363);
\draw[edge] (-0.34607, 0.17342, 0.53452) -- (-0.12978, -0.34684, 0.17817);
\draw[edge] (-0.04867, -0.13007, 0.53452) -- (0.35689, 0.13007, 0.40089);
\draw[edge] (-0.04867, -0.13007, 0.53452) -- (0.19467, -0.13874, 0.32071);
\draw[edge] (-0.04867, -0.13007, 0.53452) -- (-0.12978, -0.34684, 0.17817);
\draw[edge] (0.19467, 0.52027, 0.35635) -- (0.35689, 0.13007, 0.40089);
\draw[edge] (0.35689, 0.13007, 0.40089) -- (0.19467, -0.13874, 0.32071);
\draw[edge] (0.35689, 0.13007, 0.40089) -- (0.41096, 0.00000, 0.00000);
\draw[edge] (-0.29200, 0.04336, 0.13363) -- (-0.12978, -0.34684, 0.17817);
\draw[edge] (0.19467, -0.13874, 0.32071) -- (-0.12978, -0.34684, 0.17817);
\draw[edge] (0.19467, -0.13874, 0.32071) -- (0.41096, 0.00000, 0.00000);
\draw[edge] (-0.12978, -0.34684, 0.17817) -- (0.41096, 0.00000, 0.00000);
\node[vertex] at (0.06489, 0.17342, 0.53452)     {};
\node[vertex] at (-0.34607, 0.17342, 0.53452)     {};
\node[vertex] at (-0.04867, -0.13007, 0.53452)     {};
\node[vertex] at (0.19467, 0.52027, 0.35635)     {};
\node[vertex] at (0.35689, 0.13007, 0.40089)     {};
\node[vertex] at (-0.29200, 0.04336, 0.13363)     {};
\node[vertex] at (0.19467, -0.13874, 0.32071)     {};
\node[vertex] at (-0.12978, -0.34684, 0.17817)     {};
\node[vertex] at (0.41096, 0.00000, 0.00000)     {};
\end{tikzpicture}}
         \caption{$\KRW{\metric_2}$}
         \label{fig:4ptc}
     \end{subfigure}
     \caption{Illustration for \Cref{ex:4pt}. (a), (c): KRW polytopes of the metrics $\metric_1$, $\metric_2$. (b): Combinatorial type of the tight span. The horizontal edges are of length $1$ and $2$ for the metrics $\metric_1$ and $\metric_2$, respectively. All other edges are of length $1$.}
     \label{fig:4pt}
\end{figure}

\section{Noteworthy subfans}\label{sec:subfans}
\subsection{Tree-like metrics} An $n$-metric $\metric$ is called {\em tree-like} if it can be represented as the shortest-path metric among a subset of $n$ vertices of a weighted tree $T(\metric)$. It is known since work of Buneman \cite{Buneman} that an $n$-metric $\metric$ is tree-like if and only if it satisfies the four-point condition 
\begin{equation}\label{tree-condition}
\metric_{ij} + \metric_{kl} \leq \max\{\metric_{ik} + \metric_{jl}, \metric_{il}+\metric_{jk}\} 
\textrm{ for all } i,j,k,l\in [n].
\end{equation}

These are exactly the metric spaces whose tight span is a tree (i.e., a one-dimensional contractible polyhedral complex), see e.g., \cite[Theorem 8]{Dress}.

\begin{proposition}\label[proposition]{prop:tl}
The set of tree-like metrics supports a subfan of $\GPF{n}$ and $\Tn{n}$. In particular, the set of phylogenetic trees on $n$ leaves is a union of (relatively open) cells of $\GPF{n}$.
\end{proposition}
\begin{proof} A metric $\metric$ satisfies the four-point condition \eqref{tree-condition} for a quadruple $A=\{i,j,k,l\}\in {n\choose 4}$ if and only if, of the three numbers $\metric_{ij} + \metric_{kl},\metric_{ik} + \metric_{jl}, \metric_{il}+\metric_{jk}$, two are equal and the third does not exceed the other two. In turn, this is equivalent to saying that there is a permutation $\pi_A=(i',j',k',l')$ of the elements of $A$ such that $$\metric\in \mathscr C_{\pi_A}:= H_{(i',j'),(k',l')} \cap H_{(i',k'),(j',l')}^+.$$
 This means that a metric $\metric$ is tree-like if and only if there is a choice of a permutation $\pi_A$ for every $A\in {n\choose 4}$ such that $\metric \in \bigcap_{A\in {n\choose 4}} \mathscr C_{\pi_A}$. Thus, letting $S_A$ denote the group of permutations of the set $A$, the set of tree-like metrics can be expressed as 
$$
\bigcap_{A\in {n\choose 4}} 
\bigcup_{\pi_A \in S_A
}
\mathscr C_{\pi_A}.
$$
Since each cone  $\mathscr C_{\pi_A}$ supports a subfan of~$\GPF{n}$, it follows that the set of all tree-like metrics supports a subfan of~$\GPF{n}$ as well.
According to~\Cref{prop:refinement}.(1) these hyperplanes are all part of the type fan $\Tn{n}$.
Hence, the set of tree-like metrics is also a subfan of $\Tn{n}$.
\end{proof}

\subsection{Circular-decomposable metrics}

An $n$-metric $\metric$ is circular-decomposable if there is a cyclic ordering $\prec$ of $[n]$ such that 
\begin{equation}\label{circular_condition}
\metric_{ij} + \metric_{kl} \leq \metric_{ik} + \metric_{jl} \textrm{ whenever } i\prec j \prec k \prec l.
\end{equation}
To the best of our knowledge this condition first appeared in work of Kalmanson \cite{kalmanson}, and was proved in \cite{CF98} to be equivalent to the ``maximal compatibility" condition studied by Bandelt and Dress in \cite{BandeltDress}. 

From \Cref{circular_condition} it follows immediately that a point in $\R^{\binom{n}{2}}$ corresponds to a circular-decomposable (pseudo-)metric if and only if it is contained in the cone
\begin{equation}\label{eq:Kcirc}
    K_\prec:=\bigcap_{i\prec j \prec k \prec l} H^+_{(i,l),(j,k)}
\end{equation}
for some circular ordering $\prec$ of $[n]$.
 The following definition is from \cite{Devadoss-petti}. 

\begin{definition}
A split $A\uplus B = [n]$ is said to be compatible with a circular ordering $\prec$ if $i\prec j$ for all $i\in A$
 and all $j\in B$. Equivalently, we can label the elements of $[n]$ as $x_1,\ldots,x_n$ in such a way that $x_i\prec x_j$ if $i\in A $ and $j\in B$.
 \end{definition}

\begin{lemma} \label[lemma]{lem:circular}
For every cyclic ordering $\prec$ the set $K_{\prec}$ is a simplicial cone of dimension ${\binom{n}{2}}$. Its rays are generated by the pseudometrics associated to the splits that are compatible with $\prec$.  
Irrespective of the chosen $\prec$, the cone $K_{\prec}$ contains the cone $\elsplitsclosed{n}$ generated by the elementary splits (\Cref{elem_splits}).
\end{lemma}

\begin{proof}
The last claim follows after noticing that elementary splits are compatible with every circular ordering. The other claims are proved as items $(i)-(iii)$ in \cite[Theorem 5, p.\ 75]{BandeltDress}.
\end{proof}

\begin{theorem}\label{thm:cd}
The set of all circular decomposable pseudometrics on $n$ points is the full-dimensional simplicial subfan $\kfanclosed{n}$ of $\GPFclosed{n}$ given by all points contained in the set
$$
\bigcup_{\substack{\prec \textrm{ a circular}\\ \textrm{ordering of }[n]}} K_\prec. 
$$
Accordingly, the circular decomposable metrics form a full-dimensional subfan $\kfan{n}$ of $\GPF{n}$.
Moreover, $\kfan{n}$ is also a full-dimensional subfan of the type fan $\Tn{n}$.
\end{theorem}
\begin{proof} \Cref{eq:Kcirc} implies immediately that $\kfanclosed{n}$ is indeed a subfan of $\GPFclosed{n}$. The claim about simpliciality and dimension follows from \Cref{lem:circular}.
As above, the statement on the type fan follows from \Cref{prop:refinement}. 
\end{proof}

\begin{remark}[The space of circular split networks]\label[remark]{DP-HKT}
A metric is circular decomposable if and only if it can be realized by a {\em circular split network} \cite[\S 2.2]{Devadoss-petti}. Devadoss and Petti \cite{Devadoss-petti} define and study a simplicial fan $\operatorname{CSN}_n$ that parametrizes all circular split networks on $n$ points associated to circular decomposable metrics without trivial splits. 
\end{remark}

\begin{corollary} The set of circular decomposable metrics without trivial splits defines a subfan of $\GPFclosed{n}$ consisting of all cones of $\kfanclosed{n}$ whose closure does not contain any ray of $\elsplitsclosed{n}$. 
\end{corollary}

\subsection{Split decomposition} 
We turn to the class of totally split-decomposable spaces, as introduced by Dress and studied by Bandelt and Dress \cite{BandeltDress}, from which we recall some notions and notations.  Note that a more generals theory of split decompositions beyond metric spaces has been developed by Herrmann and Joswig in \cite{HJ08}.
 
\begin{definition}\label[definition]{def:ii}
Let $\metric$ be a symmetric function on a set $X$ and let $A, B \subseteq X$. Define the \emph{isolation index} $\alpha_{A,B}^\metric$ as \[\alpha_{A, B}^\metric = \frac{1}{2} \cdot \min_{\substack{a,a'\in A \\b,b'\in B}} \left(\max\{a'b+b'a, a'b'+ ab, aa'+bb'\}-(aa'+bb')\right), \] where $uv = \metric(u,v)$. 
\end{definition}

\begin{remark}
As noted in \cite{BandeltDress}, $\alpha^\rho_{A,B} \geq 0$ with $\alpha^\rho_{A,B} =0$ whenever $A\cap B \not= \emptyset$. Moreover, if $\metric$ is a metric and both $A$ and $B$ have at least 2 elements, then  $\alpha^\rho_{A,B} = \alpha^\rho_{\{a,a'\}, \{b,b'\}}$ for some points $a, a'\in A$, $b, b' \in B$ with $a\neq a'$, $b\neq b'$. 
In what follows, when writing expressions such as $\alpha^\rho_{\{a,a'\}, \{b,b'\}}$
we will assume that $a,a',b,b'$ are pairwise distinct points.
\end{remark}

\begin{theorem}[\cite{BandeltDress}, Theorem 2 and Corollary 2]\label{thm:dec}
    Every symmetric function $\metric: X \times X\rightarrow \mathbb{R}$ on a finite set $X$ can be expressed as \[ \metric = \metric_0 + \sum \alpha^\metric_{A,B} \cdot \delta_{A,B}\] with $\metric_0$ split-prime\footnote{A bipartition $A \uplus B=[n]$ is called a $\metric$-split if $\alpha^\metric_{A,B} > 0$. 
A symmetric function $\rho$ which does not admit any $\rho$-splits is called \emph{split-prime}. }. The metric $\metric_0$ is uniquely determined.
\end{theorem}

\begin{definition}\label{def:cd}
    A metric $\metric$ is called \emph{ totally split-decomposable} if it can be written as a positively weighted sum of split metrics. Equivalently, $\metric_0 = 0$ in \Cref{thm:dec}. 
\end{definition}

\begin{theorem}[\cite{BandeltDress}, Theorem 6] 
\label{decomptheo}
    Let $\metric: X \times X \rightarrow \mathbb{R}$ be a symmetric function with zero diagonal. 
        The following conditions are equivalent: 
    \begin{enumerate}
    \item $\metric$ is a totally split-decomposable metric. 
    \item For all $t,u,v,w,x \in X$ 
    \[\alpha_{\{t,u\}, \{v,w\}} \leq \alpha_{\{t,x\}, \{v,w\}} + \alpha_{\{t,u\}, \{v,x\}}.\]
    \end{enumerate}
\end{theorem}

\begin{remark}
From \Cref{decomptheo} it is  clear that all metrics on $4$ points are split-decomposable. For $\vert X \vert \geq 5$ we can ask the question whether two metrics in the same cone of the Wasserstein arrangement are either both split-decomposable or both not split-decomposable. The following example answers this question in the negative.
 \end{remark}

\begin{example}\label[example]{example_decomp} 
	Consider the two metrics $\metric_1, \metric_2$ on five points given by the following matrices:
\begingroup
\[\metric_1 = \begin{pmatrix}
  0 & 125 &  48 & 149 &  84\\
125 &   0 & 149 &  48 &  99\\
 48 & 149 &   0 & 125 &  77\\
149 &  48 & 125 &   0 &  92\\
 84 &  99 &  77 &  92 &   0\\
\end{pmatrix}
~~~\metric_2 = \begin{pmatrix}
    0 & 7447 &  4316 & 10083 &  5584\\
 7447 &     0 & 10083 &  4316 &  5199\\
 4316 & 10083 &     0 &  7447 &  7560\\
10083 &  4316 &  7447 &     0 &  6179\\
 5584 &  5199 &  7560 &  6179 &     0
 \end{pmatrix}
\]
\endgroup 
These metrics lie  in the same cone of the Wasserstein arrangement $\CA{5}$, and therefore the associated KRW polytopes have the same combinatorial structure.
The KRW polytope of $\metric_1$ is depicted in \Cref{fig:schlegel}. Sketches of the tight spans are depicted in \Cref{fig:tight_span_rho_1,fig:tight_span_rho_2}. 

The metric $\metric_1$ is split-decomposable but for the metric $\metric_2$ we find \[\alpha^{\metric_2}_{\{2,3\}, \{1,5\}} > \alpha^{\metric_2}_{\{3,4\}, \{1,5\}} + \alpha^{\metric_2}_{\{2,3\}, \{4,5\}}.\] 
By \Cref{decomptheo}, this certifies  that $\metric_2$ is not split-decomposable.
\end{example}

\begin{figure}[h!]
\includegraphics[scale=.44]{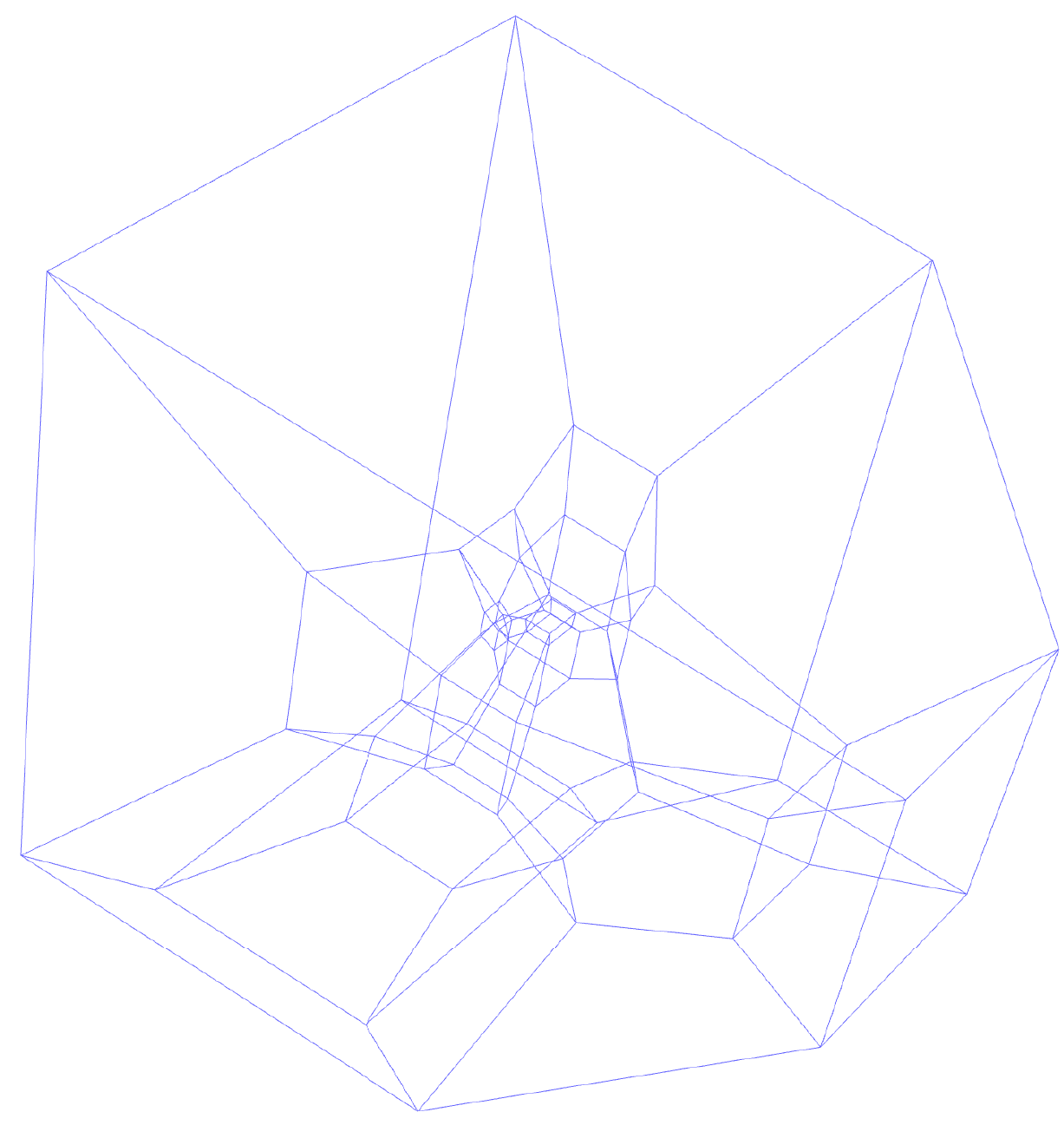}
\caption{Schlegel diagram of the polar dual of the KRW polytope (i.e., the Lipschitz polytope) of the metric $\metric_1$ from  \Cref{example_decomp}.}
\label{fig:schlegel}
\end{figure}

\begin{figure}
\includegraphics[width=0.7\textwidth]{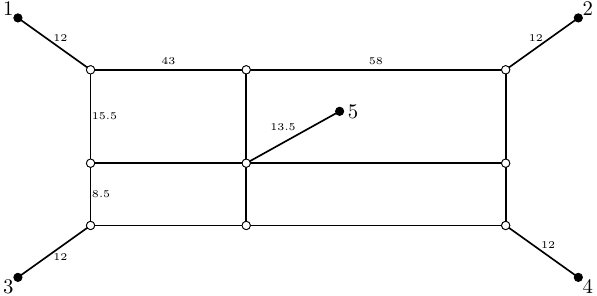} 
		\caption{Sketch of the tight span of the metric $\rho_1$ in \Cref{example_decomp}. The complex $E(\rho_1)$ has dimension $2$ and consists of four rectangular faces and five dangling edges. It has $14$ total vertices, five of which (the solid dots in the picture) are the image of the canonical embedding of the metric space into its injective hull. We have labeled some edges with a number corresponding to the distance of the edge's endpoints with respect to the distance in the injective hull, which is induced on  $E(\rho_1)\subseteq \mathbb R^{5}$ by the uniform norm (``supremum norm"). Edges that appear parallel in the picture have the same ``length". See also \Cref{gendegen}.
  }
		\label{fig:tight_span_rho_1}
  \end{figure}

		\begin{figure}
  \includegraphics[width=0.7\textwidth]{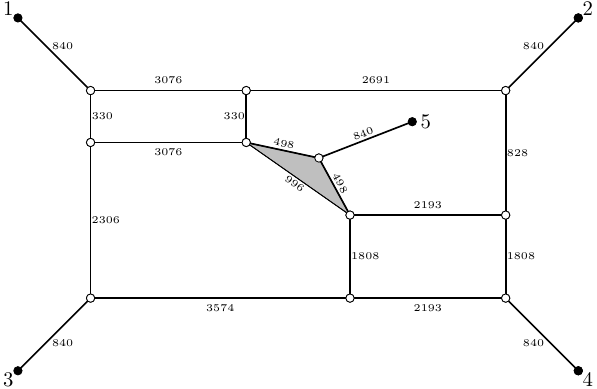}
		\caption{Sketch of the tight span of the metric $\rho_2$ in \Cref{example_decomp}. The complex $E(\metric_2)$ has $5$ two-dimensional cells (two pentagons, two quadrilaterals, one triangle) and five dangling edges. This metric is ``generic of type II" in the sense of Bandelt and Dress \cite{BandeltDress}.  
  }
		\label{fig:tight_span_rho_2}
\end{figure}

\begin{remark}\label[remark]{gendegen}
Note that the notion of ``generic" metric discussed in  \cite{BandeltDress} is not the same as the one of \Cref{def:KRW}. 
Both metrics in \Cref{example_decomp} are generic in the sense of our \Cref{def:KRW}, but only $\rho_2$ is generic in the sense of \cite{BandeltDress}, while $\rho_1$ is  ``degenerate" in the sense of Bandelt and Dress because the metrics used in the canonical decomposition of $\rho_1$ are a strict subset of those for $\rho_2$ (indeed $\rho_2$ has an additional non-zero contribution from a split-prime metric).
\end{remark}

\begin{remark}
Restricting to the metrics inside a specific cone of the Wasserstein arrangement~$\W_5$, the conditions of \Cref{decomptheo} form a local hyperplane  arrangement within this cone, where the split-decomposable metrics lie in exactly one chamber. 

Indeed, taking a look at the condition for split-decomposability of \Cref{decomptheo} for $n = 5$, since the isolation indices are defined on sets of two elements, the minimum in \Cref{def:ii} runs just over one possibility, while the maximum is determined by the chamber of $\W_5$ the metric is in: the maximum over the three values can be computed by pairwise comparison, but comparing the values $ab+a'b'$ and $a'b+ab'$ is equivalent to asking which of the two edge pairs $\{(a,b), (a',b')\}, \{(a',b), (a,b')\}$ is admissible. 
\end{remark}

\subsection{Consistent totally split-decomposable metrics} 

Let $\metric$ be a totally split-decomposable metric. In \cite[Theorem 1]{SixPoints} (see also \cite{AnExplicit})  it is proved that the tight span of $\metric$ is combinatorially isomorphic to the so-called {\em Buneman complex} if and only if $\metric$ satisfies the following \emph{six-point condition}:\\

\noindent{\eqnum\label{def:SixPoints}}
\begin{minipage}[t]{.9\linewidth}
For all $I\subseteq [n]$ with $\vert I \vert = 6$ there are distinct elements $i,j\in I$ with
\begin{equation*}
    \metric_{ij} + \metric_{kl} \leq \max\{\metric_{ik}+\metric_{jl},\metric_{jk}+\metric_{il}\}
\end{equation*}
for all pairs of distinct $k,l \in I\setminus\{i,j\}$.
\end{minipage}

\begin{proposition}\label[proposition]{prop:6p}
    The set of metrics that satisfy the six-point condition \eqref{def:SixPoints} supports a subfan of $\GPFclosed{n}$ and of the type fan $\Tn{n}$.
\end{proposition}
\begin{proof}
Notice that the inequality in \Cref{def:SixPoints} is satisfied whenever 
$\rho\in T_{ijkl}$, where
$$
T_{ijkl}=H^+_{(i,k),(j,l)} \cup H^+_{(i,l),(j,k)}
$$
is a union of (closed) chambers of $\CA{n}$ and, by \Cref{prop:refinement}, also a union of cones of $\Tn{n}$. Then, $\metric$ satisfies  the six-point condition in full if and only if
$$
\metric \in \bigcap_{
\substack{I\subseteq [n]\\
\vert I \vert = 6}}
\;\;
\bigcup_{
\substack{i,j\in I \\
i\neq j }}
\;\;
\bigcap_{
\substack{k,l\in I\setminus \{i,j\}\\
k\neq l
}}
T_{ijkl}.
$$
This proves that the set of all $\metric$ that satisfy the six-point condition is obtained uniting and intersecting sets of (closed) chambers, resp.\ of cones of $\Tn{n}$, hence it is a union of faces of the arrangement $\GPFclosed{n}$ and of the fan $\Tn{n}$. 
\end{proof}

\begin{proposition} \label[proposition]{rem:consistent}
A totally split-decomposable metric on $n\geq 6$ points that satisfies \eqref{def:SixPoints} is called {\em consistent} in \cite{SixPoints}. The set of consistent totally split-decomposable metrics does not support a subfan of the Wasserstein fan, hence it does not support a subfan of the type fan either.
\end{proposition}
\begin{proof} Let $\metric_1$, $\metric_2$ be the $5$-metrics defined in \Cref{example_decomp}, let $\tk{5}{k}$ be the ``path metric on $k$ edges" defined in \Cref{def:tk}, and recall the direct sum construction of \Cref{lem:directsum}. 

For $k\geq 1$, recall  the metrics
$
\rho_1^{(k)}:=\rho_1 \oplus \tk{5}{k}$, $\rho_2^{(k)}:= \rho_2 \oplus \tk{5}{k}
$ from the proof of \Cref{newthm}, where it is shown that, for all $k$, the metrics $\rho_1^{(k)}$ and $\rho_2^{(k)}$ are in the same cone of the Wasserstein fan $\W_{5+k}$, yet one is split-decomposable but the other is not. It will be enough to show that, for every $k\geq 1$, both these metrics satisfy the six-point condition.
Indeed: the $5$-metrics $\rho_1$ and $\rho_2$ trivially satisfy the six-point condition \eqref{def:SixPoints}. Moreover, one can prove that if an $n$-metric $\rho$ satisfies the six-point condition then so does the $(n+1)$-metric $\rho\oplus\tk{n}{1}$ (if the six-point set $I$ does not contain $n+1$ there is nothing to show; if it contains both $n$ and $n+1$ the condition reduces to the triangle inequality; if $n+1\in I$ and $n\not\in I$ then the six-point condition for $I$ is implied by the six-point condition for $(I\setminus \{n+1\}) \cup \{n\}$). Hence, for every $k\geq 1$, both $\rho_1^{(k)}$ and $\rho_2^{(k)}$ satisfy the six-point condition.
\end{proof}

\subsection{Antipodal metrics}
\newcommand{\ant}[1]{\overline{#1}}
\def\fant{\overline{\strut}}
 
Following e.g.~\cite{DHMantipodal}, an $n$-metric $\metric$ is called {\em antipodal} if there is a function $[n]\to[n]$ denoted by $i\mapsto \ant{i}$,
such that for all $i,j\in [n]$
\begin{equation}\label{cond:anti}
\metric_{i,\ant{i}} = \metric_{i,j}+\metric_{j,\ant{i}}.
\end{equation}
\begin{remark}
Equation~\eqref{cond:anti} already implies  $\ant{\ant{i}}=i$ for all $i\in [n]$:
indeed, this equation yields $\metric_{i,\ant{i}}=\metric_{i,\ant{\ant{i}}}+\metric_{\ant{\ant{i}},\ant{i}}$ and $\metric_{\ant{i},\ant{\ant{i}}}=\metric_{\ant{i},i}+\metric_{i,\ant{\ant{i}}}$ and hence $\metric_{i,\ant{\ant{i}}}=0$.
\end{remark}

Antipodality imposes strong constraints on the combinatorial structure of the tight span (e.g., a finite metric space is antipodal if and only if its tight span has a unique maximal cell, see \cite{HuKoMo1}). 
The next lemma relates antipodality of a metric
 to the face structure of the corresponding KRW polytope.

\begin{proposition}\label[proposition]{prop:anti}
The set of antipodal $n$-metrics is a union of relatively open cones of the type fan $\Tn{n}$ and, thus, of the Wasserstein fan $\GPF{n}$.
\end{proposition}
\begin{proof}
From \Cref{cor:disjointight} it follows that two metrics with the same labeled combinatorial type of the KRW polytope must have the same set of ``tight triples'', and the condition of antipodality is equivalent to specifying that certain triples are tight.
\end{proof}

\begin{remark}\label[remark]{antisubfan}
The set of metrics that are antipodal with respect to a fixed involution is topologically closed inside $\mcone{n}$, as it is cut out by intersecting $n$ subspaces of the form \eqref{cond:anti}. Since there are a finite number of involutions on a finite set and a finite intersection of closed sets is closed, the set of all antipodal metrics is also closed inside $\mcone{n}$. Thus, the set of antipodal metrics supports a subfan of the (half-open) fan $\Tn{n}$ (see  \Cref{remcone} for terminology).
\end{remark}

\section{The $f$-vector in the strict case}\label{sec:strict_case}

This section is motivated by the following result of Gordon and Petrov. Recall from \Cref{def:KRW} the definition of a generic metric.

\begin{theorem}[\cite{GordonPetrov}, Theorem 1]
Let $\metric$ be a generic $n+1$-metric. 
   Then, for $0 \leq m \leq n$, the number of $m$-dimensional faces of $\KRW{\metric}$ is $\binom{n + m}{m,m,n-m}$.  
\end{theorem}

The proof of this theorem uses the notion of {\em admissible graph} discussed in \Cref{sec:GP}. 

In fact, using \cite[Corollary 1.(3)]{GordonPetrov} the number of faces of dimension $(n-m)$ in the generic case agrees with the number of different admissible (directed) forests with $m$ edges on $n+1$ vertices. 

The strict, non-generic case is somewhat more complicated, but general formulae for faces of dimension $0$ and $1$ can be stated. The number of vertices has been computed by Gordon and Petrov.

\begin{corollary}[\cite{GordonPetrov} Corollary 1.(2)] 
Let $\metric$ be a strict metric, then the number of vertices of the KRW polytope is $f_0(\KRW{\metric}) = n \cdot (n+1)$, i.e., all points $p_{i,j}$ are vertices of KRW.
\end{corollary}

We give a formula for the number of edges of KRW polytopes of strict metrics using the graphic characterization of the faces: a set of vertices $F$ of the polytope defines a face if the intersection of the vertex-sets of all facets containing $F$ is identical to $F$. In order to check whether a given (admissible) graph $G$ is the graph $G(F)$ for some  face $F$ of the polytope (see \Cref{sec:GP}), we need to check that for every edge $e$ not in $G$, there is a maximal admissible graph (i.e., a facet) containing all edges of $G$, but not $e$. 

We now move to a computation of the number of edges of KRW polytopes of strict metrics, in terms of the following invariant.

\begin{definition} Recall the sets $\CA{n}^k$ from \Cref{def:2k_cycle_planes}. 
 For a strict metric $\metric$  on $l$ points  define $r_k(\metric) := \vert\{H \in \CA{l}^k \mid \metric \in H\}\vert$.
\end{definition}

We start with a preparatory lemma.

\begin{lemma}\label[lemma]{lem:2-adm}
    Let $\metric$ be a strict $l$-metric which is not contained in any hyperplane of $\CA{l}^2$. Then for each quadruple $x_1, y_1,x_2,y_2 \in [l]$, if the directed graph with edge-set $\{(x_1,y_1),(x_2,y_2)\}$ 
    is admissible, then it is face-defining. 
\end{lemma}
\begin{proof}
    Let $\{(x_1,y_1), (x_2,y_2)\}$ be a pair of admissible edges and let $e = (i,j)$ be another edge which is admissible together with the pair. We need to show that there is an admissible graph $G$ containing $\{(x_1,y_1), (x_2,y_2)\}$ but not $e$, such that  the addition of $e$ to $G$ results in an inadmissible set. We use the strictness of the metric and the strict inequality:
     \begin{equation} \label{ineq} 
     	\metric_{x_1,y_1} + \metric_{x_2,y_2} < \metric_{x_1, y_2} + \metric_{x_2, y_1}.
    \end{equation}
    Note that this inequality holds strictly for $\metric$ as the graph is admissible and by assumption the metric does not lie on any hyperplane in $\CA{l}^2$.

    The trivial case is $\vert e \cap \{x_1,y_1,x_2,y_2\}\vert = 2$, then $e$ is one of $(x_1, y_2), (x_2, y_1)$ as $\metric$ is strict. In this case, the graph we are looking for is the pair of edges $\{(x_1,y_1), (x_2,y_2)\}$ together with the other edge, i.e., $(x_2, y_1)$ or $(x_1, y_2)$: indeed, the set of edges $\{(x_1,y_1), (x_2,y_2),(x_1, y_2), (x_2, y_1)\}$ is inadmissible since by assumption  $\metric$ is contained in no hyperplane of $\CA{l}^2$.

 Now let $\vert e \cap \{x_1,y_1,x_2,y_2\}\vert \in\{0,1\}$. 
If $j\not\in\{x_1,y_1,x_2,y_2\}$, consider the pair of edges $(j, y_1), (j,y_2)$. If neither edge forms an admissible graph together with $\{(x_1,y_1), (x_2,y_2)\}$, then \[\metric_{j, y_2} + \metric_{x_2, y_1} < \metric_{j, y_1} + \metric_{x_2, y_2}\quad\textrm{ and } \quad\metric_{j, y_1} + \metric_{x_1, y_2} < \metric_{j, y_2} + \metric_{x_1, y_1}.\] Adding the two inequalities gives a contradiction to the inequality \eqref{ineq}. Thus the set $\{(x_1,y_1), (x_2,y_2)\}$ together with one of the edges $(j, y_1), (j,y_2)$ forms an admissible graph $G$. Adding $e$ to $G$ produces an inadmissible graph: indeed, $e=(i,j)$ together with either of the edges $(j, y_1), (j,y_2)$ is a directed path of length $2$, which cannot be admissible since $\metric$ is a strict metric, see \Cref{strict_nopaths}. Thus $G$ is the required admissible graph. 
If $j\in\{x_1,y_1,x_2,y_2\}$, then $i\not\in\{x_1,y_1,x_2,y_2\}$ and the same argument with the edge pair $(x_1, i), (x_2, i)$ concludes the proof.
\end{proof}

\begin{proposition}
Let $\metric$ be a strict metric on $l = n+1$ points. Then the number of edges satisfies \[f_1(\KRW{\metric}) = \binom{n+2}{2, 2, n-2} -2\cdot r_2(\metric). \] 
\end{proposition}

 Note that the multinomial coefficient is the number of faces in the generic case, where $\metric$ does not lie on any hyperplane. 

 \begin{proof}
Let $\rho$ be a strict metric. By \Cref{strict_nopaths}, admissible  graphs on two edges are of one of the two types depicted in \Cref{figure-ab1b2}.
\begin{figure}[h]
\begin{subfigure}[t]{0.45\textwidth}
\centering
\begin{tikzpicture}
\node(x1) at (0,0) {$x_1$};
\node(x2) at (1,0) {$x_2$};
\node(y1) at (0,1) {$y_1$};
\node(y2) at (1,1) {$y_2$};
\draw[->] (x1.north) -- (y1.south);
\draw[->] (x2.north) -- (y2.south);
\end{tikzpicture}
\caption{}
\end{subfigure}
\begin{subfigure}[t]{0.45\textwidth}
\centering
\begin{tikzpicture}
\node(x12) at (0,1) {$x_1=x_2$};
\node(y1) at (-1,0) {$y_1$};
\node(y2) at (1,0) {$y_2$};
\draw[->] (x12.south) -- (y1.north);
\draw[->] (x12.south) -- (y2.north);
\end{tikzpicture}
\begin{tikzpicture}
\node(x12) at (0,1) {$x_1=x_2$};
\node(y1) at (-1,0) {$y_1$};
\node(y2) at (1,0) {$y_2$};
\draw[<-] (x12.south) -- (y1.north);
\draw[<-] (x12.south) -- (y2.north);
\end{tikzpicture}
\caption{}
\end{subfigure}
\caption{The two types of admissible graphs on two edges.}\label{figure-ab1b2}
\end{figure}

Each of the $\frac{(n+1)!}{(n-1)!}$ graphs of type (b) is always trivially admissible. Let us count the admissible graphs of type (a): for every choice of an ordered $4$-tuple $(x_1,y_1,x_2,y_2)$ consider the quantities
$$
s^+:= \metric_{x_1,y_1} + \metric_{x_2,y_2},\quad\quad
s^-:= \metric_{x_1,y_2} + \metric_{x_2,y_1}.
$$
There are three cases: either $s^+<s^-$, in which case the graph with the edge set $\{(x_1,y_1),\\(x_2,y_2)\}$ is admissible and the one with the edge set $\{(x_1,y_2),(x_2,y_1)\}$ is not, or $s^+>s^-$, in which case the graph with the edge set $\{(x_1,y_2),(x_2,y_1)\}$ is admissible and the one with the edge set $\{(x_1,y_1),(x_2,y_2)\}$ is not, or $s^+=s^-$, in which case $\rho$ is contained in $H_{(x_1,x_2),(y_1,y_2)}\in \mathcal{W}_l^2$. 

If $r_{2}(\rho)=0$, then there is no instance of the third case and so exactly one of the other alternatives must hold. Moreover, in this setting the admissible graphs obtained in the first two cases are face-defining according to \Cref{lem:2-adm}, and we obtain a total of $\frac{(n+1)!}{4(n-3)!}$ admissible graphs of type (a). Thus, when $\rho$ is not contained in any member of $\mathcal{W}_l^2$ the proposition holds.

If $r_{2}(\rho)>0$, let $H_1,\ldots,H_{r_2(\rho)}$ be the hyperplanes of $\mathcal{W}_l^2$ containing $\rho$. The count of subgraphs of type (b) does not change. We now consider the graphs of type (a). Let $C_{H}$ be a cycle corresponding to $H \in \CA{l}^2$ (see  \Cref{rem:cycle_arr}). What happens if we move from either half-space onto the hyperplane $H$ is that all graphs remain admissible.
Moreover, if an admissible graph $G$ contains the set $C_{H}^+$ or $C_{H}^-$, respectively, the union of $G$ with $C_{H}^-$ or $C_{H}^+$, respectively, becomes admissible, too.
Thus, $C_{H}^+, C_{H}^-$ are no longer  faces of the polytope and we therefore subtract them from the count. Faces of dimension $1$ which correspond to paths on three vertices are not affected. 

Without loss of generality, we can regroup the hyperplanes $H_1,\ldots,H_{r_2(\rho)}$ by the ground set of their underlying cycles.
Each quadruple $\{u,v,w,t\}$ either appears zero, once or three times, since the intersection \[H_{(v,w),(u,t)} \cap H_{(v,u),(w,t)}\cap H_{(v,t), (w,u)}\] has codimension $2$ for any quadruple $\{u,v,w,t\}$.
Thus if $\metric$ is in the intersection of two among the three hyperplanes defined by cycles on a single quadruple, it automatically is in the intersection of all three. 

The first case is trivial, the case of the ground set appearing once is described above and in the last case, we subtract all $6$ faces ($2$ per cycle) from the generic count. 
Doing this iteratively for each quadruple gives the result.
 \end{proof}

\begin{example}
For a low-dimensional example, see \Cref{tab:metrics4}, where the $f$-vectors and corresponding inequalities are listed for $l = 4$. 
\end{example}

 \section{Explicit computations}\label{sec:compuations}
 
 In this section, we explicitly study the KRW polytopes of metrics on $4$, $5$, and $6$ points.
 An accompanying data set of one generic metric of each combinatorial type of these sizes is published on Zenodo~\cite{zenodo}.
 
\subsection{The case $n=4$}\label{sec_n4}

There is only one generic KRW polytope up to symmetry. The Wasserstein arrangement is depicted in \Cref{fig:arrangement4}, where the cycles corresponding to the hyperplanes are shown next to them. 

In this case, the arrangement agrees with the braid arrangement $\mathcal{A}_3$ up to the lineality space and thus has 6 chambers. 

The strict but not generic cases are illustrated in items (2)-(4) in  \Cref{tab:metrics4}. Items 2 and 3 are metrics on the rank one intersections of the hyperplane arrangement, and item 4 is on a rank 2 intersection. Observe that while the $f$-vector only depends on the rank of the intersection, their tight spans differ. 
\begin{center}
\begin{figure}
    \includegraphics[width=0.8\textwidth]{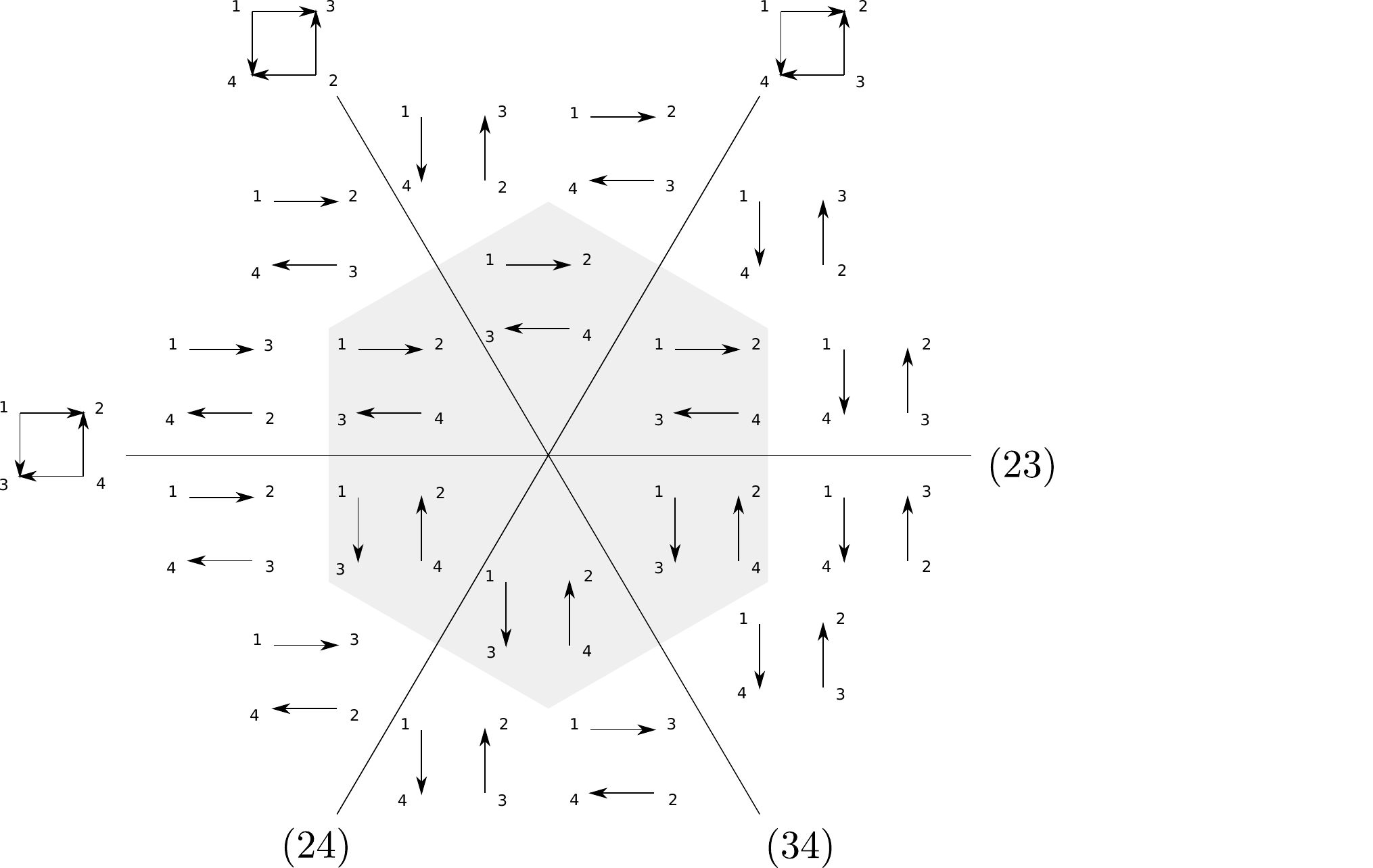} 
		\caption{Wasserstein arrangement for $n = 4$. Each hyperplane corresponds to a cycle as described in  \Cref{rem:cycle_arr} and each chamber has one set of oriented edges of each cycle that is admissible. We  obtain the admissible graphs of a neighboring chamber by applying the permutation specified on the hyperplane separating the chambers. The edge pairs in the grey area are the ones corresponding to the horizontal hyperplane. } 
		\label{fig:arrangement4}
\end{figure}
\end{center}
\newcolumntype{C}{>{\centering\arraybackslash}m{7em}}
\renewcommand\arraystretch{0.5}
\begin{table}\sffamily
\begin{tabular}{l*5{C}@{}}
\toprule
\textsc{Nr.} & \textsc{case} & \textsc{sample metric} & \textsc{$f$-vector} & \textsc{polytope} & \textsc{tight span}  \\ 
\midrule
1 & $\metric_{12|34} > \metric_{13|24} > \metric_{14|23}$ & $\begin{pmatrix}
     0&8&7&5 \\ 8&0 &5&7\\7&5&0&8\\ 5 & 7 &8 &0
\end{pmatrix}$ & (12 30 20) & \includegraphics[width=6em]{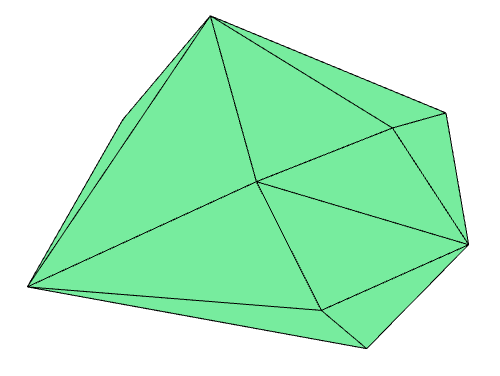} & \includegraphics[width=6em]{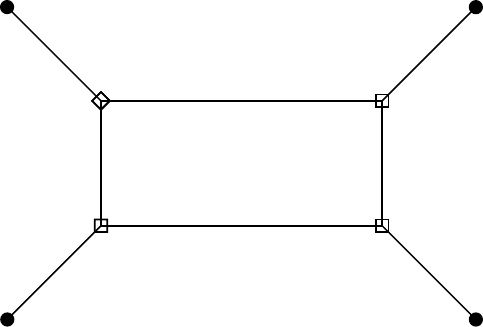}\\ 
2 &  $\metric_{12|34} > \metric_{13|24} = \metric_{14|23}$& $\begin{pmatrix}
     0&6&5&5 \\ 6&0 &5&5\\5&5&0&6\\ 5 & 5 &6 &0
\end{pmatrix}$ & (12 28 18) & \includegraphics[width=6em]{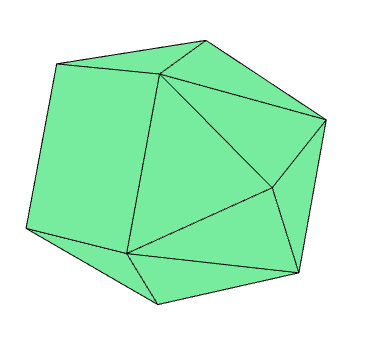} & \includegraphics[width=4em]{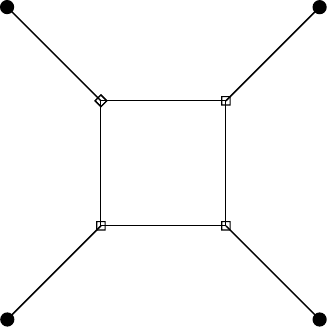} \\ 
3 & $\metric_{12|34} = \metric_{13|24} > \metric_{14|23}$ & $\begin{pmatrix}
     0&3&3&2 \\ 3&0 &2&3\\3&2&0&3\\ 2 & 3 &3 &0
\end{pmatrix}$ & (12 28 18) & \includegraphics[width=6em]{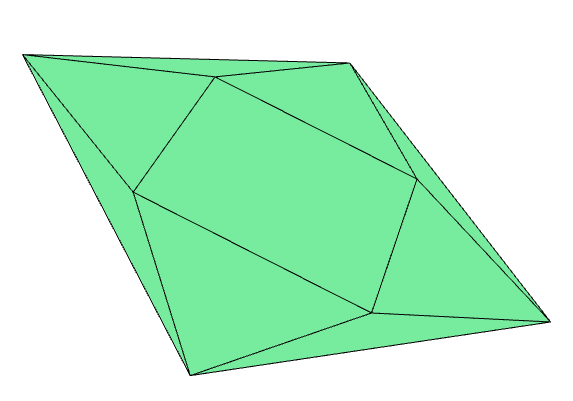} & \includegraphics[width=6em]{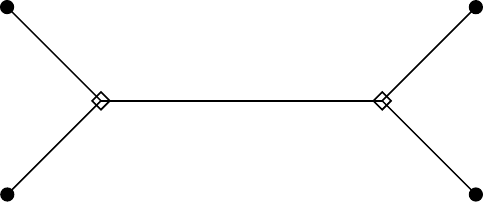} \\ 
4 & $\metric_{12|34} = \metric_{13|24} = \metric_{14|23}$ & $\begin{pmatrix}
     0&1&1&1 \\ 1&0 &1&1\\1&1&0&1\\ 1 & 1 &1 &0
\end{pmatrix}$ & (12 24 14) & \includegraphics[width=6em]{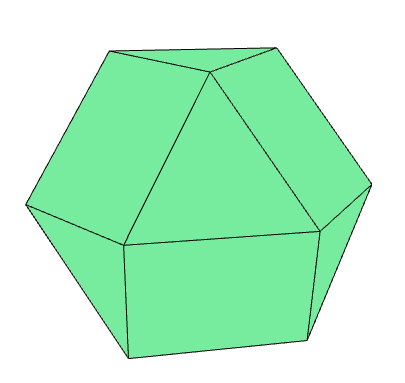} & \includegraphics[width=3em]{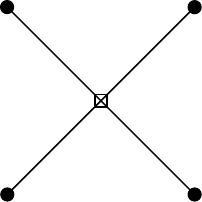} \\ 
\bottomrule 
\end{tabular}
\caption{Types of strict metrics for $n = 4$, their KRW polytopes and the corresponding tight spans. The notation $\metric_{ij|kl}$ is the value of the split of the four points in the sets $\{i,j\}$ and $\{k,l\}$. This is the only relevant information for the classification because every four-point metric is split decomposable and the elementary splits form the lineality space of $\CA{4}$(see \Cref{lem:lineality_space}).}
\label{tab:metrics4}
\end{table} 
\begin{figure}
		\includegraphics[width=0.5\textwidth]{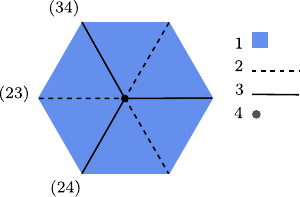} 
		\caption{Wasserstein arrangement for $n= 4$, coded by the tight spans. For detailed encoding of the arrangement see \Cref{fig:arrangement4} and the numbers 1-4 correspond to the tight spans in \Cref{tab:metrics4}. Number 4 is the cone of convex combinations of elementary splits.} 
		\label{fig:tight_span_4}
\end{figure}

\subsection{The case $n=5$}\label{n_gleich_5}
This case was studied using the Julia package \texttt{OSCAR}~\cite{OSCAR-book} and the \texttt{Polymake.jl} package, whose methods were used to find sample inner points of the maximal cones of the arrangement, which are recorded in \Cref{tab:metrics5_gen}. 
The characteristic polynomial of the Wasserstein arrangement~$\CA{5}$ is $\chi_{\CA{5}}(t)=t^{10} - 15t^9 + 90t^8 - 260t^7 + 350t^6 - 166t^5$ which does not factor over the integers.
Thus by Zaslavsky's theorem, the arrangement~$\CA{5}$ has $882$ chambers which agrees with the number of maximal cones of the fan $\GPF{5}$~\cite{Zas75}.
There are 12 generic KRW polytopes up to symmetry that are pairwise not isomorphic. The $f$-vector of these polytopes is the vector $(20\,\, 90\,\, 140\,\, 70)$, as was proved in \cite{GordonPetrov}. We give a description of these metrics in ~\Cref{sec:gen_5_metrics}. 

Moreover, there are 65 combinatorial types of strict but not generic KRW polytopes. 
We computed these using the \texttt{polymake}~\cite{polymake} hyperplane arrangement package~\cite{polymake_arrangements} by enumerating all cells of the Wasserstein arrangement.
We record descriptions of these 65 combinatorial types and their metrics in~\Cref{sec:strict}.

\subsection{The case $n=6$.}
For $6$-point metrics we employed two approaches.
First, we computed the number of chambers of the arrangement $\CA{6}$ using the \texttt{julia} module \texttt{CountingChambers.jl} \cite{BEK23}.
This is an arrangement in $\R^{15}$  consisting of $105$ hyperplanes with a $6$-dimensional lineality space, see \Cref{prop:size_wn}.
We found that $\CA{6}$ has $6,677,863,200$ chambers.

As explained in~\Cref{sec:triangulations}, the generic KRW polytopes correspond exactly  to certain regular triangulations of $RP^0_n$, the (full) root polytope together with the origin, see \Cref{RootTriangulations}.
Jörg Rambau computed  that there are $17,695,320$  symmetric central regular triangulations of $RP^0_6$ using the latest version of his software \texttt{TOPCOM}~\cite{Ram23}.
This count agrees by the above discussion with the number of labeled combinatorial types of generic KRW polytopes.
Up to symmetry, there are exactly $25,224$ such triangulations.
This number therefore agrees with the number of KRW polytopes of $6$-point metrics up to the action of the symmetry group $S_6$.
We record both statements in the following theorem.

\begin{theorem}
	There are exactly $17,695,320$ labeled  combinatorial types of generic KRW polytopes corresponding to $6$-point metrics.
	Under the action of the symmetric group $S_6$ these polytopes come in $25,224$  orbits.
\end{theorem}

\renewcommand*{\bibfont}{\footnotesize}
\printbibliography

@article {DelucchiHoessly,
    AUTHOR = {Delucchi, Emanuele and Hoessly, Linard},
     TITLE = {Fundamental polytopes of metric trees via parallel connections
              of matroids},
   JOURNAL = {European J. Combin.},
  FJOURNAL = {European Journal of Combinatorics},
    VOLUME = {87},
      YEAR = {2020},
     PAGES = {103098, 18},
      ISSN = {0195-6698,1095-9971},
   MRCLASS = {05B35 (52B12)},
  MRNUMBER = {4081480},
MRREVIEWER = {Kevin\ Grace},
       DOI = {10.1016/j.ejc.2020.103098},
       URL = {https://doi.org/10.1016/j.ejc.2020.103098},
}

@Book{BasicPhyl,
 Author = {Dress, Andreas and Huber, Katharina T. and Koolen, Jacobus and Moulton, Vincent and Spillner, Andreas},
 Title = {Basic phylogenetic combinatorics.},
 ISBN = {978-0-521-76832-0},
 Year = {2012},
 Publisher = {Cambridge: Cambridge University Press},
 Language = {English},
 DOI = {10.1017/CBO9781139019767},
 zbMATH = {5707088},
 Zbl = {1298.92008}
}

@article{AF24,
	title = {Tropical combinatorics of max-linear Bayesian networks},
	journal = {Journal of Symbolic Computation},
	volume = {134},
	pages = {102518},
	year = {2026},
	issn = {0747-7171},
	author = {Carlos Améndola and Kamillo Ferry},
	}

@Article{Miesch,
 Author = {Miesch, Benjamin},
 Title = {Gluing hyperconvex metric spaces},
 FJournal = {Analysis and Geometry in Metric Spaces},
 Journal = {Anal. Geom. Metr. Spaces},
 ISSN = {2299-3274},
 Volume = {3},
 Pages = {102--110},
 Year = {2015},
 Language = {English},
 DOI = {10.1515/agms-2015-0007},
 zbMATH = {6440825},
 Zbl = {1321.54053}
}

@Article{DST,
 Author = {De Loera, Jes{\'u}s A. and Sturmfels, Bernd and Thomas, Rekha R.},
 Title = {Gr{\"o}bner bases and triangulations of the second hypersimplex},
 FJournal = {Combinatorica},
 Journal = {Combinatorica},
 ISSN = {0209-9683},
 Volume = {15},
 Number = {3},
 Pages = {409--424},
 Year = {1995},
 DOI = {10.1007/BF01299745},
 zbMATH = {806480},
 Zbl = {0838.52015}
}

@article {GordonPetrov,
    AUTHOR = {Gordon, J. and Petrov, F.},
     TITLE = {Combinatorics of the {L}ipschitz polytope},
   JOURNAL = {Arnold Math. J.},
  FJOURNAL = {Arnold Mathematical Journal},
    VOLUME = {3},
      YEAR = {2017},
    NUMBER = {2},
     PAGES = {205--218},
      ISSN = {2199-6792,2199-6806},
   MRCLASS = {52B12},
  MRNUMBER = {3664266},
MRREVIEWER = {Satyan\ L.\ Devadoss},
       DOI = {10.1007/s40598-017-0063-0},
       URL = {https://doi.org/10.1007/s40598-017-0063-0},
}

@Article{CDM,
 Author = {Chen, Tianran and Davis, Robert and Mehta, Dhagash},
 Title = {Counting equilibria of the {Kuramoto} model using birationally invariant intersection index},
 FJournal = {SIAM Journal on Applied Algebra and Geometry},
 Journal = {SIAM J. Appl. Algebra Geom.},
 ISSN = {2470-6566},
 Volume = {2},
 Number = {4},
 Pages = {489--507},
 Year = {2018},
 DOI = {10.1137/17M1145665},
 zbMATH = {6989290},
 Zbl = {1408.13077}
}

@Article{Chen,
 Author = {Chen, Tianran},
 Title = {Unmixing the mixed volume computation},
 FJournal = {Discrete \& Computational Geometry},
 Journal = {Discrete Comput. Geom.},
 ISSN = {0179-5376},
 Volume = {62},
 Number = {1},
 Pages = {55--86},
 Year = {2019},
 DOI = {10.1007/s00454-019-00078-x},
 zbMATH = {7066920},
 Zbl = {1417.52021}
}

@Article{AnExplicit,
 Author = {Dress, A. and Huber, K. T. and Moulton, V.},
 Title = {An explicit computation of the injective hull of certain finite metric spaces in terms of their associated {Buneman} complex},
 FJournal = {Advances in Mathematics},
 Journal = {Adv. Math.},
 ISSN = {0001-8708},
 Volume = {168},
 Number = {1},
 Pages = {1--28},
 Year = {2002},
 DOI = {10.1006/aima.2001.2039},
 URL = {semanticscholar.org/paper/a0c9039af64896a8f29b56a5b0cb825dc1239d92},
 zbMATH = {1800228},
 Zbl = {1014.54018}
}

@Article{SixPoints,
 Author = {Dress, A. and Huber, K. T. and Koolen, J. H. and Moulton, V.},
 Title = {Six points suffice: {How} to check for metric consistency},
 FJournal = {European Journal of Combinatorics},
 Journal = {Eur. J. Comb.},
 ISSN = {0195-6698},
 Volume = {22},
 Number = {4},
 Pages = {465--474},
 Year = {2001},
 DOI = {10.1006/eujc.2000.0448},
 URL = {semanticscholar.org/paper/5be14c436761bb75b9ab8357d964f6d4cf41b609},
 zbMATH = {1643796},
 Zbl = {0985.05016}
}

@Article{HKT,
 Author = {Hacking, Paul and Keel, Sean and Tevelev, Jenia},
 Title = {Stable pair, tropical, and log canonical compactifications of moduli spaces of del {Pezzo} surfaces},
 FJournal = {Inventiones Mathematicae},
 Journal = {Invent. Math.},
 ISSN = {0020-9910},
 Volume = {178},
 Number = {1},
 Pages = {173--227},
 Year = {2009},
 DOI = {10.1007/s00222-009-0199-1},
 zbMATH = {5602518},
 Zbl = {1205.14012}
}

@InProceedings{Rambau,
  Title                    = {TOPCOM: Triangulations of Point Configurations and Oriented Matroids},
  Author                   = {Rambau, J{\"o}rg},
  Booktitle                = {Proceedings of the International Congress of Mathematical Software},
  Year                     = {2002},
  URL                      = {http://www.zib.de/PaperWeb/abstracts/ZR-02-17}
}

@book{HRS, place={Cambridge}, title={Phylogenetic Networks: Concepts, Algorithms and Applications}, publisher={Cambridge University Press}, author={Huson, Daniel H. and Rupp, Regula and Scornavacca, Celine}, year={2010}}

@Article{Becedasetal,
 Author = {Becedas, Adrian and Kohn, Kathl{\'e}n and Venturello, Lorenzo},
 Title = {Voronoi diagrams of algebraic varieties under polyhedral norms},
 FJournal = {Journal of Symbolic Computation},
 Journal = {J. Symb. Comput.},
 ISSN = {0747-7171},
 Volume = {120},
 Pages = {14},
 Note = {Id/No 102229},
 Year = {2024},
 DOI = {10.1016/j.jsc.2023.102229},
 zbMATH = {7725354},
 Zbl = {1519.14056}
}

@Article{venturelloetal,
 Author = {{\c{C}}elik, T{\"u}rk{\"u} {\"O}zl{\"u}m and Jamneshan, Asgar and Mont{\'u}far, Guido and Sturmfels, Bernd and Venturello, Lorenzo},
 Title = {Wasserstein distance to independence models},
 FJournal = {Journal of Symbolic Computation},
 Journal = {J. Symb. Comput.},
 ISSN = {0747-7171},
 Volume = {104},
 Pages = {855--873},
 Year = {2021},
 DOI = {10.1016/j.jsc.2020.10.005},
 zbMATH = {7312503},
 Zbl = {1460.62205}
}

@Article{BHV,
 Author = {Billera, Louis J. and Holmes, Susan P. and Vogtmann, Karen},
 Title = {Geometry of the space of phylogenetic trees},
 FJournal = {Advances in Applied Mathematics},
 Journal = {Adv. Appl. Math.},
 ISSN = {0196-8858},
 Volume = {27},
 Number = {4},
 Pages = {733--767},
 Year = {2001},
 DOI = {10.1006/aama.2001.0759},
 zbMATH = {1763534},
 Zbl = {0995.92035}
}

@Article{Dress89,
 Author = {Dress, Andreas W. M.},
 Title = {Towards a classification of transitive group actions on finite metric spaces},
 FJournal = {Advances in Mathematics},
 Journal = {Adv. Math.},
 ISSN = {0001-8708},
 Volume = {74},
 Number = {2},
 Pages = {163--189},
 Year = {1989},
 DOI = {10.1016/0001-8708(89)90008-X},
 zbMATH = {4119226},
 Zbl = {0683.54041}
}

@article {BandeltDress,
    AUTHOR = {Bandelt, Hans-J\"{u}rgen and Dress, Andreas W. M.},
     TITLE = {A canonical decomposition theory for metrics on a finite set},
   JOURNAL = {Adv. Math.},
  FJOURNAL = {Advances in Mathematics},
    VOLUME = {92},
      YEAR = {1992},
    NUMBER = {1},
     PAGES = {47--105},
      ISSN = {0001-8708,1090-2082},
   MRCLASS = {54E35 (05C99)},
  MRNUMBER = {1153934},
MRREVIEWER = {Henry\ Martyn\ Mulder},
       DOI = {10.1016/0001-8708(92)90061-O},
       URL = {https://doi.org/10.1016/0001-8708(92)90061-O},
}

@misc {zenodo,
    author = {Delucchi, Emanuele and M\"uhlherr, Leonie and K\"uhne, Lukas and Rambau, J\"org},
    title = {Dataset of {K}antorovich-{R}ubinstein-{W}asserstein {P}olytopes of {M}etric {S}paces on up to 6 {P}oints},
    publisher = {Zenodo},
    year={2024},
    doi = {10.5281/zenodo.12773907},
    eprint = {http://dx.doi.org/10.5281/zenodo.12773907},
}

@article {JS19,
    AUTHOR = {Joswig, Michael and Schr\"oter, Benjamin},
     TITLE = {Parametric shortest-path algorithms via tropical geometry},
   JOURNAL = {Math. Oper. Res.},
  FJOURNAL = {Mathematics of Operations Research},
    VOLUME = {47},
      YEAR = {2022},
    NUMBER = {3},
     PAGES = {2065--2081},
      ISSN = {0364-765X,1526-5471},
   MRCLASS = {68R10 (14T90 90B06 90C24 90C35)},
  MRNUMBER = {4506364},
    eprint={1904.01082},
	archivePrefix={arXiv},
	primaryClass={math.CO}
}

@article {Isbell,
	AUTHOR = {Isbell, J. R.},
	TITLE = {Six theorems about injective metric spaces},
	JOURNAL = {Comment. Math. Helv.},
	FJOURNAL = {Commentarii Mathematici Helvetici},
	VOLUME = {39},
	YEAR = {1964},
	PAGES = {65--76},
	ISSN = {0010-2571,1420-8946},
	MRCLASS = {54.35},
	MRNUMBER = {182949},
	MRREVIEWER = {M.\ Kat\v{e}tov},
	DOI = {10.1007/BF02566944},
	URL = {https://doi.org/10.1007/BF02566944},
}

@article {Dress,
	AUTHOR = {Dress, Andreas W. M.},
	TITLE = {Trees, tight extensions of metric spaces, and the
	cohomological dimension of certain groups: a note on
	combinatorial properties of metric spaces},
	JOURNAL = {Adv. in Math.},
	FJOURNAL = {Advances in Mathematics},
	VOLUME = {53},
	YEAR = {1984},
	NUMBER = {3},
	PAGES = {321--402},
	ISSN = {0001-8708},
	MRCLASS = {05C10 (20F32 20J05 54C25 54F50)},
	MRNUMBER = {753872},
	MRREVIEWER = {Peter\ J.\ Slater},
	DOI = {10.1016/0001-8708(84)90029-X},
	URL = {https://doi.org/10.1016/0001-8708(84)90029-X},
}

@article{CJK24,
	author = {Laura Casabella and Michael Joswig and Lars Kastner},
	title = {Subdivisions of Hypersimplices: With a View Toward Finite Metric Spaces},
	journal = {Experimental Mathematics},
	volume = {34},
	number = {4},
	pages = {808--823},
	year = {2025},
	publisher = {Taylor \& Francis},
	}

@book {Phylogenetics,
	AUTHOR = {Semple, Charles and Steel, Mike},
	TITLE = {Phylogenetics},
	SERIES = {Oxford Lecture Series in Mathematics and its Applications},
	VOLUME = {24},
	PUBLISHER = {Oxford University Press, Oxford},
	YEAR = {2003},
	PAGES = {xiv+239},
	ISBN = {0-19-850942-1},
	MRCLASS = {92D15 (05C05 05C90 92D40)},
	MRNUMBER = {2060009},
	MRREVIEWER = {Vincent\ L.\ Moulton},
}

@article {Lan13,
	AUTHOR = {Lang, Urs},
	TITLE = {Injective hulls of certain discrete metric spaces and groups},
	JOURNAL = {J. Topol. Anal.},
	FJOURNAL = {Journal of Topology and Analysis},
	VOLUME = {5},
	YEAR = {2013},
	NUMBER = {3},
	PAGES = {297--331},
	ISSN = {1793-5253,1793-7167},
	MRCLASS = {20F65 (30L05 52C10)},
	MRNUMBER = {3096307},
	DOI = {10.1142/S1793525313500118},
	URL = {https://doi.org/10.1142/S1793525313500118},
}

@article {SturmfelsYu,
    AUTHOR = {Sturmfels, Bernd and Yu, Josephine},
     TITLE = {Classification of six-point metrics},
   JOURNAL = {Electron. J. Combin.},
  FJOURNAL = {Electronic Journal of Combinatorics},
    VOLUME = {11},
      YEAR = {2004},
    NUMBER = {1},
     PAGES = {Research Paper 44, 16},
      ISSN = {1077-8926},
   MRCLASS = {51K05 (05C12 52B05 52B55)},
  MRNUMBER = {2097310},
MRREVIEWER = {Ezra\ N.\ Miller},
       DOI = {10.37236/1797},
       URL = {https://doi.org/10.37236/1797},
}

@article {dali-delucchi-michalek,
    AUTHOR = {D'Al{\`{i}}, A. and Delucchi, E. and Micha{\l}ek,
              M.},
     TITLE = {Many faces of symmetric edge polytopes},
   JOURNAL = {Electron. J. Combin.},
  FJOURNAL = {Electronic Journal of Combinatorics},
    VOLUME = {29},
      YEAR = {2022},
    NUMBER = {3},
     PAGES = {Paper No. 3.24, 42},
      ISSN = {1077-8926},
   MRCLASS = {52B20 (05A15 13P10 13P25 52B12)},
  MRNUMBER = {4458147},
MRREVIEWER = {Michael\ von Thaden},
       DOI = {10.37236/10387},
       URL = {https://doi.org/10.37236/10387},
}

@article{Francis_Steel,
    author = {Francis, Andrew R. and Steel, Mike},
    title = "{Which Phylogenetic Networks are Merely Trees with Additional Arcs?}",
    journal = {Systematic Biology},
    volume = {64},
    number = {5},
    pages = {768-777},
    year = {2015},
    issn = {1063-5157},
    doi = {10.1093/sysbio/syv037},
    url = {https://doi.org/10.1093/sysbio/syv037}
}

@Article{DHMantipodal,
 Author = {Dress, A. and Huber, K. T. and Moulton, V.},
 Title = {Antipodal metrics and split systems},
 FJournal = {European Journal of Combinatorics},
 Journal = {Eur. J. Comb.},
 ISSN = {0195-6698},
 Volume = {23},
 Number = {2},
 Pages = {187--200},
 Year = {2002},
 DOI = {10.1006/eujc.2001.0556},
 URL = {semanticscholar.org/paper/9d877044ee8c7755557c6dd0a9918a34e7b9699b},
 zbMATH = {1740608},
 Zbl = {1018.05024}
}

@Article{HuKoMo1,
 Author = {Huber, K. T. and Koolen, J. H. and Moulton, V.},
 Title = {The tight span of an antipodal metric space. {I}: combinatorial properties},
 FJournal = {Discrete Mathematics},
 Journal = {Discrete Math.},
 ISSN = {0012-365X},
 Volume = {303},
 Number = {1-3},
 Pages = {65--79},
 Year = {2005},
 DOI = {10.1016/j.disc.2004.12.018},
 zbMATH = {2242139},
 Zbl = {1082.54017}
}

@Article{Fischer_Francis,
 Author = {Fischer, Mareike and Francis, Andrew},
 Title = {The space of tree-based phylogenetic networks},
 FJournal = {Bulletin of Mathematical Biology},
 Journal = {Bull. Math. Biol.},
 ISSN = {0092-8240},
 Volume = {82},
 Number = {6},
 Pages = {17},
 Note = {Id/No 70},
 Year = {2020},
 DOI = {10.1007/s11538-020-00744-9},
 zbMATH = {7212726},
 Zbl = {1444.92070}
}

@Article{Francis_Semple_Steel,
 Author = {Francis, Andrew and Semple, Charles and Steel, Mike},
 Title = {New characterisations of tree-based networks and proximity measures},
 FJournal = {Advances in Applied Mathematics},
 Journal = {Adv. Appl. Math.},
 ISSN = {0196-8858},
 Volume = {93},
 Pages = {93--107},
 Year = {2018},
 DOI = {10.1016/j.aam.2017.08.003},
 URL = {hdl.handle.net/10092/15217},
 zbMATH = {6794705},
 Zbl = {1372.05034}
}

@article {Devadoss-petti,
	AUTHOR = {Devadoss, Satyan L. and Petti, Samantha},
	TITLE = {A space of phylogenetic networks},
	JOURNAL = {SIAM J. Appl. Algebra Geom.},
	FJOURNAL = {SIAM Journal on Applied Algebra and Geometry},
	VOLUME = {1},
	YEAR = {2017},
	NUMBER = {1},
	PAGES = {683--705},
	ISSN = {2470-6566},
	MRCLASS = {05E45 (14H10 52B11 92D15)},
	MRNUMBER = {3732939},
	MRREVIEWER = {Rahim\ Zaare-Nahandi},
	DOI = {10.1137/16M1103129},
	URL = {https://doi.org/10.1137/16M1103129},
}

@article {kalmanson,
	AUTHOR = {Kalmanson, Kenneth},
	TITLE = {Edgeconvex circuits and the traveling salesman problem},
	JOURNAL = {Canadian J. Math.},
	FJOURNAL = {Canadian Journal of Mathematics. Journal Canadien de
	Math\'{e}matiques},
	VOLUME = {27},
	YEAR = {1975},
	NUMBER = {5},
	PAGES = {1000--1010},
	ISSN = {0008-414X,1496-4279},
	MRCLASS = {05C35 (90B10 90C05)},
	MRNUMBER = {396329},
	MRREVIEWER = {L.\ V.\ Quintas},
	DOI = {10.4153/CJM-1975-104-6},
	URL = {https://doi.org/10.4153/CJM-1975-104-6},
}

@article {CF98,
	AUTHOR = {Chepoi, Victor and Fichet, Bernard},
	TITLE = {A note on circular decomposable metrics},
	JOURNAL = {Geom. Dedicata},
	FJOURNAL = {Geometriae Dedicata},
	VOLUME = {69},
	YEAR = {1998},
	NUMBER = {3},
	PAGES = {237--240},
	ISSN = {0046-5755,1572-9168},
	MRCLASS = {52A01 (52A30)},
	MRNUMBER = {1609385},
	DOI = {10.1023/A:1004907919611},
	URL = {https://doi.org/10.1023/A:1004907919611},
}

@article {BEK23,
	AUTHOR = {Brysiewicz, Taylor and Eble, Holger and K\"{u}hne, Lukas},
	TITLE = {Computing characteristic polynomials of hyperplane
	arrangements with symmetries},
	JOURNAL = {Discrete Comput. Geom.},
	FJOURNAL = {Discrete \& Computational Geometry. An International Journal
	of Mathematics and Computer Science},
	VOLUME = {70},
	YEAR = {2023},
	NUMBER = {4},
	PAGES = {1356--1377},
	ISSN = {0179-5376,1432-0444},
	MRCLASS = {52C35 (52B15)},
	MRNUMBER = {4670362},
	DOI = {10.1007/s00454-023-00557-2},
	URL = {https://doi.org/10.1007/s00454-023-00557-2},
}

@article {Bandelt-Dress-Trees,
	AUTHOR = {Bandelt, Hans-J\"{u}rgen and Dress, Andreas},
	TITLE = {Reconstructing the shape of a tree from observed dissimilarity
	data},
	JOURNAL = {Adv. in Appl. Math.},
	FJOURNAL = {Advances in Applied Mathematics},
	VOLUME = {7},
	YEAR = {1986},
	NUMBER = {3},
	PAGES = {309--343},
	ISSN = {0196-8858,1090-2074},
	MRCLASS = {05C05},
	MRNUMBER = {858908},
	MRREVIEWER = {J.\ M. S. Sim\~{o}es-Pereira},
	DOI = {10.1016/0196-8858(86)90038-2},
	URL = {https://doi.org/10.1016/0196-8858(86)90038-2},
}

@article {Buneman,
	AUTHOR = {Buneman, Peter},
	TITLE = {A note on the metric properties of trees},
	JOURNAL = {J. Combinatorial Theory Ser. B},
	FJOURNAL = {Journal of Combinatorial Theory. Series B},
	VOLUME = {17},
	YEAR = {1974},
	PAGES = {48--50},
	ISSN = {0095-8956},
	MRCLASS = {05C05},
	MRNUMBER = {363963},
	MRREVIEWER = {E.\ A.\ Nordhaus},
	DOI = {10.1016/0095-8956(74)90047-1},
	URL = {https://doi.org/10.1016/0095-8956(74)90047-1},
}

@article {HJ08,
	AUTHOR = {Herrmann, Sven and Joswig, Michael},
	TITLE = {Splitting polytopes},
	JOURNAL = {M\"{u}nster J. Math.},
	FJOURNAL = {M\"{u}nster Journal of Mathematics},
	VOLUME = {1},
	YEAR = {2008},
	PAGES = {109--141},
	ISSN = {1867-5778,1867-5786},
	MRCLASS = {52B12},
	MRNUMBER = {2502496},
}

@article {Tran,
    AUTHOR = {Tran, Ngoc Mai},
     TITLE = {Enumerating polytropes},
   JOURNAL = {J. Combin. Theory Ser. A},
  FJOURNAL = {Journal of Combinatorial Theory. Series A},
    VOLUME = {151},
      YEAR = {2017},
     PAGES = {1--22},
      ISSN = {0097-3165,1096-0899},
   MRCLASS = {14T05 (52B20)},
  MRNUMBER = {3663485},
MRREVIEWER = {Jaiung\ Jun},
       DOI = {10.1016/j.jcta.2017.03.011},
       URL = {https://doi.org/10.1016/j.jcta.2017.03.011},
}

@book {GKZ,
	AUTHOR = {Gelfand, I. M. and Kapranov, M. M. and Zelevinsky, A.
	V.},
	TITLE = {Discriminants, resultants, and multidimensional determinants},
	SERIES = {Mathematics: Theory \& Applications},
	PUBLISHER = {Birkh\"{a}user Boston, Inc., Boston, MA},
	YEAR = {1994},
	PAGES = {x+523},
	ISBN = {0-8176-3660-9},
	MRCLASS = {14N05 (13D25 14M25 15A69 33C70 52B20)},
	MRNUMBER = {1264417},
	MRREVIEWER = {I.\ Dolgachev},
	DOI = {10.1007/978-0-8176-4771-1},
	URL = {https://doi.org/10.1007/978-0-8176-4771-1},
}

@article {Vershik,
	AUTHOR = {Vershik, A. M.},
	TITLE = {Classification of finite metric spaces and combinatorics of
	convex polytopes},
	JOURNAL = {Arnold Math. J.},
	FJOURNAL = {Arnold Mathematical Journal},
	VOLUME = {1},
	YEAR = {2015},
	NUMBER = {1},
	PAGES = {75--81},
	ISSN = {2199-6792,2199-6806},
	MRCLASS = {52B12 (54E35)},
	MRNUMBER = {3331969},
	DOI = {10.1007/s40598-014-0005-z},
	URL = {https://doi.org/10.1007/s40598-014-0005-z},
}

@book {DLRS,
    AUTHOR = {De Loera, Jes\'us A. and Rambau, J\"org and Santos, Francisco},
     TITLE = {Triangulations},
    SERIES = {Algorithms and Computation in Mathematics},
    VOLUME = {25},
      NOTE = {Structures for algorithms and applications},
 PUBLISHER = {Springer-Verlag, Berlin},
      YEAR = {2010},
     PAGES = {xiv+535},
      ISBN = {978-3-642-12970-4},
   MRCLASS = {52B55 (05C10 52B05 57Q15 68U05)},
  MRNUMBER = {2743368},
       DOI = {10.1007/978-3-642-12971-1},
       URL = {https://doi.org/10.1007/978-3-642-12971-1},
}

@incollection {polymake_arrangements,
    AUTHOR = {Kastner, Lars and Panizzut, Marta},
     TITLE = {Hyperplane arrangements in {\tt polymake}},
 BOOKTITLE = {Mathematical software---{ICMS} 2020},
    SERIES = {Lecture Notes in Comput. Sci.},
    VOLUME = {12097},
     PAGES = {232--240},
 PUBLISHER = {Springer, Cham},
      YEAR = {2020},
      ISBN = {978-3-030-52200-1; 978-3-030-52199-8},
   MRCLASS = {},
  MRNUMBER = {4139491},
       DOI = {10.1007/978-3-030-52200-1\_23},
       URL = {https://doi.org/10.1007/978-3-030-52200-1_23},
}

@Book{ZieglerLoP,
 Author = {Ziegler, G{\"u}nter M.},
 Title = {Lectures on polytopes},
 FSeries = {Graduate Texts in Mathematics},
 Series = {Grad. Texts Math.},
 ISSN = {0072-5285},
 Volume = {152},
 ISBN = {3-540-94365-X},
 Year = {1995},
 Publisher = {Berlin: Springer-Verlag},
 Language = {English},
 DOI = {10.1007/978-1-4613-8431-1},
 zbMATH = {722614},
 Zbl = {0823.52002}
}

@incollection {polymake,
    AUTHOR = {Gawrilow, Ewgenij and Joswig, Michael},
     TITLE = {polymake: a framework for analyzing convex polytopes},
 BOOKTITLE = {Polytopes---combinatorics and computation ({O}berwolfach,
              1997)},
    SERIES = {DMV Sem.},
    VOLUME = {29},
     PAGES = {43--73},
 PUBLISHER = {Birkh\"auser, Basel},
      YEAR = {2000},
      ISBN = {3-7643-6351-7},
   MRCLASS = {52B55 (68U05)},
  MRNUMBER = {1785292},
}

@article {LP1,
	AUTHOR = {Lam, Thomas and Postnikov, Alexander},
	TITLE = {Alcoved polytopes. {I}},
	JOURNAL = {Discrete Comput. Geom.},
	FJOURNAL = {Discrete \& Computational Geometry. An International Journal
	of Mathematics and Computer Science},
	VOLUME = {38},
	YEAR = {2007},
	NUMBER = {3},
	PAGES = {453--478},
	ISSN = {0179-5376,1432-0444},
	MRCLASS = {52B12 (52B40)},
	MRNUMBER = {2352704},
	MRREVIEWER = {Jes\'{u}s\ A.\ De Loera},
	DOI = {10.1007/s00454-006-1294-3},
	URL = {https://doi.org/10.1007/s00454-006-1294-3},
}

@article {LP2,
    AUTHOR = {Lam, Thomas and Postnikov, Alexander},
     TITLE = {Polypositroids},
   JOURNAL = {Forum Math. Sigma},
  FJOURNAL = {Forum of Mathematics. Sigma},
    VOLUME = {12},
      YEAR = {2024},
     PAGES = {Paper No. e42, 67},
      ISSN = {2050-5094},
   MRCLASS = {52B40 (05B35 20F55)},
  MRNUMBER = {4718184},
       DOI = {10.1017/fms.2024.11},
       URL = {https://doi.org/10.1017/fms.2024.11},
}

@article {Zas75,
    AUTHOR = {Zaslavsky, Thomas},
     TITLE = {Facing up to arrangements: face-count formulas for partitions
              of space by hyperplanes},
   JOURNAL = {Mem. Amer. Math. Soc.},
  FJOURNAL = {Memoirs of the American Mathematical Society},
    VOLUME = {1},
      YEAR = {1975},
     PAGES = {vii+102},
      ISSN = {0065-9266,1947-6221},
   MRCLASS = {05A15},
  MRNUMBER = {357135},
MRREVIEWER = {E.\ Jucovi\v c},
       DOI = {10.1090/memo/0154},
       URL = {https://doi.org/10.1090/memo/0154},
}

@book{OSCAR-book,
	editor = {Decker, Wolfram and Eder, Christian and Fieker, Claus and Horn, Max and Joswig, Michael},
	title = {The {C}omputer {A}lgebra {S}ystem {OSCAR}: {A}lgorithms and {E}xamples},
	year = {2025},
	publisher = {Springer},
	series = {Algorithms and {C}omputation in {M}athematics},
	volume = {32},
	edition = {1},
	url = {https://link.springer.com/book/9783031621260},
	issn = {1431-1550},
}

@article{Kal88,
    AUTHOR = {Kalai, Gil},
     TITLE = {A new basis of polytopes},
   JOURNAL = {J. Combin. Theory Ser. A},
  FJOURNAL = {Journal of Combinatorial Theory. Series A},
    VOLUME = {49},
      YEAR = {1988},
    NUMBER = {2},
     PAGES = {191--209},
      ISSN = {0097-3165,1096-0899},
   MRCLASS = {52A25},
  MRNUMBER = {964383},
MRREVIEWER = {P.\ McMullen},
       DOI = {10.1016/0097-3165(88)90051-9},
       URL = {https://doi.org/10.1016/0097-3165(88)90051-9},
}

@misc{Ram23,
    author={Rambau, J{"o}rg},
    title={Symmetric lexicographic subset reverse search for the enumeration of circuits, cocircuits, and triangulations up to symmetry},
    howpublished = "\url{https://www.wm.uni-bayreuth.de/de/team/rambau_joerg/TOPCOM/SymLexSubsetRS-2.pdf}",
    year={2023}
}

@article{McMullen,
	author = {Mc{M}ullen, P. },
	journal = {Geometriae Dedicata},
	number = {1},
	pages = {83--99},
	title = {Representations of polytopes and polyhedral sets},
	volume = {2},
	year = {1973}}

@misc{EKM25,
	title={When alcoved polytopes add}, 
	author={Nick Early and Lukas Kühne and Leonid Monin},
	year={2025},
	eprint={2501.17249},
	archivePrefix={arXiv},
	primaryClass={math.CO},
	url={https://arxiv.org/abs/2501.17249}, 
}

@misc{BSS26,
	title={Type Fans of Alcoved Polytopes, Nestohedra, and Simple Games}, 
	author={Aenne Benjes and Raman Sanyal and Benjamin Schröter},
	year={2026},
	note={in preparation}, 
}

@article{SS18,
	title = {Lipschitz polytopes of posets and permutation statistics},
	journal = {Journal of Combinatorial Theory, Series A},
	volume = {158},
	pages = {605-620},
	year = {2018},
	issn = {0097-3165},
	doi = {https://doi.org/10.1016/j.jcta.2018.04.006},
	url = {https://www.sciencedirect.com/science/article/pii/S0097316518300591},
	author = {Raman Sanyal and Christian Stump},
}

\appendix

\section{Polyhedral complexes associated to pseudometrics}\label{ats}

Our goal in this appendix is to describe the difference between the tight span $E(\rho)$ and the polyhedral complex $E^\sharp(\rho)$  that we call  ``pruned tight span" in \Cref{sec:tight_spans} (see \Cref{pts,ctts}) in an explicit and self-contained way. The distinction is in order because in general the complexes $E^\sharp (\rho)$ and $E(\rho)$  do not coincide, although they are sometimes called by the same name in the literature.

Recall the polyhedron
$$P(\metric):=\{x\in \R^{n} \mid x_i+x_j\geq \metric_{ij}\textrm{ for }i,j\in [n]\}$$
and let $E(\metric)$ be the complex of bounded faces of $P(\metric)$. As we have seen in \Cref{sec:tight_spans}, $E(\metric)$ is the set underlying a realization of the injective hull of $\metric$ (also named ``tight span" in \cite{Dress}).

\begin{definition}\label{AppDef}
Given $\metric\in \R^{{n\choose 2}}$, define
$$P^{\sharp}(\metric):=\{x\in \R^{n} \mid x_i+x_j\geq \metric_{ij}\textrm{ for }i,j\in [n],\, i\neq j\}$$ 
and let $E^\sharp (\metric)$ be the complex of bounded faces of $P^\sharp (\metric)$. 
\end{definition}

In order to study the relationship between $E^\sharp (\rho)$ and $E(\rho)$ we will start by defining, for every $k\in [n]$, two points $\hk{k}{},\vk{k}{}\in \R^n$ as follows:

\begin{eqnarray}
& \hk{k}{i} & :=  \metric_{ki}\quad\textrm{ for all }i;\\
& \vk{k}{i} & := \begin{cases}
\xi_k & i=k\\
\metric_{ik}-\xi_k &\textrm{ otherwise}
\end{cases},\quad\textrm{ where }\\
& \xi_k & :=   \frac{1}{2}\min\{\metric_{ik}+\metric_{kj} - \metric_{ij}\mid k\not \in \{i,j\}\}.
\end{eqnarray}

\begin{remark} Note that in the terminology of \cite{BandeltDress} the number $\xi_k$ is the {\em isolation index} of $k$. For every pseudometric $\metric$ and every $k\in[n]$ we have $\xi_k\geq 0$; moreover, $\xi_k = 0$ if and only if $\vk{k}{}=\hk{k}{}$.
\end{remark}

We will prove the following statement.
\begin{theorem}\label{thm1}\label{AppAntennas}
	The points $\vk{k}{}$ for $k\in [n]$ are vertices of the complex $E^\sharp(\metric)$, and the latter is a polyhedral subcomplex of $E(\metric)$. More precisely, $E(\metric)$ is obtained by adding to $E^\sharp(\metric)$ the vertex $\hk{k}{}$ and the edge $e^{(k)}:=\conv\{\hk{k}{},\vk{k}{}\}$ for every $k$ with $\xi_k>0$.
\end{theorem}

Before proving the theorem we state a corollary about the metric fan $\MF{n}$. Recall the natural labelings of tight spans in \Cref{ctts}.

\begin{corollary}\label[corollary]{EMF}
The labeled combinatorial type of $E(\metric)$ is determined by the face of $\MF{n}$ containing $\metric$.
\end{corollary}
\begin{proof} We know by \Cref{hjrem} that the labeled combinatorial type of $E^\sharp(\metric)$ is determined by the stratum of $\pstratsharp{n}$ containing $\metric$, and the latter is a union of cones of $\MF{n}$. Since by \Cref{thm1} the complex $E(\metric)$ differs from $E^\sharp(\metric)$ by attaching one dangling edge at $\vk{k}{}$ for every $k$ with $\xi_k>0$, it is enough to show that, for every $k\in [n]$, the set of all pseudometrics with $\xi_k=0$ is a union of faces of the pseudometric cone $\mconeclosure{n}$. Indeed, let $K_{ij}^k$ be the hyperplane in $\R^{{n\choose 2}}$ with equation $x_{ik} + x_{kj} - x_{ij}=0$. This hyperplane clearly supports $\mconeclosure{n}$, hence $F_{ij}^k:= \mconeclosure{n} \cap K_{ij}^k$ is a face of $\mconeclosure{n}$. The set of all pseudometrics with $\xi_k=0$ is the union $\bigcup_{i,j\neq k} F_{ij}^k$.
\end{proof}

We need to prepare the proof of \Cref{thm1} by a series of lemmata.

\begin{lemma}\label[lemma]{subcomplex} For every pseudometric $\rho$ we have
\begin{enumerate}
\item[(i)] $E^\sharp (\metric)\subseteq \R^n_{\geq 0}$; \label{positivo}
\item[(ii)] The complex $E^{\sharp}(\metric)$ is a subcomplex of $E(\rho)$.
\end{enumerate}
\end{lemma}
\begin{proof}
For the first claim, we show that $E^\sharp (\metric)$ has no vertex outside the positive orthant. Consider $x\in E^\sharp (\metric)$ and suppose that $x_1<0$. Then for every $i,j\neq 1$ we have $x_i+x_j > x_i+x_1+x_1+x_j \geq \metric_{1i} + \metric_{1j}\geq \metric_{ij}$, hence $x_i+x_j>\metric_{ij}$. This means that the defining equations of $P^\sharp(\metric)$ that hold with equality at $x$ are at most those of the type $x_1+x_i\geq \metric_{ij}$. Since the intersection of the hyperplanes $x_1+x_i = \metric_{ij}$ has codimension at most $n-2$, the face of $P^\sharp(\metric)$ containing $x$ cannot be a vertex. 

For the second claim let $F^\sharp$ be a face of $E^\sharp(\metric)$, i.e., a bounded face of $P^\sharp(\metric)$.  Surely $F:=F^\sharp \cap \R^n_{\geq 0}$ is a bounded face of $P(\metric)$, and by claim (i) we have $F=F^\sharp$. Thus $F^\sharp$ is a bounded face of $P(\metric)$ as well. This show that every face of $E^\sharp(\metric)$ is a face of $E(\metric)$. The claim follows.
\end{proof}

\begin{lemma}\label[lemma]{ray}
Let $\metric$ be a pseudometric. Then, for every $k\in [n] $ the following holds.
\begin{itemize}
\item[(i)]
For all $k\in [n]$, the point $\vk{k}{}$ is a vertex of $E^\sharp(\metric)$;
\item[(ii)] The complex $E^\sharp(\metric)$ contains the ray $R_k:= \vk{k}{} + \R_{\geq 0}s^{(k)}$, where
$$s^{(k)} = e_1+\ldots e_{k-1} - e_k + e_{k+1} +\ldots + e_{n}.$$
\end{itemize}
\end{lemma}

\begin{proof}
It is straightforward to check that all inequalities that define $P^\sharp$ are satisfied by $\vk{k}{}$, thus  $\vk{k}{}\in P^\sharp$. We show that $\vk{k}{}$ is indeed a vertex, hence a bounded face, of $P^\sharp$.  Let $i_0$ and $j_0$ be indices for which the minimum defining $\xi_k$ is attained, i.e., $\metric_{i_0k}+ \metric_{kj_0}-\metric_{i_0j_0}$. Now $\vk{k}{}$ lies on the hyperplane $H_0$ with equation $x_{i_0}+x_{j_0}=\metric_{i_0j_0}$. Moreover, $\vk{k}{}$ is contained in the hyperplanes $H_{ik}:=\{x\in \R^n \mid x_i+x_k=\metric_{ik}\}$ for all $i\neq k$. The normals of the planes $H_0$ and $H_{ik}$ are a linearly independent set  of rank $n$ and thus their intersection, since it is nonempty, has dimension $0$.
\end{proof}

\begin{lemma}\label[lemma]{partition} \strut
\begin{itemize}
\item[(i)]
$E(\metric)\cap \coh{k} =\{ \hk{k}{i} \}$ and in particular every $\hk{k}{}$ is a vertex of $E$. \label{hvert}
\item[(ii)] If $\xi_k>0$ then $E(\metric)\setminus \{\vk{k}{}\}$ is disconnected, and the induced partition of the vertices of $E(\metric)$ is a bipartition with $\hk{k}{}$ in a singleton block. \label{mega}
\end{itemize}
\end{lemma}

\begin{proof}
The first statement is \cite[Theorem 3.(ii)]{Dress}.

For the second statement notice first that for $i,j\neq k$ we have \\
$\vk{k}{i}+\vk{k}{j} = \metric_{ki}+\metric_{kj}-2\xi_k \geq \metric_{ki}+\metric_{kj} - (\metric_{ki}+\metric_{kj} - \metric_{ij} ) = \metric_{ij}$.

Moreover, for all $i\neq k$ we have
\begin{equation*}\label{equg}
\vk{k}{i}+\vk{k}{k}=\metric_{ki} - \xi_k + \xi_k = \metric_{ki}
\end{equation*}

Thus $\vk{k}{}$ satisfies the hypotheses of \cite[Theorem 6]{Dress} with respect to the bipartition $[n]=\{k\}\uplus ([n]\setminus \{k\})$. This theorem says that $E(\metric)\setminus \vk{k}{}$ is the union of two open (in $E$) subsets:
$$
\underbrace{\{x\in E(\metric) \mid x_{k}< \vk{k}{k}\}}_{=:A} \uplus
\underbrace{\{
x\in E(\metric) \mid x_{i}< \vk{k}{i} \textrm{ for some }i\in [n]\setminus\{k\}
\}}_{=:B}
$$
If $\xi_k>0$ neither subset is empty, since in that case $\hk{k}{}\in A$ (because $\hk{k}{k}=0<\xi_k=\vk{k}{k}$) and $\hk{j}{}\in B$ for all $j\neq k$ (indeed $\vk{k}{}\in E(\metric)$ means $\vk{k}{j}=\metric_{jk}-\xi_k\geq 0$ for $j\neq k$. Thus $\hk{j}{k}=\metric_{jk}\geq \xi_k=\vk{k}{k}$, hence $\hk{j}{}\not\in A$, i.e., $\hk{k}{}\in B$).

It remains to show that all all-positive vertices of $E(\metric)$ are in $B$. If $n\leq 2$ there is nothing to show. Thus let $n\geq 3$. An all-positive vertex $y$ of $E(\metric)$  is, by \Cref{subcomplex}, a vertex of $E^\sharp(\metric)$. Since the intersections of all hyperplanes $x_i+x_k=\metric_{ik}$ with $i\neq k$ has codimension at most $n-2$, there is at least a pair $i\neq j$ with $i,j\neq k$ such that $y_i+y_j=\metric_{ij}$. Now we can compute
$$
2y_k + \metric_{ij} = 2y_k + y_i+y_j \geq \metric_{ik} + \metric_{jk}.
$$
Thus 
$$y_k\geq \frac{1}{2}(\metric_{ik} + \metric_{jk} - \metric_{ij} )\geq \xi_k$$
hence $y\not\in A$ and so $y\in B$.
\end{proof}

\begin{lemma}\label[lemma]{edge}
If $\xi_k>0$, the line segment $e^{(k)}$ joining $\hk{k}{}$ $\vk{k}{}$ is a cell of $E(\metric)$; moreover, it is the only cell of $E(\metric)$ containing $\hk{k}{}$.
\end{lemma}
\begin{proof}
	Notice that, with the notation of \Cref{ray}.(ii), we have $\hk{k}{} = \vk{k}{} + \xi_ks^{(k)}$ and, for all $i\neq k$,  $\vk{k}{i} + \lambda s^{(k)}_i > 0$ for $0<\lambda < \xi_k $. Thus since $R_k$ is a face of $P^\sharp(\metric)$, $e^{(k)}=R_k\cap \R^n_{\geq 0}$ is a face of $P(\metric)$, which is bounded by construction. This proves the first claim.
	
	For the second claim assume by way of contradiction that $\hk{k}{}$ is contained in another cell. Then, this cell would have to have at least another vertex $w$ (different from $\vk{k}{}$), so even after removing $\vk{k}{}$ there would be a path in $E(\metric)$ connecting $\hk{k}{}$ with $w$. But this contradicts the disconnectedness claim of \Cref{partition}.(ii).
\end{proof}

\begin{proof}[Proof of \Cref{thm1}]
The first claim is \Cref{ray}.(i), the second is  \Cref{subcomplex}.(ii), the third follows by combining \Cref{partition}.(i) and \Cref{edge}.
\end{proof}

\section{Five point metrics}\label{sec:five_point_metrics_comp}

\subsection{Table of generic metrics on 5 points}\label{sec:gen_5_metrics}
We give a description of the $12$ combinatorial types of generic KRW polytopes. The symmetric group $S_5$ acts on the chambers by coordinate permutation. The number of chambers yielding the same metric as a given chamber is 120 divided by the order of the stabilizer of the chamber under this action. 

In particular, two chambers yield combinatorially equivalent polytopes if and only if there is $\sigma \in S_5$ mapping one chamber to the other. The corresponding KRW polytopes encoded as Polymake objects can be found in Zenodo \cite{zenodo}.

\renewcommand\arraystretch{0.5}
\begin{longtable}{@{}llllllllllll @{}}
\toprule

\textsc{Nr.}& $\metric_{12} $& $ \metric_{13} $&$ \metric_{14}$&$ \metric_{15} $&$ \metric_{23} $&$ \metric_{24} $&$ \metric_{25} $&$ \metric_{34} $&$ \metric_{35} $&$ \metric_{45}$  & \textsc{stabilizer} \\

\midrule
1 &   8 &18 & 15 &11& 14& 15& 15& 7& 11& 14   & $\langle (13)(24)\rangle$  \\
2 &  8 & 20 & 15 & 13 & 16 & 19 &13 &9 &11& 16  & \groupnamea \\ 
3 &  8& 18& 19& 11& 18& 15& 15& 7& 11& 14  &  $\langle (13)(24)\rangle$  \\
4 &  10& 22& 17& 11& 16& 15& 17& 9 &15& 20  & \groupnamea\\
5 &  6 &12 &8& 6& 9& 8& 8& 6& 8& 12   & \groupnamea\\
6 &  6& 11& 7& 6& 8& 9& 6& 6& 7& 11  &  $\langle (15)(34)\rangle$  \\ 
7 &  10& 22& 15& 13& 16& 21& 11& 11& 13& 20 & \groupnamea\\ 
8 &    12& 24& 25& 11& 24& 17& 19& 11& 17& 24  &  $\langle (25)(34)\rangle$  \\ 
9 &    7& 16& 11& 10& 11& 16& 10& 7& 10& 10 &  $\langle (13)(24), (12)(34)\rangle$ \\ 
10 &   14& 30& 21& 17& 20& 15& 27& 13& 17& 26  & $ \langle (13)(24)\rangle $ \\ 
11 &   7& 13& 8& 7& 9& 7& 10& 7& 8& 13 &  $ \langle (14)(35)\rangle $\\ 
12 &  11& 20& 20& 11& 20& 11& 20& 11& 11& 20   & $D_{10}$ in $S_5$. \\ 
\bottomrule
\caption{Types of generic metrics for $n = 5$.}
\label{tab:metrics5_gen}
\end{longtable}

\subsection{Table of strict metrics on $5$ points}\label{sec:strict}
In the following table, each row represents one of the $65$ combinatorial types of KRW polytopes of strict metric spaces on $5$ points. For each combinatorial type, we list a sample metric $\metric$, the $f$-vector of the polytope $\KRW{\metric}$ and the stabilizer subgroup of the cone of $\GPF{5}$ containing $\metric$ (see \Cref{n_gleich_5}).
The corresponding KRW polytopes encoded as Polymake objects can be found in Zenodo \cite{zenodo}. 

\renewcommand\arraystretch{0.5}

\begin{longtable}{@{}llllllllllll p{1in}@{}}
	\toprule
	\textsc{$f$-vector} &\textsc{Nr.} & $\metric_{12} $& $ \metric_{13} $&$ \metric_{14}$&$ \metric_{15} $&$ \metric_{23} $&$ \metric_{24} $&$ \metric_{25} $&$ \metric_{34} $&$ \metric_{35} $&$ \metric_{45}$  & \textsc{stabilizer}\\
	
	\midrule

20 60 70 30  & 1.1 & 2& 2& 2& 2& 2& 2& 2& 2& 2& 2& $S_5$\\
&&&&&&&&&&& \\ 
20 72 94 42  & 2.1 & 2& 3& 3& 2& 3& 3& 2& 2& 3& 3& $D_{12}$\\
& 2.2 & 3& 4& 3& 3& 3& 4& 4& 3& 3& 4& $D_{12}$\\
&&&&&&&&&&& \\ 
20 76 102 46  & 3.1 & 4& 6& 5& 5& 6& 5& 5& 5& 5& 6& $C_2\times C_2$\\
&&&&&&&&&&& \\ 
20 80 112 52  & 4.1 & 4& 8& 7& 5& 8& 7& 5& 5& 7& 8& $C_2$\\
& 4.2 & 3& 5& 4& 3& 4& 5& 4& 3& 4& 5& $C_2$\\
& 4.3 & 6& 10& 7& 7& 8& 9& 9& 7& 7& 10& $C_2$\\
&&&&&&&&&&& \\ 
20 80 114 54  & 5.1 & 2& 4& 4& 3& 4& 4& 3& 2& 3& 3& $D_8$\\
& 5.2 & 4& 8& 7& 7& 8& 7& 7& 5& 5& 6& $C_2\times C_2$\\
& 5.3 & 3& 4& 3& 3& 4& 3& 3& 3& 3& 4& $C_2\times C_2$\\
& 5.4 & 10& 14& 9& 9& 10& 11& 11& 9& 9& 14& $D_8$\\
&&&&&&&&&&& \\ 
20 82 116 54  & 6.1 & 3& 5& 4& 4& 4& 5& 3& 3& 3& 4& $C_2$\\
&&&&&&&&&&& \\ 
20 84 122 58  & 7.1 & 6& 10& 9& 7& 8& 7& 9& 5& 7& 8& $C_2$\\
& 7.2 & 3& 6& 5& 4& 5& 6& 3& 3& 4& 5& \groupnamea\\
& 7.3 & 4& 7& 5& 5& 5& 7& 5& 4& 4& 6& \groupnamea\\
& 7.4 & 6& 12& 9& 9& 10& 11& 7& 7& 7& 10& \groupnamea\\
& 7.5 & 6& 12& 9& 7& 10& 7& 9& 7& 9& 12& \groupnamea\\
& 7.6 & 6& 10& 7& 7& 8& 9& 5& 7& 7& 10& $C_2$\\
& 7.7 & 6& 12& 9& 7& 10& 11& 9& 7& 9& 12& \groupnamea\\
&&&&&&&&&&& \\ 
20 84 124 60  & 8.1 & 4& 10& 9& 7& 10& 9& 7& 5& 7& 8& $C_2$\\
& 8.2 & 6& 12& 9& 9& 10& 11& 11& 7& 7& 10& $C_2$\\
& 8.3 & 3& 5& 4& 3& 5& 4& 3& 3& 4& 5& $C_2$\\
& 8.4 & 5& 8& 5& 5& 6& 6& 6& 5& 5& 8& $C_2$\\
&&&&&&&&&&& \\ 
20 86 128 62  & 9.1 & 3& 6& 5& 4& 5& 6& 5& 3& 4& 5& $C_2$\\
& 9.2 & 6& 12& 11& 7& 10& 9& 9& 5& 9& 10& \groupnamea\\
& 9.3 & 8& 14& 11& 9& 10& 11& 13& 7& 9& 12& $C_2$\\
& 9.4 & 3& 7& 6& 5& 6& 7& 4& 3& 4& 5& $C_2$\\
& 9.5 & 4& 8& 6& 5& 6& 8& 5& 4& 5& 7& \groupnamea\\
& 9.6 & 2& 4& 3& 3& 3& 4& 3& 2& 2& 3& \groupnamea\\
& 9.7 & 6& 14& 11& 9& 12& 13& 7& 7& 9& 12& \groupnamea\\
& 9.8 & 6& 14& 11& 11& 12& 13& 9& 7& 7& 10& \groupnamea\\
& 9.9 & 8& 16& 11& 11& 12& 15& 11& 9& 9& 14& \groupnamea\\
& 9.10 & 8& 14& 15& 9& 14& 11& 13& 7& 9& 12& $C_2$\\
& 9.11 & 8& 16& 13& 9& 12& 9& 13& 7& 11& 14& \groupnamea\\
& 9.12 & 6& 14& 11& 7& 12& 9& 9& 7& 11& 14& \groupnamea\\
& 9.13 & 8& 16& 11& 9& 12& 11& 13& 9& 11& 16& \groupnamea\\
& 9.14 & 10& 16& 11& 9& 12& 13& 11& 9& 11& 16& $C_2$\\
& 9.15 & 4& 7& 5& 4& 6& 4& 5& 4& 5& 7& \groupnamea\\
& 9.16 & 12& 20& 13& 11& 14& 13& 15& 11& 13& 20& $C_2$\\
& 9.17 & 6& 12& 9& 7& 10& 11& 5& 7& 9& 12& $C_2$\\
& 9.18 & 8& 14& 9& 9& 10& 13& 9& 9& 9& 14& \groupnamea\\
& 9.19 & 8& 16& 11& 11& 12& 15& 7& 9& 9& 14& \groupnamea\\
& 9.20 & 8& 16& 11& 9& 12& 7& 13& 9& 11& 16& $C_2$\\
&&&&&&&&&&& \\ 
20 88 134 66  & 10.1 & 6& 14& 11& 9& 12& 13& 11& 7& 9& 12& \groupnamea\\
& 10.2 & 6& 14& 13& 9& 12& 11& 11& 5& 9& 10& $C_2$\\
& 10.3 & 8& 16& 13& 9& 12& 13& 13& 7& 11& 14& \groupnamea\\
& 10.4 & 4& 9& 7& 6& 7& 9& 6& 4& 5& 7& \groupnamea\\
& 10.5 & 6& 16& 13& 11& 14& 15& 9& 7& 9& 12& \groupnamea\\
& 10.6 & 8& 18& 13& 11& 14& 17& 11& 9& 11& 16& \groupnamea\\
& 10.7 & 8& 18& 13& 13& 14& 17& 13& 9& 9& 14& \groupnamea\\
& 10.8 & 8& 16& 17& 9& 16& 13& 13& 7& 11& 14& \groupnamea\\
& 10.9 & 8& 18& 15& 9& 14& 11& 13& 7& 13& 16& \groupnamea\\
& 10.10 & 10& 20& 15& 11& 14& 13& 17& 9& 13& 18& \groupnamea\\
& 10.11 & 8& 18& 13& 9& 14& 13& 13& 9& 13& 18& \groupnamea\\
& 10.12 & 5& 9& 6& 5& 7& 7& 6& 5& 6& 9& \groupnamea\\
& 10.13 & 4& 8& 6& 4& 7& 5& 5& 4& 6& 8& \groupnamea\\
& 10.14 & 6& 11& 7& 6& 8& 7& 8& 6& 7& 11& \groupnamea\\
& 10.15 & 12& 22& 15& 11& 16& 15& 15& 11& 15& 22& \groupnamea\\
& 10.16 & 8& 16& 11& 9& 12& 15& 9& 9& 11& 16& \groupnamea\\
& 10.17 & 6& 10& 6& 6& 7& 8& 6& 6& 6& 10& $C_2$\\
& 10.18 & 8& 18& 13& 11& 14& 17& 7& 9& 11& 16& \groupnamea\\
& 10.19 & 10& 20& 13& 13& 14& 19& 11& 11& 11& 18& \groupnamea\\
& 10.20 & 12& 22& 23& 11& 22& 15& 19& 11& 15& 22& $C_2$\\
& 10.21 & 12& 26& 19& 15& 18& 11& 23& 11& 15& 22& $C_2$\\
& 10.22 & 10& 20& 13& 11& 14& 11& 17& 11& 13& 20& \groupnamea\\

\bottomrule
\caption{Types of strict metrics for $n = 5$.}
\label{tab:metrics5_strict}
\end{longtable} 

\end{document}